\documentclass{amsart}
\usepackage{amscd,amssymb,amsmath,amsthm}
\usepackage[dvips]{graphicx}
\usepackage[all]{xy}
\theoremstyle{plain}
\newtheorem{definition}{Definition}
\newtheorem{proposition}{Proposition}
\newtheorem{theorem}[proposition]{Theorem}

\newtheorem{lemma}[proposition]{Lemma}
\newtheorem{remark}[proposition]{Remark}
\newtheorem*{proposition*}{Proposition}
\newtheorem*{theorem*}{Theorem}
\newtheorem*{corollary*}{Corollary}
\newtheorem*{lemma*}{Lemma}
\newtheorem*{remark*}{Remark}
\newtheorem*{example*}{Example}

\newcommand{\Z}{\mathbb{Z}}
\newcommand{\Q}{\mathbb{Q}}
\newcommand{\R}{\mathbb{R}}
\newcommand{\C}{\mathbb{C}}

\begin{document}

\title{Point-Splitting Perturbative Chern-Simons II}
 
\author{Vito Iacovino}


\email{vito.iacovino@gmail.com}

\date{version: \today}


\begin{abstract}
We extend the formalism of Point-Splitting-Perturbative-Chern-Simons to include Wilson Loops. 
\end{abstract}

\maketitle

\section{Introduction}

Let $M$ be a compact oriented manifold. In a preview paper \cite{PSPCS1} we introduced the mathematical theory of pure (i.e., without Wilson loops) Point-Splitting Perturbative Chern-Simons ($PSPCS$) on $M$.
In this paper we extend the formalism of \cite{PSPCS1} to include Wilson Loops. 

In this paper we use the same notion of $PS$-cycle used in \cite{PSPCS1}, which depends on the choice of an annulus data,  The difference with \cite{PSPCS1} arise in the definition of nice-$PS$-cycle where we drop the assumption made in \cite{PSPCS1} roughly stating that the loop associated to each vertex shrink to a point. In order to make this condition precise we introduce the notion of limit-support of a nice-$PS$-cycle (whereas in \cite{PSPCS1} we limited ourself to consider only trivial limit-support).  
Point-Splitting-Homology is defined as the space of equivalence classes of nice-$PS$-cycles, where we quotient by isotopies that do not deform the limit-support (that is, we consider only small isotopies). We denote it as $PSH(M| U)$, where $U$ i a smooth Euler Structure $U$ corresponding to the annulus data (see \cite{PSPCS1}).
To each oriented embedded link $\mathbf{w}$ it is associated a subspace $PSH(M| \mathbf{w}, U)$ of $PSH(M| U)$. The subspaces of $PSH(M| U)$ generated by $PSH(M| \mathbf{w}, U)$ for different $\mathbf{w}$ is dense in $PSH(M| U)$.

$PSH(M| U)$ comes with a natural notion of isotopy, which in particular deform the limit-support. 
To an isotopy of links correspond an isotopy of the associated subspaces $PSH(M| \mathbf{w}, U)$. However the notion of isotopy in $PSH$ is much more general of the isotopy of links and its formulation does not refer to links.
This fact is of critical importance in order to understand Skein relations:  classically the crossing of the link or the twisting of the frame are singular processes, however in $PSH$ each process is translated in a \emph{smooth} process. 
The Skein relations are translated in the existence of small isotopies in $PSH$ between certain $PSH$-classes. 
The $PS$-Homology provides the natural ambient space where the Skein relations live. 

As in \cite{PSPCS1} to each $PS$-cycle $Z$ it is associated a partition function $\mathfrak{P}(Z)$ which satisfies the Quantum Master Equation. 
The construction of  \cite{PSPCS1}  extends directly in the context of this paper. In the definition of  \cite{PSPCS1} roughly each vertex of the graph is replaced by a small loop, here we allow also not small loops, the formula used in \cite{PSPCS1} applies in the same way.  We stress that the picture in $PSPCS$ is quite different from the one of the standard approach where the vertices of the graph and the components of the Wilson loop are treated on a different footing. 

Next we extend the definition of Coherent $PSH$-class of \cite{PSPCS1}. We fix an oriented embedded link $\mathbf{w}$ and a frame $\mathbf{fr}$ of $M$ compatible with $\mathbf{w}$ (as in \cite{PSPCS1} we require that the Euler Structure $U$ is compatible with $\mathbf{fr}$). We extend the notion of Flat-Approaching-$PS$-cycle in this context. With this definition the argument of \cite{PSPCS1} applies and yields an element $[Z^{\mathbf{w}, \mathbf{fr}}] \in PSH(M| \mathbf{w}, U)$. 
Considering only frames of $M$ compatible with $\mathbf{w}$ may be restrictive since not each frame of $\mathbf{w}$ may be extended to a frame of $M$.
In Chern-Simons Theory it is standard to consider frames of $\mathbf{w}$ introducing  \emph{normalized} expectation values of the Wilson Loop defined quotienting by the pure (i.e. without Wilson Loops) Chern-Simons Partition function. The anomaly of pure Chern-Simons is automatically cancelled, whereas the anomaly associated to the link is cancelled adding correction terms defined using the frame of $\mathbf{w}$. 
In order to produce objects depending on framed links in $PSPCS$ 
we introduce a suitable quotient of $PSH$ which we call normalized-$PSH$ and we denote it by $PSH^{\intercal}(M| U)$. We show that  $PSH^{\intercal}(M| \mathbf{w}, U) $  is a rank one free modulo over the ring of formal power series on two variables $\Q[[\hbar,a]]$. To a frame $\nu$ of $\mathbf{w}$ and a compatible Euler Structure $U$  it is  associated canonically a normalized-$PSH$-class  $[Z^{(\mathbf{w}, \nu,U)}] \in PSH^{\intercal}(M|  U)$. This class is a generator of $PSH^{\intercal}(M| \mathbf{w}, U) $ as $\Q[[\hbar,a]]$-module.

When the Euler Structure $U$ is fixed, the compatibility condition between $\nu$ and $U$ used to define the class $[Z^{(\mathbf{w}, \nu,U)}]$ put a quite strong condition on the frame. We remove this condition introducing a new geometric data which we call Euler Alignment.
Given a frame $\nu$,  an Euler Alignment $\kappa$ determines an essential unique deformation of $U$ to an Euler Structure compatible with $\nu$.
The result mentioned above yields a canonical normalized $PSH$-class $[Z^{(\mathbf{w}, \nu, \kappa)}] \in PSH^{\intercal}(M|  U)$ associated to the triple $(\mathbf{w}, \nu, \kappa)$.




A well known proposal in the physical literature states that the normalized expectation values of Wilson Loops satisfies the Skein relations  (see \cite{jones}). This proposal motivated the mathematical notion of Skein Module.
The standard definition of Skein involves oriented embedded links equipped with an isotopy class of frames, and comes with a notion of isotopy. Here we consider the Skein before quotienting by isotopies. In the context of $PSPCS$ this definition is modified in order to include the extra geometric data provided by the isotopy classes of Euler Alignments.
We show that  the normalized $PSH$-classes $[Z^{(\mathbf{w}, \nu, \kappa)}] $ satisfy an extended version of the skein relations.
We use this result to show that $PSH^{\intercal}(M|  U)$ is isomorphic (up to small isotopies) to the extended Skein. This correspondence is compatible with isotopies. Considering a suitable completion of the Skein (which also includes singular links) we obtain a strict isomorphism.

\section{Point Splitting Homology}

In this section we define Point-Slitting-Homology. We adapt the definition of \cite{PSPCS1} using an extended version of the decorated graphs introduced in \cite{PSPCS1}.  
The definition of nice $PS$-cycles of \cite{PSPCS1} is generalized in order to include a not trivial limit support.

\subsection{Decorated Graphs} \label{graphs-section}

A decorated graph  $G $ is defined by the data 
\begin{equation} \label{array-point-splitting}
    ( \chi_{\vartriangle} , d_{\vartriangle} , V_{\vartriangle}, Comp_0 , (V_c,d_c, g_c)_c , (H_v )_v  , E )
\end{equation} 
where 
\begin{itemize}
\item $ V_{\vartriangle}$ is a  finite set, $ \chi_{\vartriangle}  \in \Z_{\leq 0}, d_{\vartriangle}  \in \Z_{\geq 0}$  ; 
\item for each $c \in Comp_0 $, $V_c$ is a finite sets and $d_c, g_c \in \Z_{\geq 0}$;

Set 
$$V(G)= V_{\vartriangle} \sqcup \bigsqcup_{c \in Comp_0} V_c  .$$ 
\item for each $v \in V(G)$,  $H_v$ is a finite cyclically ordered set; 

Set 
$$H(G)= \bigsqcup_{v \in V(G)} H_v .$$
\item $E$ is a partition of $H$ in subset of cardinality two or one.  The sets of cardinality two are called internal edges $E^{in}(G)$, the sets of cardinality one are called external edges $E^{ex}(G)$;
\end{itemize}
For each $c \in Comp_0$ we set 
$$\chi_c= 2- 2g_c - |V_c|  - d_c.$$
We say that $G$ is stable if
$$ \chi_c - \frac{1}{2} |H_c| <0 \quad \forall c \in Comp_0.$$

There is an obvious notion of isomorphism of decorated graphs. 
We denote by $ \mathfrak{G}$ the set of isomorphism classes of stable decorated graphs.



Set 
$$ \chi(G) = \chi_{\vartriangle}  + \sum_{c \in Comp_0} \chi_c  - |E^{in}(G)|, \quad   d(G)= d_{\vartriangle} + \sum_c d_c  .$$
We denote by $\mathfrak{G}(\chi)$ the set of stable graphs $G$ with $\chi(G)= \chi$.





\subsubsection{Operation $\delta_v$, $v \in V(G)$}

Given $G \in \mathfrak{G}$ and $v \in V(G)$ with $H_v = \emptyset$ we define the graph $\delta_{v} G$ as follows:
\begin{itemize}
\item when $v \in V_{c}$ for $c \in Comp_0(G)$ ,  $\delta_{v} G$ is defined discarding $v$ from $V_{c}$ and replacing $d_{c}$ with $d_{c} +1$;
\item when $v_0 \in V_{\vartriangle}$ ,  $\delta_{v} G$ is defined discarding $v$ from $V_{\vartriangle}$ and replacing $d_{\vartriangle}$ with $d_{\vartriangle} +1$. 
\end{itemize}

\subsubsection{Operation $\delta_e$,  $e \in E^{in}(G) $}

Given $G \in \mathfrak{G}$  and $e = \{ h_1,h_2 \} \in E^{in}(G) $ we define the graph $\delta_e G$  as follows:


\begin{itemize}
\item
Assume $h_1 \in H_{v_1}$ and $h_2 \in H_{v_2}$ with $v_1 , v_2 \in V(G), v_1 \neq v_2$.  Define the order sets  $I_1$ and $I_2$  from  
$$H_{v_1 }= \{ h_1 , I_1  \} ,H_{v_2 }= \{ h_2 , I_2  \},$$ 
where the identities are considered as cyclically ordered sets. $V(\delta_e G)$ is obtained from $V(G)$ discarding $v_1,v_2$  and adding a  new vertex, which we denote by $v_0$. Set 
$$ H_{v_0} =  \{ I_1 , I_2  \} .$$
We have the sub-cases:
\begin{itemize}
\item $v_1 \in V_{c_1}$, $v_2 \in V_{c_2}$ for  $c_1 , c_2 \in Comp_0(G)$ with $c_1 \neq c_2$: $Comp(\delta_e G)$ is obtained from $Comp(G)$ replacing the components $c_1$ and $c_2$ with a unique component. Denote by $c_0$ the new component. $V_{c_0}$ is obtained from $V_{c_1} \sqcup V_{c_2} $ replacing the vertices $v_1$ and $v_2$ with $v_0$. We set
$d_{c_0}=d_{c_1} + d_{c_2}$, $  g_{c_0} = g_{c_1}+ g_{c_2}$.

\item $v_1 \in V_{\vartriangle}$, $v_2 \in V_{c}$ for  $ c \in Comp_0(G)$: $Comp_0(\delta_e G)$ is obtained discarding $c$  from $Comp_0(G)$, $V_{\vartriangle}$ is obtained from $V_{\vartriangle} \sqcup V_{c} $ replacing the vertices $v_1$ and $v_2$ with $v_0$,
$\chi_{\vartriangle}(\delta_e G) = \chi_{\vartriangle} + \chi_{c_0} -1$,
$d_{\vartriangle}(\delta_e G)=d_{\vartriangle} + d_{c_0}$.

\item $v_1,v_2 \in V_{c_0}$ for some $c_0  \in Comp_0(G)$:  $Comp_0(\delta_e G) =Comp_0(G)$,
$V_{c_0}$ is obtained replacing the vertices $v_1$ and $v_2$ with $v_0$, the genus $g_{c_0}$ is increased by one, 
 $d_{c_0}$ does not change.

\item $v_1,v_2 \in V_{\vartriangle}$:  
$V_{\vartriangle}$ is obtained replacing the vertices $v_1$ and $v_2$ with $v_0$,
$\chi_{\vartriangle}(\delta_e G) = \chi_{\vartriangle}  -1$, $d_{\vartriangle}(\delta_e(G))= d_{\vartriangle}(G)$.


\end{itemize}
\item
Assume $h_1,h_2 \in H_{v_0}$ for $v_0 \in V(G)$. Define the order sets  $I_1$ and $I_2$  from the relation 
$$H_{v_0 }= \{ h_1 , I_1 , h_2, I_2 \},$$ 
where the identity is considered as an identity of cyclic order sets .    $V(\delta_e G)$ is obtained from $V(G)$ replacing $v_0$ with two new vertices. Denote by $v_1, v_2$ the new vertices. We set 
$$ H_{v_1} =  I_1 \text{    }  H_{v_2} =  I_2 ,$$
as cyclic order sets.
We set $Comp_0(\delta_e G) = Comp_0(G)$. We have the sub-cases:
\begin{itemize}
\item $v_0 \in V_{c_0}$ with  $c_0 \in Comp_0(G)$: $V_{c_0}(\delta_e G)$ is obtained replacing $v_0$ with $v_1, v_2$,    $g_{c_0}$ and $d_{c_0}$ do not change. 

\item $v_0 \in V_{\vartriangle}$: $V_{\vartriangle}(\delta_e G)$ is obtained replacing $v_0$ with $v_1, v_2$,    
$\chi_{\vartriangle}(\delta_e G) = \chi_{\vartriangle}(G)  -1$.  $d_{\vartriangle}$ does not change.


\end{itemize}

\end{itemize}

\subsubsection{Marked Graphs} \label{marking-section}

Define the set of  decorated graphs $\mathfrak{G}_l$  as the set of pairs $(G,m)$, where $G \in \mathfrak{G}$ and  
$$m= \{ E_0, E_1,..., E_l \}$$ 
with $E_0 \subset E_1 \subset ... \subset E_l \subset E(G)$ be an increasing finite sequence of subsets of $E(G)$ with $E^{ex}(G) \subseteq E_0$.
The definition the operator $\delta_e$  works on $\mathfrak{G}_l$ for $e \in E^{in}(G) \setminus E_l$.  

 
 Denote
 $$ \mathfrak{G}_* = \sqcup_l  \mathfrak{G}_l .$$

Let  $cut_{E_0}G$ be the graph obtained cutting the edges $E_0$. So $E^{in }(cut_{E_0}G)= E^{in}(G) \setminus E_0^{in}, H(cut_{E_0}G)= H(G)$.  
We consider the decorated graphs up to the equivalence relation generated by
$$ (G,m ) \sim ( cut_{E_0}G, cut_{E_0}m   ) $$ 
where 
$$cut_{E_0}m =   \{ E^{ex}(cut_{E_0}G), E_1,..., E_l \}.$$
 
\subsection{Point-Splitting Cycles}  \label{PS-section}

As in \cite{PSPCS1} , in the following we consider singular chains up to triangulations  and reparametrizations.

Annulus Data are defined as in \cite{PSPCS1}. Point Splitting Cycles  are defined as in \cite{PSPCS1}, using the graphs of  section \ref{graphs-section} in place of the graphs considered in \cite{PSPCS1}:

Fix an Annulus Data $\underline{Z}^{\circledcirc}$. 
A point splitting cycle  $Z$ with annulus data $\underline{Z}^{\circledcirc}$ consists on a collection of chains
\begin{equation} \label{collection-chains-definition}
( Z_{G,m} )_{(G,m) \in \mathfrak{G}_*}. 
\end{equation}
we require that 
\begin{itemize}
\item  the transversality condition is satisfied;
\item  the forgetful compatibility is satisfied  with annulus data $\underline{Z}^{\circledcirc}$;
\item the cycle condition is satisfied .
\end{itemize}

The transversality condition, forgetful compatibility condition and cycle condition are stated as in \cite{PSPCS1}.

We denote by $ \mathcal{Z}(M|  {Z}^{\circledcirc}) $ the vector space of $PS$-cycles with annulus data $\underline{Z}^{\circledcirc}$.



\subsubsection{Isotopies} \label{isotopies-section}

Similarly we can consider isotopies of $PS$-cycles:

Let $\tilde{Z}^{\circledcirc}$ be an isotopy between annulus data ${Z}^{\circledcirc,0}$ and  ${Z}^{\circledcirc,1}$. Take $Z^0 \in \mathcal{Z}(M|  {Z}^{\circledcirc,0})$ and $Z^1 \in \mathcal{Z}(M|  {Z}^{\circledcirc,1})$. An isotopy  $\tilde{Z} \in \tilde{\mathcal{Z}}$ of $PS$-cycles  between $Z^0$ and $Z^1$  consists in a collection of Borel-Moore chains
\begin{equation} \label{collection-chains-isotopy}
\tilde{Z}=  ( \tilde{Z}_{G,m} )_{G,m} ,
\end{equation}
with
$$\tilde{Z}_{G,m} \in C_{|H(G)| + |m|+1} ( M^{H(G)} \times \R  , \mathfrak{o}_{H(G)})^{\text{Aut}(G,m)}  .$$
We require that
\begin{itemize}
\item there exists  $\lambda^{\dagger} >0$  big enough such that
$$ \tilde{Z}^{<-\lambda^{\dagger}} =  Z^0  \times  \R_{<-\lambda^{\dagger}}  , \quad \tilde{Z}^{>\lambda^{\dagger}} =   Z^1  \times \R_{>\lambda^{\dagger}}   ;$$

 \item 
the transversality condition is satisfied;


\item the forgetful compatibility condition is satisfied.

\item the cycle condition is satisfied.


\end{itemize}

We denote the space of isotopies of $PS$-cycles with $ \tilde{\mathcal{Z}}(M| \tilde{Z}^{\circledcirc}) $.

\subsubsection{Isotopy of Isotopies}

Isotopy of isotopies of $PS$-cycles are defined as in the case of $PS$-cycles of \cite{PSPCS1}, using the decorated graphs introduced in this paper.

For simplicity  of notations below we omit  the dependence on the annulus data. We fix an isotopy of isotopies of annulus data.

\begin{definition} 
Let $ Z^{0, 0},   Z^{1, 0}  ,    Z^{ 0, 1}  ,    Z^{ 1 , 1} $ be $PS$-cycles. Let $ Z^{0, \bullet}$ be an isotopy between $Z^{0, 0}$ and $  Z^{ 0, 1}$,  $ Z^{1, \bullet}$ isotopy between $Z^{1, 0}$ and $  Z^{ 1, 1}$, $ Z^{ \bullet,0}$ isotopy between $Z^{0, 0}$ and $  Z^{ 1, 0}$,  $ Z^{ \bullet,1}$ isotopy between $Z^{0, 1}$ and $  Z^{ 1, 1}$.

An isotopy of isotopies of $PS$-cycles between $ Z^{0, \bullet},   Z^{1, \bullet}  ,    Z^{ \bullet, 0}  ,    Z^{ \bullet , 1} $ consists in a collection 
of Borel-Moore chains
\begin{equation} \label{collection-chains-isotopy-isotopy}
\tilde{ \tilde{Z}}=  (\tilde{\tilde{Z}}_{G,m} )_{G,m} ,
\end{equation}
with
$$\tilde{\tilde{Z}}_{G,m} \in C_{|H(G)| + |m|+2} (M^{H(G)} \times \R_{\lambda} \times \R_{\lambda'}  , \mathfrak{o}_{H(G)})^{\text{Aut}(G,m)}  .$$
such that
\begin{itemize}


\item
There exists  $\lambda^{\dagger} \gg 0$ big enough such that:
\begin{itemize}
\item $\tilde{\tilde{Z}}^{\lambda' <-\lambda^{\dagger} } =  Z_{ G,m}^{0, \bullet} \times \R_{\lambda'<-\lambda^{\dagger}}$;
\item $\tilde{\tilde{Z}}^{\lambda' >\lambda^{\dagger} } =  Z_{ G,m}^{1, \bullet} \times \R_{\lambda'>\lambda^{\dagger}} $;
\item  $\tilde{\tilde{Z}} ^{\lambda <-\lambda^{\dagger} } =  Z_{ G, m}^{ \bullet, 0 } \times \R_{\lambda < -\lambda^{\dagger}}$;
\item $\tilde{\tilde{Z}}^{ \lambda >\lambda^{\dagger} } =  Z_{ G,m}^{ \bullet, 1} \times  \R_{\lambda >\lambda^{\dagger}}$  .
\end{itemize}

\item  the transversality condition is satisfied;
\item forgetful compatibility condition is satisfied;

\item the cycle condition is satisfied:


\end{itemize}

We denote by $  \tilde{\tilde{ \mathcal{Z}}}$ the vector space of isotopies of isotopies.

\end{definition}

\subsubsection{Extension Lemmas}
 
 The extension Lemmas of \cite{PSPCS1} can be easily adapted to the framework of this paper. 
 
 The partial order on the graphs and the truncated set of $PS$-cycles $\mathcal{Z}_{\prec G}$ are defined as in \cite{PSPCS1} .




\begin{lemma} \label{extension-PS-lemma}

Let $G \in \mathfrak{G}$. For any $Z_{\prec G} \in \mathcal{Z}_{\prec G}$ there exists  $Z_{\preccurlyeq G}  \in \mathcal{Z}_{\preccurlyeq G}$ extending $Z_{\prec G}$. 

\end{lemma}






  \begin{lemma} \label{extension-PS-lemma-isotopy}
  
Let $G \in \mathfrak{G}$ with $H(G) \neq \emptyset$. 
 Let $Z_{\preccurlyeq G}^0, Z_{\preccurlyeq G}^1 \in  \mathcal{Z}_{\preccurlyeq G} $ and  let  $\tilde{Z} _{\prec G} \in  \tilde{ \mathcal{Z}}_{\prec G} $ be an isotopy between
 $Z_{\prec G}^0$ and $ Z_{\prec G}^1$. 
 There exists  
 $\tilde{Z} _{\preccurlyeq G} \in  \tilde{ \mathcal{Z}}_{\preccurlyeq G} $ isotopy between 
 $Z_{\preccurlyeq G}^0$ and $Z_{\preccurlyeq G}^1 $ extending $\tilde{Z} _{\prec G} $.     
 \end{lemma}

An application of these two Lemmas yields the following:

  \begin{lemma} \label{extension-isotopies-inductive}
 Let $\underline{Z}_{\prec G}^0 \in  \mathcal{Z}_{\prec G} $,  $\underline{Z}_{\prec G}^1 \in  \mathcal{Z}_{\prec G} $ and    $\tilde{Z} _{\prec G} \in  \tilde{ \mathcal{Z}}_{\prec G} $  isotopy between $\underline{Z}_{\prec G}^0$ and $\underline{Z}_{\prec G}^1$.
 Let $\underline{Z}_{\preccurlyeq G}^0    \in  \mathcal{Z}_{\preccurlyeq G} $ extending $\underline{Z}_{\prec G}^0 $.
 

There exist
$\underline{Z}_{\preccurlyeq G}^1 \in  \mathcal{Z}_{\preccurlyeq G} $ and 
 $\tilde{Z} _{\preccurlyeq G} \in  \tilde{ \mathcal{Z}}_{\preccurlyeq G} $  isotopy between $\underline{Z}_{\preccurlyeq G}^0$ and $\underline{Z}_{\preccurlyeq G}^1  $ such that
\begin{itemize} 
\item $\underline{Z}_{\preccurlyeq G}^1  $ extends  $\underline{Z}_{\prec G}^1  $;
\item  $\tilde{Z} _{\preccurlyeq G}  $ extends $\tilde{Z} _{\prec G}  $. 
 \end{itemize}
 \end{lemma}

We also have the analogous for isotopies of isotopies:

\begin{lemma} \label{extension-isotopies-isotopies}

Let  $ Z_{ \preccurlyeq G}^{0, 0},   Z_{ \preccurlyeq G}^{1, 0}  ,    Z_{ \preccurlyeq G}^{ 0, 1}  ,    Z_{ \preccurlyeq G}^{ 1 , 1} $  be $PS$-cycles . Let $ Z_{ \preccurlyeq G}^{0, \bullet}$ be an isotopy between $Z_{ \preccurlyeq G}^{0, 0}$ and $  Z_{ \preccurlyeq G}^{ 0, 1}$,  $ Z_{ \preccurlyeq G}^{ \bullet,0}$ be an isotopy between $Z_{ \preccurlyeq G}^{0, 0}$ and $  Z_{ \preccurlyeq G}^{ 1, 0}$,  $ Z_{ \preccurlyeq G}^{ \bullet,1}$ be an isotopy between $Z_{ \preccurlyeq G}^{0, 1}$ and $  Z_{ \preccurlyeq G}^{ 1, 1}$,
$ Z_{ \prec G}^{1, \bullet}$ be an isotopy between $Z_{ \prec G}^{1, 0}$ and $  Z_{ \prec G}^{ 1, 1}$.

Let $\tilde{\tilde{Z}} _{\prec G} \in  \tilde{\tilde{ \mathcal{Z}}}_{\prec G} $  be an isotopy of isotopies between  $Z_{ \prec G}^{ \bullet ,0 },  Z_{ \prec G}^{ \bullet ,1 }, Z_{ \prec G}^{0 ,\bullet} ,Z_{ \prec G}^{1 ,\bullet} $. 

There exists $Z_{ \preccurlyeq G}^{1, \bullet} \in \tilde{ \mathcal{Z}}_{\preccurlyeq G} $ extending $Z_{ \prec G}^{1, \bullet}$ and  $\tilde{\tilde{Z}}_{\preccurlyeq G} \in  \tilde{\tilde{ \mathcal{Z}}}_{\preccurlyeq G} $ isotopy of isotopies between  $Z_{ \preccurlyeq G}^{ \bullet ,0 },  Z_{ \preccurlyeq  G}^{ \bullet ,1 }, Z_{ \preccurlyeq  G}^{0 ,\bullet} ,Z_{ \preccurlyeq  G}^{1 ,\bullet} $ extending
$\tilde{\tilde{Z}} _{\prec G} $. 

 \end{lemma}

In the following we use the notation $\underline{Z}_{G}$ for the array $(\underline{Y}_{G,m})_m$.

In the case  $H(G) \neq \emptyset$ the extension of Lemma  \ref{extension-PS-lemma} is essentially unique in the following sense: if $Z_{\preccurlyeq G}  \in \mathcal{Z}_{\preccurlyeq G}$ with vanishing restriction $Z_{\prec G} \in \mathcal{Z}_{\prec G}$ there exists $\underline{Y}_{G} = (\underline{Y}_{G,m})_m$ 
with $\underline{Y}_{G,m} \in  C_{|H(G)| + |m|+1}(\mathfrak{L}_{\hat{V}(G)}(M) , \mathfrak{o}(H(G)))^{\text{Aut}(G,m)} $ such that 
 \begin{equation} \label{trivial-cycle}
 \underline{Z}_{G,m} = \partial^{red}  \underline{Y}_{G,m}  -   \sum_{0 \leq i \leq |m|} (-1)^i {\underline{Y}}_{G, \partial_i m}   .
  \end{equation}
This fact can be seen as a consequence of Lemma \ref{extension-PS-lemma-isotopy}. $\underline{Y}_{G}$  can be constructed directly using a simplified version of the inductive argument of the proof of Lemma  \ref{extension-PS-lemma} as follows. The case $E^{ex}(G) \neq \emptyset$ is trivial since we use forgetful compatibility. Let us assume $E^{ex}(G) = \emptyset$. 
Assume we have defined $\underline{Y}_{G, m'}$ and $\underline{Y}_{G,m' \sqcup \{ E(G)\}}$ for each $|m'| <|m|$ with $E(G) \notin m'$.
Define 
$  \underline{Y}_{G,  m  }$ as a small perturbation of  $  (-1)^l(   \sum_i (-1)^i \underline{Y}_{G, \partial_i m \sqcup \{ E(G) \}}+  \underline{Z}_{G,m}$ and 
$\underline{Z}_{G,m \sqcup \{ E(G)\}}$ as the chain associated to the related cobordism.



\subsection{Nice $PS$-cycles}   \label{nice-section}

We now extend the definition of Nice $PS$-cycle given in \cite{PSPCS1}.

\subsubsection{Limit Support}

In order to extend the  definition of Nice $PS$-cycle of \cite{PSPCS1}  we need to define the notion of limit support of a sequence of chains on the multi-loop space.   In the context of   \cite{PSPCS1} is only considered trivial limit support.

For $G \in \mathfrak{G}$ and $\underline{C}= ( {\underline{C}}^{(n)} )_n$ a sequence of chains  on $\mathfrak{L}_GM$, 
we say that the limit support of $\underline{C}$
is finite if there exists a finite set of multi-loops
$\{ \mathbf{w}^j\}_{j \in J}$  such that for each $\epsilon>0$, the following happen for  $n$ big enough. For each $(w_v)_v \in supp (C^{(n)})$ there exists  
 $I \subset V_{\vartriangle}(G)$ such that
\begin{itemize}
\item  for each $c \in Comp_0$, $Diam(\sqcup_{v \in V_c} supp(w_v)) < \epsilon $;
\item   for each $v \in V_{\vartriangle} \setminus I$, $Diam(supp(w_v)) < \epsilon $;
\item there exists  $j \in J$ such that
$(w_v)_{v \in I}$ is $\epsilon$-close    to $\mathbf{w}^j$     in the $C^0$ topology .
\end{itemize}
We set 
$$\lim supp(\underline{C}) =  \{ \mathbf{w}^j\}_{j \in J} $$
where $\{ \mathbf{w}^j\}_{j \in J}$ is  the minimal set of multi-loops for which the above condition is satisfied.

\subsubsection{Nice $PS$-cycles}

\begin{definition} 
 A nice annulus data $\underline{Z}^{\circledcirc, \diamond}$ consists in a sequence of annulus data $(\underline{Z}^{\circledcirc, (n)})_n$ such that 
 \begin{itemize}
\item $ \lim Supp(\underline{Z}_{G}^{\circledcirc, \diamond})= \emptyset $ for each $G \in Ann0$ ;
\item there exists  a sequence $(\tilde{\underline{Z}}^{\circledcirc, int, (n)})_n$, where $ \tilde{\underline{Z}}^{\circledcirc, int, (n)}$  is an isotopy of annulus data between $ \tilde{\underline{Z}}^{\circledcirc,  (n-1)} $ and $ \tilde{\underline{Z}}^{\circledcirc,  (n)}$, with
$ \lim Supp( pr_*(\tilde{\underline{Z}}^{\circledcirc, int}) =  \emptyset.$
\end{itemize} 
\end{definition}

  \begin{definition}
 Fix  a  nice annulus data $\underline{Z}^{\circledcirc, \diamond}$.  
  A nice $PS$-cycle  $Z^{\diamond } $  consists in a sequence of $PS$-cycles $( Z^{(n)} )_n$ with
\begin{itemize}
\item $Z^{(n)} \in \mathcal{Z}(M| \underline{Z}^{\circledcirc, (n)})$ for each $n$;
\item   $\lim Supp(\underline{Z}_{G,m})$  finite for each $(G,m) \in \mathfrak{G}_*$;
\item  there exists a sequence $(\tilde{Z}^{ int, (n)})_n$, where  $\tilde{Z}^{int, (n)}  \in \tilde{\mathcal{Z}}(M| \tilde{\underline{Z}}^{\circledcirc, int, (n)})$  isotopy of  $PS$-cycles   between $Z^{(n-1)} $ and $Z^{(n)}$, with
$$ \lim Supp( pr_*(\tilde{\underline{Z}}_{G,m}^{int})) =   \lim Supp(\underline{Z}_{G,m})  \qquad \forall (G,m) \in \mathfrak{G}_*.$$
\end{itemize} 
\end{definition}
We denote by $\mathcal{Z}(M| \underline{Z}^{\circledcirc, \circ})$ the space of nice $PS$-cycles with annulus data  $\underline{Z}^{\circledcirc, \diamond}$ . 



\begin{lemma} \label{extension-nice}
For any $\underline{Z}_{\prec G}^{\diamond} \in  \mathcal{Z}_{\prec G}^{\diamond} $ there exists $\underline{Z}_{\preccurlyeq G}^{\diamond} \in  \mathcal{Z}_{\preccurlyeq G}^{\diamond} $ extending $\underline{Z}_{\prec G}^{\diamond}$ with 
\begin{equation} \label{lim-supp-inductive}
 \lim Supp (\underline{Z}_{\preccurlyeq G}^{\diamond})= \lim Supp (\underline{Z}_{\prec G}^{\diamond}).
 \end{equation}
\end{lemma}
\begin{proof}
Take  $\underline{Z}_{\preccurlyeq G}^{(0)} \in  \mathcal{Z}_{\preccurlyeq G}$ extending $\underline{Z}_{\preccurlyeq G}^{(0)} \in  \mathcal{Z}_{\preccurlyeq G}$. Using induction on $n$, assume we have defined $\underline{Z}_{\preccurlyeq G}^{(n)}$. Apply Lemma  \ref{extension-isotopies-inductive} 
to $\underline{Z}_{\preccurlyeq G}^{(n)}, \tilde{ \underline{Z}}^{int, (n)}_{\prec G}$ to obtain $\underline{Z}_{\preccurlyeq G}^{(n+1)}, \tilde{ \underline{Z}}^{int, (n)}_{\preccurlyeq G}$. It is easy to adapt  the proof of Lemma  \ref{extension-isotopies-inductive} in order to satisfies condition  (\ref{lim-supp-inductive}).
\end{proof}



\subsubsection{Isotopies}

To define isotopies of nice $PS$-cycles we need first to define the one-parameter family version of the limit support introduced above. 

We consider families of multi loops $ \mathfrak{L}_GM \times \R$ of the type:
\begin{equation} \label{basic-family-loops}
\{ (\mathbf{w}^\lambda,\lambda) \}_{ a \leq  \lambda \leq b}
\end{equation}
where $\mathbf{w}^\lambda$ is a piecewise smooth one parameter family of multi-loops, $a$ is a real number or  $ - \infty$,   $b$ is a real number or  $ + \infty$. We assume that $\mathbf{w}^\lambda$ is independent on $\lambda$ for $\lambda \gg 0$ and $\lambda \ll 0$.


Let $\tilde{\underline{C}}= ( {\tilde{\underline{C}}}^{(n)} )_n$ be a sequence of chains on $ \mathfrak{L}_GM \times \R$.  We say that the limit support $\lim \tilde{supp}(\tilde{\underline{C}}) $ of $\tilde{\underline{C}}$
is finite if there exists a set  
 $\{ (\mathbf{w}^j, \lambda^j) \}_{j \in J} $ which is union of a finite numbers of families as (\ref{basic-family-loops})
 such that for each $\epsilon>0$, for  $n$ big enough, for each $( (w_v)_v, \lambda) \in supp (\tilde{C}^{(n)})$ there exists   $I \subset V_{\vartriangle}(G)$ such that
\begin{itemize}
\item  for each $c \in Comp_0$, $Diam(\sqcup_{v \in V_c} supp(w_v)) < \epsilon $;
\item   for each $v \in V_{\vartriangle} \setminus I$, $Diam(supp(w_v)) < \epsilon $;
\item 
there exists  $j \in J$ such that 
$( (w_v)_{v \in I}, \lambda)$ is close    to $( \mathbf{w}^j, \lambda^j)$ in the $C^0$ topology
\end{itemize}
 We write
$$\lim \tilde{supp}(\tilde{\underline{C}}) =  \{ ( \mathbf{w}^j ,\lambda^j) \}_{j \in J} ,$$
where $ \{ ( \mathbf{w}^j, \lambda^j) \}_{j \in J}$ is the minimal set such that the properties above hold.


\begin{definition}
Let $\tilde{\underline{Z}}^{\circledcirc, \diamond}$ be an isotopy of nice annulus data  between nice annulus data  $\underline{Z}^{0, \circledcirc, \diamond}$ and $\underline{Z}^{1,\circledcirc, \diamond}$.

For  $Z^{0,\diamond} \in \mathcal{Z}(M|   Z^{\circledcirc, 0 ,\diamond}    )$ and  $Z^{1 , \diamond} \in \mathcal{Z}(M|   Z^{\circledcirc, 1 ,\diamond}    )$ 
an isotopy  of nice $PS$-cycle between   $Z^{0, \diamond} $ and $Z^{1, \diamond} $  with annullus data $\tilde{\underline{Z}}^{\circledcirc, \diamond}$ consists in a sequence $\tilde{Z}^{\diamond }= ( \tilde Z^{(n)} )_n$ with $\tilde{Z}^{(n)} \in \tilde{\mathcal{Z}}(M| \tilde{\underline{Z}}^{\circledcirc, (n)})$   isotopy of $PS$-cycles between $Z^{0, (n)}$ and $Z^{1,(n)}$ such that
$\lim \tilde{Supp} (\tilde Z_{G,m})$ is  finite  for each $(G,m)$.
\end{definition}

We denote by $ \tilde{\mathcal{Z}}(M| \tilde{\underline{Z}}^{\circledcirc, \diamond})$ the vector space of isotopies of nice $PS$-cycle  with nice annullus data $\tilde{\underline{Z}}^{\circledcirc, \diamond}$.

\subsubsection{Obstructions}  \label{obstruction-section}

 
Let $Z^{0,\diamond} , Z^{1 , \diamond} \in \mathcal{Z}(M|   Z^{ \circledcirc ,\diamond}    )$ . 
 Let $\tilde{\underline{Z}}_{\prec G} \in  \tilde{\mathcal{Z}}_{\prec G}(M| \tilde{\underline{Z}}^{\circledcirc, \diamond})$ be an isotopy between   $Z_{\prec G}^{0,\diamond} , Z_{\prec G}^{1 , \diamond}$ .  
 In this subsection we determine the conditions for which there exists an isotopy $\tilde{\underline{Z}}_{\preccurlyeq G}$ extending $\tilde{\underline{Z}}_{\prec G}$ extending the result of \cite{PSPCS1}.    
  
As in \cite{PSPCS1}  the obstruction to the existence of $\tilde{\underline{Z}}_{\prec G}^{(n)}$ is concentrated on the graphs $G$ with  $H(G) = \emptyset$. The main difference with \cite{PSPCS1}  is that now we need to consider the possibility that   $V(G) \neq \emptyset$. 



Fix $\lambda^{\dagger}$ big enough as in definition of isotopy. 
Denote with $pr$ the projection orthogonal to the $\lambda$. variable restricted to $(-\lambda^{\dagger}, \lambda^{\dagger})$.
Define the obstruction as
$$\mathfrak{Obs}_G^{(n)} =  \underline{Z}^{1,(n)}_G -  \underline{Z}^{0,(n)}_G + pr_*( \sum_{(G', e')}  \delta_{e'} \tilde{\underline{Z}}_{G'}^{(n)}|_{(-\lambda^{\dagger}, \lambda^{\dagger})} +   \sum_{(G', v')} \delta_{v'} \tilde{\underline{Z}}_{G'}^{(n)}|_{(-\lambda^{\dagger}, \lambda^{\dagger})   }) .$$
(Here we have restricted the isotopy on the interval $(-\lambda^{\dagger}, \lambda^{\dagger})$)


From the definition of limit support, for $n$ big enough, 
there exists a collection of rational numbers $(\mathfrak{obs}_{(G, \mathbf{w})}^{(n)})_{\mathbf{w} \in  \lim supp(\mathfrak{Obs}_G)} $
such that  $\mathfrak{Obs}_G^{(n)} $ agrees with
$$   \sum_{\mathbf{w} \in \in \lim supp(\mathfrak{Obs}_G)} (\mathfrak{obs}_{(G, \mathbf{w})} ) \mathbf{w}$$
up to a small isotopy.


We have the following
\begin{lemma}
The obstruction $\mathfrak{obs}_{(G, \mathbf{w})}^{(n)}$ does not depend on $n$. $\tilde{\underline{Z}}_{ G}$ exists if and only if $\mathfrak{obs}_{(G, \mathbf{w})} =0$ for each $\mathbf{w} \in \lim supp(\mathfrak{Obs}_G)$.
\end{lemma}
The Lemma is proved as in \cite{PSPCS1}.

\subsection{Point-Splitting-Homology}

We say that an isotopy $\tilde{Z}  \in \tilde{\mathcal{Z}}(M| {\underline{Z}}^{\circledcirc})$ has finite limit support if $\lim Supp(pr_*(\tilde{Z}_{G,m}))$ is finite for each $(G,m)$ , where $pr :   \mathfrak{L}_{ G}(M) \times [0,1] \rightarrow   \mathfrak{L}_{ G}(M)  $ is the obvious projection.
Note that this condition implies that $\lim Supp (Z^{0, \diamond}  ) =  \lim Supp (Z^{1, \diamond}  ) $.

\begin{definition}

 The Point Splitting Homology  $PSH(M| \underline{Z}^{\circledcirc, \diamond})$ is defined as the vector space of  nice $PS$-cycles $ \mathcal{Z}(M|   Z^{\circledcirc ,\diamond}    )$ modulo isotopies with finite limit support.


\end{definition}

We have the following
 \begin{proposition} \label{transfer-annulus}
 Let $\tilde{Z}^{\circledcirc, \diamond}$ be an isotopy between nice annulus data  $Z^{\circledcirc, 0 , \diamond }$  and  $Z^{\circledcirc, 1 , \diamond}$.




 
 For each $Z^{0, \diamond} \in \mathcal{Z}(M|   Z^{\circledcirc, 0 ,\diamond}    )$ there exists $Z^{1, \diamond} \in \mathcal{Z}(M|   Z^{\circledcirc, 1 ,\diamond}    )$ and $\tilde{Z}^{\diamond} \in \tilde{ \mathcal{Z}}(M|   \tilde{Z}^{\circledcirc , \diamond}    )$  isotopy between $Z^0$ and $Z^1  $ with 
  \begin{equation} \label{limit-supp-transfer-annulus}
  \lim Supp (Z^1)= \lim Supp (pr_* \tilde{Z}) = \lim Supp (Z^0). 
  \end{equation}
 
 The  $PSH$-class $[Z^1] \in PSH(M|   Z^{\circledcirc, 1 ,\diamond}    )$ is determined by the $PSH$-class $[Z^0] \in PSH(M|   Z^{\circledcirc, 0,\diamond}    )$. We obtain a map
 \begin{equation} \label{transfer-map-annulus}
 transfer_{  \tilde{Z}^{\circledcirc , \diamond} }   : PSH(M| Z^{\circledcirc, 0 , \diamond}  )  \rightarrow PSH(M| Z^{\circledcirc, 1 , \diamond }  ) .
  \end{equation}
Take $\tilde{Z}^{\circledcirc, \prime }$ another isotopy between nice annulus data  $Z^{\circledcirc, \diamond, 0 }$  and  $Z^{\circledcirc, 1 }$. If $\tilde{Z}^{\circledcirc, \diamond, \prime }$  is related to  $\tilde{Z}^{\circledcirc , \diamond}$ by an isotopy of isotopies we have
$$transfer_{  \tilde{Z}^{\circledcirc , \diamond} }  =  transfer_{  \tilde{Z}^{\circledcirc , \diamond, \prime} } .$$
 
 \end{proposition}

 \begin{proof}

As usual, we proceed by induction on  graphs. Assume we have defined $\tilde{\underline{Z}}_{\prec G}^{(n)}, {\underline{Z}}_{\prec G}^{1,(n)}$ for each $n$.
Apply  Lemma  \ref{extension-isotopies-inductive} to ${\underline{Z}}_{\preccurlyeq G}^{0,(n)}, \tilde{\underline{Z}}_{\prec G}^{(n)}$ to obtain 
${\underline{Z}}_{\preccurlyeq G}^{1,(n)}, \tilde{\underline{Z}}_{\preccurlyeq G}^{(n)}$.
The construction can be made such that (\ref{limit-supp-transfer-annulus}) is satisfied.
 


Let $Z^{1,\prime} \in \mathcal{Z}(M|   Z^{\circledcirc, \diamond, 1}    )$ and  $\tilde{Z}^{\prime} \in \tilde{ \mathcal{Z}}(M|   \tilde{Z}^{\circledcirc }    )$ be a $PS$-cycle and an isotopy of $PS$-cycle obtained from $\tilde{Z}^{\circledcirc, \prime }$.  Applying inductively Lemma \ref{extension-isotopies-isotopies}  we can define an isotopy of isotopies between $\tilde{Z}$ and $\tilde{Z}^{\prime}$. In particular, it yields an isotopy between $Z^1$ and $Z^{1,\prime}$. 


\end{proof}

\subsubsection{Euler Structures}

In \cite{PSPCS1} we have proved a relation between Annulus data and Euler Structures.  In order to understand this relation it is useful to consider an alternative formulation of the notion of Euler Structure. We consider pairs $U^{\blacktriangle}=(U_0 , U_1)$ where
\begin{itemize}
\item
 $U_0$ is a three chain on $M \times M$ close in the $C^0$ topology to $Diag_M$ and transversal to $Diag_M$;
\item $U_1$ is a one-chain on $M$ with
$$\partial U_1 = U_0 \cap Diag_M.$$
Here we consider $U_0 \cap Diag_M$ as a chain on $Diag_M \cong M$.  
\end{itemize}

There is an obvious notion of isotopy of Euler Structures. Denote by $\mathfrak{Eul}(M)^{\blacktriangle}$  the set of isotopy classes of the $U^{\blacktriangle}$ .  The group $H_1(M, \Z)$ acts  on $\mathfrak{Eul}(M)^{\blacktriangle}$ by addiction on $U_1$.
With this action, $\mathfrak{Eul}(M)^{\blacktriangle}$ is a torsor on $H_1(M, \Z)$.  
It is easy to see that  
$  \mathfrak{Eul}(M) \cong \mathfrak{Eul}(M)^{\blacktriangle},$ where $\mathfrak{Eul}(M)$ denotes the space of homology classes of the usual Euler Structures. 

To an annulus data $Z^{\circledcirc}$ in \cite{PSPCS1} we associated an Euler Structure $U^{\blacktriangle}$,  canonically defined up to a set of contractible choices.  

Proposition \ref{transfer-annulus} tells us that $PSH(M |   U^{\blacktriangle} ) $ is well defined.


\subsection{Factorization Property}

In this section we extend the factorization property considered in \cite{PSPCS1} to the $PSH$-classes considered in this paper. 

Let $\mathfrak{G}^{\divideontimes} \subset \mathfrak{G}$ be the set of decorated  graphs $G \in \mathfrak{G}$ with $V_{\vartriangle}  =\emptyset, \chi_{\vartriangle}=d_{\vartriangle} =0$. The set $\mathfrak{G}^{\divideontimes}$ may be identified with the set graphs used in \cite{PSPCS1}.

\subsubsection{Factorization $PS$-cycles} \label{factorization-cycles-section}

 Set 
$$\mathfrak{G}^{!}=  \mathfrak{G} \times \mathfrak{G}^{\divideontimes}.$$
Denote by $G^1 \oplus G^2$ the disjoint union of the decorated graphs $G^1 \in \mathfrak{G} $ and $G^2 \in \mathfrak{G}^{\divideontimes}$.

Let $\mathfrak{G}^{!}_*$ the set of pairs $((G^1,G^2) , m)$ with $(G^1,G^2) \in \mathfrak{G}^{!}$ and $m$ is a marking  for $G^1 \oplus G^2$   as in  Subsection \ref{marking-section} . Note that 
define $m$ is equivalent to define  a marking $m^1$ for $G^1$ and a marking $m^2$ for $G^2$. We write $m= m^1 \oplus m^2$.

Fix an  annulus data $\underline{Z}^{\circledcirc}$.
A factorization $PS$-cycles $Z^!$ 
with annulus data $\underline{Z}^{\circledcirc}$  
consists on a collection of chains
$$( Z^!_{(G^1,G^2) , m)} \}_{((G^1,G^2),m) \in \mathfrak{G}_*^!} $$
with
$$ Z^!_{(G^1,G^2) , m)}  \in C_{|H(G^1)| + |H(G^2)|+ |m|} ( M^{H(G^1) \sqcup H(G^2)} , \mathfrak{o}_{H(G^1) \sqcup H(G^2)})^{\text{Aut}(G^1,G^2) , m)}  .$$
such that
\begin{itemize}
\item the transversality condition is satisfied; 
\item  the forgetful compatibility is satisfied  with annulus data $\underline{Z}^{\circledcirc}$ ;
\item the cycle condition is satisfied .
\end{itemize}




The transversality condition, the forgetful compatibility condition and the cycle condition for factorization $PS$-cycles are defined analogously to Subsection \ref{PS-section}.




 We denote by $\mathcal{Z}^!$  the vector space of factorization $PS$-cycles.

Isotopies of  factorization $PS$-cycles can be defined analogously  modifying the definition of isotopies of  subsection ( \ref{isotopies-section}).


\subsubsection{Factorization Map}

Define the factorization map between the usual space of $PS$-cycles and the factorization $PS$-cycles:
\begin{equation} \label{factorization-map}
 \mathbf{fact} :  \mathcal{Z} \rightarrow  \mathcal{Z}^!
\end{equation}
$$    \mathbf{fact}(Z)_ {( G^1,G^2),m  ) }  := Z_{(G^1 \oplus G^2,  m) }. $$

\begin{lemma}
For each $PS$-cycle $Z$,  $   \mathbf{fact}(Z)$ is a factorization $PS$-cycle.
\end{lemma}

We extend (\ref{factorization-map}) to isotopies
\begin{equation} \label{factorization-map-isotopy}
 \mathbf{fact} : \tilde{\mathcal{Z}} \rightarrow  \tilde{\mathcal{Z}}^!
\end{equation}
in a straightforward way.


\subsubsection{Product}

Define the  product
\begin{equation} \label{product-factorization}
  \boxtimes : \mathcal{Z}  \times  \mathcal{Z}^{\divideontimes} \rightarrow \mathcal{Z}^! ,
  \end{equation}
$$   (Z^1, Z^2) \rightarrow  Z^1\boxtimes Z^2$$
\begin{equation} \label{cup-formula}
(Z^1\boxtimes Z^2)_{(G^1,G^2), m } := \sum_{0 \leq r \leq l}  Z^1_{(G^1, m^1_{[0,r]})} \times  Z^2_{(G^2, m^2_{[r,l]})}. 
\end{equation}
Here we have used the notation $m_{[a,b]} := \{ E_i \}_{a \leq i \leq b}$ when $m = \{ E_i \}_{0 \leq i \leq l}$ and $0 \leq a \leq b \leq l$.

\begin{lemma}
For each $PS$-cycles $Z^1, Z^2$,  $ Z^1\boxtimes Z^2$ is a factorization $PS$-cycle.
\end{lemma}

We extend the product (\ref{product-factorization} ) to isotopies
\begin{equation} \label{product-factorization-isotopies}
  \boxtimes : \tilde{\mathcal{Z}}  \times  \tilde{\mathcal{Z}}^{\divideontimes} \rightarrow \tilde{\mathcal{Z}}^! ,
  \end{equation}

\subsubsection{Nice PS-cycles}

Nice factorization $PS$-cycles $\mathcal{Z}^{!, \diamond}$ can be defined as in section  \ref{PS-section} using  factorization $PS$-cycles  instead of $PS$-cycles. The isotopies $\tilde{Z}^{int,(n)}$ used in the definition are isotopies of  factorization $PS$-cycles.

From $\mathcal{Z}^{!, \diamond}$ we can define the factorization $PS$-homology $PSH^!(M  | \underline{Z}^{\circledcirc, \diamond}  )$ taking the equivalence classes.


  \begin{lemma}
 For ${Z}^{ \diamond}  = (Z^{(n)})_n \in \mathcal{Z}^{\diamond}$  
 $$ \mathbf{fact}(Z^{\diamond}):=(\mathbf{fact}(Z^{(n)}))_n$$
  defines a nice factorization $PS$-cycle. 
\end{lemma}

\begin{lemma}
For ${Z}^{ 1, \diamond}  = (Z^{1, (n)})_n \in   \mathcal{Z}^{ \diamond}$ and $ {Z}^{ 2, \diamond}  = (Z^{2, (n)})_n \in   \mathcal{Z}^{ \divideontimes, \diamond}$  , 
$$Z^{1, \diamond} \boxtimes Z^{2, \diamond}:= ( Z^{1,(n)} \boxtimes Z^{2,(n)})_n$$
 defines a nice factorization $PS$-cycle. 
\end{lemma}

\begin{lemma}
The product  (\ref{cup-formula}) induces  a product in $   PSH(M| \underline{Z}^{\circledcirc, \diamond})$
\begin{equation} \label{product-nice-homology}
    PSH(M| \underline{Z}^{\circledcirc, \diamond}) \boxtimes PSH^{\divideontimes}(M| \underline{Z}^{\circledcirc, \diamond}  ) \rightarrow PSH^!(M  | \underline{Z}^{\circledcirc, \diamond}  ).
\end{equation}

\end{lemma}

\subsubsection{Factorization Property}

\begin{definition}
We say that   $[Z^{\diamond}] \in    PSH(M| \underline{Z}^{\circledcirc, \diamond})$
 satisfies the factorization property if 
$$  \mathbf{fact}([Z^{\diamond}]) = [Z^{\diamond}] \boxtimes [Z^{\diamond, \divideontimes}] .$$
Here $Z^{\diamond, \divideontimes}$ is the restriction of $Z^{\diamond}$ to $\mathfrak{G}^{\divideontimes}$.
\end{definition}

Note that the factorization property $[Z^{\diamond}]$ implies in particular that  $ [Z^{\diamond, \divideontimes}] $ satisfies the factorization property in the sense of \cite{PSPCS1}.

\section{Partition Function }

Let $M$ be a compact oriented three-manifold and let $\mathfrak{g}$ be a principal flat vector bundle over $M$.  Denote by $H^*(M, \mathfrak{g})$  the cohomology  of $M$ with coefficients in $\mathfrak{g}$ endowed with the odd symplectic form induced by the pairing
$$\langle \alpha \otimes X,  \alpha' \otimes X' \rangle =
 \int_M \alpha \wedge \alpha'  \langle  X,  X' \rangle_{\mathfrak{g}} .$$
I
In \cite{PSPCS1} to a $PS$-cycle $Z$  of pure Chern-Simons it is associated a partition function $\mathfrak{P}(Z)$ satisfying the Quantum Master Equation. In this section we show how to slightly modify the construction of \cite{PSPCS1} in order to extend the definition of Partition Function to the $PS$-cycles considered in this paper. 
In the context of $PSPCS$ the role played by the Wilson loops and vertices of the graphs are treated on the same footing. 




\subsection{Partition Function}

In this section we define the partition function of a $PSH$-class as a formal power series on the variable $x \in (H^*(M, \mathfrak{g})[1])^*$. The construction follows the same lines of the case without Wilson loops considered in \cite{PSPCS1}, in contrast with  the standard approach to $PCS$ where a configuration space of points is introduced for each embedded link.

As in \cite{PSPCS1}, the partition function will depend not only on the choice of flat connection of the $\mathfrak{g}$-bundle, but also on the choice of a representation
$$ \rho :  \mathfrak{g} \rightarrow \mathfrak{gl}(n, \C) .$$
In the following we implicitly identify $\mathfrak{g}$ with a sub-Lie-algebra of $\mathfrak{gl}(n, \C)$ using $\rho$.

As in \cite{PSPCS1} we fix a finite dimensional vector subspace $\mathcal{A} < \Omega^*(M, \mathfrak{g})$ of the space of closed differential forms of $M$ with value in $\mathfrak{g}$ such that its map in cohomology of $M$ yields an isomorphism:
\begin{equation} \label{harmonic-isomorphism}
   \mathcal{A} \xrightarrow{\,\raisebox{-0.4ex}{$\sim$}\,} H^*(M, \mathfrak{g}) .
   \end{equation}
Equivalently we have  a linear map
$$ r_{\mathcal{A}}: H^*(M, \mathfrak{g}) \rightarrow \Omega^*(M, \mathfrak{g})$$
such that 
$$ r_{\mathcal{A}}([\alpha])= \alpha \quad \forall \alpha \in \mathcal{A} .$$

Applying  ( \ref{harmonic-isomorphism}) to the first factor of   $Identity \in H^*(M, \mathfrak{g})[1] \otimes (H^*(M, \mathfrak{g})[1])^*$ we obtain an element 
$$\mathfrak{x} \in \mathcal{A} \otimes \R[x] .$$

\subsubsection{Trace Map}

In the context of this paper we cannot use the trace map of \cite{PSPCS1}. We now generalize the trace map to the case that the loops are not necessary small.

Let $G \in \mathfrak{G}$. 
For $h \in H(G)$ denote by $\pi_h : M^{H(G)} \rightarrow M $  the projection on the point $h$, set
$$ \mathfrak{g}_h = \pi_h^* (\mathfrak{g}) . $$

Define the bundle $\mathfrak{g}_{H(G)}$ over $M^{H(G)}$
$$ \mathfrak{g}_{H(G)} = \bigwedge (\bigoplus_{h \in H(G)} \mathfrak{g}_h) .$$

Denote by $\mathfrak{g}_{H(G)}^{\vee}$ the dual of the vector bundle $\mathfrak{g}_{H(G)} $.

For $(w_v,t_v)_v \in \mathfrak{L}_{G}(M) $, the map
$$ \bigwedge_{h \in H(G)} X_h \rightarrow    \prod_v  \text{Tr} ( \prod_{h \in H_v}^{cyclic} ( X_{h} \text{hol}_{w_v^h}  ) ), $$
with
$$ X_h \in \mathfrak{g}_{w_v(t_v(h))}  \quad \forall h \in H(G),$$
defines  an element of the fiber of $\mathfrak{g}_{H(G)}^{\vee}$ on the point $(w_v(t_v(h)))_h \in M^{H(G)}$, with coefficients in $  \mathfrak{o}(H(G))$.
 Here we have used the notation $w_v^h$ for the restriction of $w_v$ to the interval between the marked point $h$ and the next marked point.



We obtain a map
\begin{equation} \label{multi-loop-fiber-map}
C_*(\mathfrak{L}_{G}(M) , \mathfrak{o}(H(G))^{\text{Aut}(G,m)} \rightarrow C_*(M^{H(G)},(\mathfrak{g}_{H(G)})^{\vee} )^{\text{Aut}(G,m)}.
\end{equation}

\subsubsection{Partition Function}

Fix a propagator $P \in \Omega_2(\text{Conf}_2(M),   \pi_1^*(\mathfrak{g}) \otimes \pi_2^*(\mathfrak{g})  )$  as in \cite{PSPCS1}.  The construction of $P$ depends on the choice of the sub-space $\mathcal{A}$ in  (\ref{harmonic-isomorphism}) and a differential form representing the generator of $H^2(\partial \text{Conf}_2(M) , \Z )^-$. Different choices leads to isotopic propagators.

As in \cite{PSPCS1}, to a decorated graph $(G,m)$  we can associate a manifold with corners $Conf_{G,m}(M)$.


To any pair $(G,m)$ such that, for each $k$ with $ 1 \leq k \leq l$, we have
  $$E_k= E_0 \sqcup \{e_1,e_2,  . . . ,e_k   \}, $$
for some ordered subset $\{e_1,e_2,  . . . ,e_l  \}$  of $E^{in}$, we associate a differential form
$$\underline{\Omega}_{G,m} \in  \R[x] \otimes \Omega^*(Conf_{G,m}(M),  \mathfrak{g}_{H(G)}  ) $$
defined by
\begin{equation} \label{graph}
\underline{\Omega}_{G,m}   = (-1)^l \bigwedge_{e \in E(G) \setminus E_l} \pi_e^*(P_s) \wedge \bigwedge_{i} \pi_{e_i}^*(dP_s) \wedge \bigwedge_{e \in E^{ext}(G)}  \pi_e^*( \mathfrak{x} )  .
\end{equation}
Here $P_s$ is the super-propagator corresponding to $P$, i.e. , the image of $P$ under the map $ \pi_1^*(\mathfrak{g}) \otimes \pi_2^*(\mathfrak{g})   \rightarrow \bigwedge(\pi_1^*(\mathfrak{g}) \oplus \pi_2^*(\mathfrak{g})) .$

As in \cite{PSPCS1}, the transversality condition implies that the the image of $ \underline{Z}_{G,m} $ by the map  (\ref{multi-loop-fiber-map}) can be lifted to the blow-up  $Conf_{G,m}(M))$. We use this to define the pairing $ \langle   \underline{\Omega}_{G,m} ,   \underline{Z}_{G,m}  \rangle$.

 
 The partition function of the $PS$-cycles $Z$ is defined as the formal power series
\begin{equation} \label{partition-definition}
    \mathfrak{P} (Z) :=  \sum_{G,m}  \hbar^{- \chi(G)}  N^{|D(G)|}  \langle   \underline{\Omega}_{G,m} ,   \underline{Z}_{G,m}  \rangle .
\end{equation}
Here,  as in \cite{PSPCS1} ,
$\mathfrak{P} (Z)$ defines a formal Laurent series $ \R((\hbar))[[x]]  $. Write
$$ \mathfrak{P} (Z) = \sum_{l,k} \hbar^l p_{k,l}(x) $$
where,  for each $k \geq 0$  , $p_{k,l}(x)$ is an homogeneus polynomial of degree $k$,  The stability condition on the graphs  tells us
$$ p_{k,l} \neq 0 \Rightarrow k + 3l \geq 0. $$

\subsubsection{Isotopies}

 The partition function for isotopies can be defined analogously as follows. 

Following \cite{PSPCS1}, take an isotopy of propagators $\tilde{P}$ .  To each  graph $(G,m)$ we associate the differential form
$$\tilde{\underline{\Omega}}_{G,m} \in  \R[x] \otimes \Omega^*(Conf_{G,m}(M) \times \R, \mathfrak{g}_{H(G)})   $$
defined by
\begin{equation} \label{graph-isotopy}
\tilde{\underline{\Omega}}_{G,m}   = (-1)^{|m|} \bigwedge_{e \in E(G) \setminus E_l} \tilde{\pi}_e^*(\tilde{P}_s) \wedge \bigwedge_{i} \tilde{\pi}_{e_i}^*(d \tilde{P}_s) \wedge \bigwedge_{e \in E^{ext}(G)}  \tilde{\pi}_e^*( \tilde{\mathfrak{x}} )  
\end{equation}


 To an isotopy $ \tilde{Z} = (\tilde{Z}_{G,m})_{G,m}$ of  $PS$-cycles it is associated the isotopy of partition functions defined by
\begin{equation} \label{partition-definition-isotopies}
    \mathfrak{P} (\tilde{Z}) :=  \sum_{G,m}  \hbar^{- \chi(G)}  N^{d(G)} \langle   \tilde{\underline{\Omega}}_{G,m} ,   \tilde{\underline{Z}}_{G,m}  \rangle
\end{equation}

As in \cite{PSPCS1}, 
$ \mathfrak{P} (\tilde{Z})$  defines a formal Laurent series with coefficients piecewise smooth differential forms on $\R$. We can write
$$ \mathfrak{P} (\tilde{Z}) = \sum_{l,k} \hbar^l \tilde{p}_{k,l}(x) $$
where now $\tilde{p}_{k,l}(x)$ is an homogeneus polynomial of degree $k$ with coefficients in piecewise smooth differential forms on $\R$.  The stability condition on the graphs  tells us
$$ \tilde{p}_{k,l} \neq 0 \Rightarrow k + 3l \geq 0 . $$

We can write
$$ \mathfrak{P} (\tilde{Z})= \mathfrak{Q} (\tilde{Z}) + \mathfrak{K} (\tilde{Z}) d \lambda , $$
where $\mathfrak{Q} (\tilde{Z})$ and $\mathfrak{K} (\tilde{Z}) $ are formal Laurent series with coefficients piecewise smooth functions on $\R$.

We have , for generic value of $\lambda$,
$$\mathfrak{Q} (\tilde{Z})= \{ \lambda \mapsto \mathfrak{P} ({Z}^{\lambda}) \}.$$
where $ \underline{{Z}}^{\lambda}$  is  the $PS$-cycle defined by the restriction of $\tilde{Z}$ to $\lambda$.

\begin{lemma}
$ \mathfrak{Q} (\tilde{Z})$ is a continuos function on the variable $\lambda$. 
\end{lemma}




\begin{proposition} \label{QME-proposition}
$\mathfrak{P} (\tilde{Z})$ satisfies the $QME$:
\begin{equation} \label{master-equation}
    d_{\lambda} \mathfrak{P} (\tilde{Z}) + \hbar \Delta \mathfrak{P} (\tilde{Z}) =0
\end{equation}
\end{proposition}

\section{Coherent Cycles} \label{coherent-section}

Fix an  an embedded oriented link
$$\mathbf{w} = ( w_v^{\sharp} )_{v \in V^{\sharp}} .$$  

We say that a frame $\mathbf{fr} \in \mathfrak{Fr}(M)$ is compatible with $\mathbf{w}$ 
 if there exists 
$\mathbf{a}=(a_1,a_2,a_3) \in \R^3$    such that $T_z \mathbf{w} = \langle a_1 \mathbf{fr}_1(z) +  a_1 \mathbf{fr}_2(z)  +  a_3 \mathbf{fr}_3(z) \rangle $ for each $z \in \mathbf{w}$.
 Denote by $ \mathfrak{Fr}(M, \mathbf{w}  )$ the set of frames  of $M$ compatible with $\mathbf{w}$.


We say that un Euler Structure $U$ is compatible with $\mathbf{fr}   \in \mathfrak{Fr}(M,   \mathbf{w}  ) $ if it is constant when written in the frame $\mathbf{fr} $, i.e., there exists $\mathbf{a}=(a_1,a_2,a_3) \in \R^3 \setminus \{ 0,0,0  \}$    such that  
$U =  a_1 \mathbf{fr}_1 +  a_1 \mathbf{fr}_2  +  a_3 \mathbf{fr}_3 $.
The space of Euler Structures  compatible with a frame $\mathbf{fr}  $  is homeomorphic to $ \R^3 \setminus \{ 0,0,0  \}$.

Two Euler Structures compatible with  $\mathbf{fr}$  are isotopic with isotopy determined up to isotopy of isotopies, hence Lemma \ref{deformation-cycles} yields canonical identifications among the spaces 
$PSH(  M|  \mathbf{w}, U ) $ when $U$ varies among the Euler Structures compatible with  $\mathbf{fr}$.
So we can set 
$$PSH(M |   \mathbf{fr}) = PSH(  M|   U) $$
for any $U$   compatible with  $\mathbf{fr}$. 

Let $\mathbf{w}^0$ and  $\mathbf{w}^1$ be oriented embedded links. Let $ \tilde{\mathbf{w}}$ be an isotopy between $\mathbf{w}^0$ and  $\mathbf{w}^1$.
Let $\mathbf{fr}^0,\mathbf{fr}^1   \in \mathfrak{Fr}(M  )  $ be frames compatible with $\mathbf{w}^0$ and  $\mathbf{w}^1$  respectively. To each isotopy $\tilde{\mathbf{fr}}$  between   $\mathbf{fr}^0$ and $\mathbf{fr}^1$ compatible with $ \tilde{\mathbf{w}}$ it is associated a map
$$ Transfer_{ \tilde{\mathbf{w}},   \tilde{\mathbf{fr}}   }:  PSH(M |     \mathbf{w}^0  , \mathbf{fr}^0) \rightarrow PSH(M |    \mathbf{w}^1, \mathbf{fr}^1)  .$$

\begin{theorem} \label{coherent-theorem}
Let $\mathbf{w}$  be an embedded oriented link.
To each $\mathbf{fr} \in \mathfrak{Fr}(M  )$ compatible with  $\mathbf{w}$ it is associated canonically a $PSH$-class 
$$ [Z^{  \mathbf{w}, \mathbf{fr} }   ] \in  PSH(M |    \mathbf{w},  \mathbf{fr}).$$

 Let $ \tilde{\mathbf{w}}$ be an isotopy between oriented embedded links $\mathbf{w}^0$ and  $\mathbf{w}^1$.
For each isotopy $\tilde{\mathbf{fr}}$  between frames  $\mathbf{fr}^0$ and $\mathbf{fr}^1$ compatible with $ \tilde{\mathbf{w}}$ we have
$$Transfer_{ \tilde{\mathbf{w}},  \tilde{\mathbf{fr}}}( [Z^{  \mathbf{w}^0, \mathbf{fr}^0 }   ]  )   =   [Z^{  \mathbf{w}^1, \mathbf{fr}^1 }   ]$$
\end{theorem}

The proof of the theorem follows the same lines as in \cite{PSPCS1}  using the notion of   flat $PS$-cycle and flat approaching $PS$-cycle. However these definition are slightly more complicated.

After the suitable modifications of the definitions the proofs of  Proposition \ref{existence-flat-proposition} ,
Proposition \ref{unicity-flat-proposition},
Proposition \ref{existence-asymptotic-proposition},
Proposition \ref{unicity-asymptotic-proposition} are analogous to the corresponding one in \cite{PSPCS1}. We obtain a flat-approaching-$PS$-cycle  $Z^{\wr}$ canonically defined up to isotopy.


 In    \cite{PSPCS1} the Coherent Cycle was obtained directly from $Z^{\wr}$ using the trivial partition for any graph. 
  In the context of this paper we need to make some extra work to define $Z_{  \mathbf{w}, \mathbf{fr} }  $  from $Z^{\wr}$ since
 for the graphs considered in this paper does not necessarily exists the trivial partition,. See subsection \ref{proof-coherent-section}.



\subsection{Flat $PS$-cycles} \label{flat-section}




Fix a finite set $V^{\sharp}$.  Denote by $\mathfrak{G}(V^{\sharp})$ the set of decorated graphs  $G \in \mathfrak{G}$ equipped with an injective map 
$$V^{\sharp} \hookrightarrow V_{\vartriangle}(G) . $$

Set 
$$V^+(G)=  (V(G) \setminus  V^{\sharp}) \sqcup \bigsqcup_{v^{\sharp} \in V^{\sharp}} H_{v^{\sharp}}.$$

\begin{definition}
An admissible partition $\mathcal{P}$ of  $V^+(G)$ is a partition of $V^+(G)$ equipped with an array 
$$ ( \mathcal{P}^{v^{\sharp}} )_{v^{\sharp} \in V^{\sharp}}$$
 where each $\mathcal{P}^{v^{\sharp}}$ is a subset of $\mathcal{P}$ endowed with a cyclic order. We assume the following conditions:
\begin{itemize}
\item  the sets $ \{ \mathcal{P}^{v^{\sharp}} \}_{v^{\sharp} \in V^{\sharp}}$ are disjoint;
\item  for any $v^{\sharp} \in V^{\sharp}$ we have  
$$H_{v^{\sharp}} \subseteq  \bigsqcup_{P \in \mathcal{P}^{v^{\sharp}}} P$$ 
and the cyclic order of $H_{v^{\sharp}}$ is compatible with the cyclic order of $ \mathcal{P}^{v^{\sharp}}$;
\item any element of $V_{\vartriangle}$ belongs to a $P \in \mathcal{P}^{v^{\sharp}}$ for some $v^{\sharp} \in V^{\sharp}$;
\item for each $P \in \mathcal{P}^{v^{\sharp}}$ there exists at least an element of $V_{\vartriangle}$ or of $H_{v^{\sharp}}$ belonging to $P$.
\end{itemize}

We denote by $\mathfrak{G}^{part}$  the set of pairs $(G, \mathcal{P})$ where $G \in \mathfrak{G}(V^{\sharp})$ and $ \mathcal{P}$ is an admissible partition of $V^+(G)$.
$\mathfrak{G}^{part}_*$ denotes the set of triples $(G, m,\mathcal{P})$ with $(G,m)  \in  \mathfrak{G}_* $   and $(G, \mathcal{P}) \in \mathfrak{G}^{part}$.

\end{definition}

We set
$$\mathcal{P}^0 =    \mathcal{P} \setminus  \bigsqcup_{v^{\sharp} \in V^{\sharp}} \mathcal{P}^{v^{\sharp}}. $$
From the conditions above it follows that an element $P \in  \mathcal{P}$ belongs to $\mathcal{P}^0$ if and only if $P \subset V_0(G)$.

We endow  the set of admissible partitions of $V^+(G)$ with the following partial order:  $\mathcal{P'} \preccurlyeq \mathcal{P}$  if  
\begin{itemize}
\item   for  each $v_1,v_2 \in V^+(G)$
 $$v_1 \sim_{\mathcal{P}'} v_2 \Rightarrow v_1 \sim_{\mathcal{P}} v_2 ;$$
\item the cyclic order of $ \mathcal{P}^{\prime, v^{\sharp}}$ is compatible with the  the cyclic order of $ \mathcal{P}^{ v^{\sharp}}$.
\end{itemize}

Let $Z^{\flat, \circledcirc}$ be a flat annulus data in the sense of  \cite{PSPCS1}.
A \emph{ flat $PS$-cycle} $Z^{\flat}$, with annuls data  $Z^{\flat, \circledcirc}$, 
 consists in a collection of Borel-Moore chains 
\begin{equation} \label{flat-complex}
   Z^{\flat}= ( Z^{\flat}_{G,m, \mathcal{P}} )_{(G,m, \mathcal{P} ) \in \mathfrak{G}^{part}_*}
\end{equation} 
$$Z^{\flat}_{G,m, \mathcal{P}} \in C_{|H(G)| + |m|}((\R^3)^{H(G)} , \mathfrak{o}_{H(G)}  )^{\text{Aut}(G,m, \mathcal{P})}$$ 
such that :
\begin{itemize}
\item  the transversality condition is satisfied;
\item  the  forgetful compatibility condition is satisfied;
\item the cycle condition is satisfied;
\item the support condition is satisfied;
\item the cluster condition is satisfied;
\item  the translation condition is satisfied.
\end{itemize}

The conditions appearing in the definition are extensions of the conditions appearing in \cite{PSPCS1}, which are explained below.

\subsubsection{Transversality Condition}

For each $h \in H(G)$ denote by $v_h$ the vertex to which $h$ is attached. Set 
$$E^{in}(G, \mathcal{P})  = \sqcup_j \{ \{h,h' \}  \in E^{in}(G) | v_{h} \sim_{\mathcal{P}} v_{h'} \} ,   \quad E(G, \mathcal{P}) = E^{in}(G, \mathcal{P}) \sqcup E^{ex}(G).$$

We say that the transversality condition is satisfied for $(G,m, \mathcal{P})$ if
$Z^{\flat}_{G,m, \mathcal{P}}$ is transversal  to   $ \cap_{e \in E'} \pi_e^{-1}(Diag)$ for each subset $E' \subset E^{in }(G , \mathcal{P})  \setminus E^l$;

\subsubsection{Forgetful Compatibility Condition}



   





As in \cite{PSPCS1}, to state the forgetful compatibility condition we need to define the multi-loop space associated to a partitioned graph. Differently to the case considered in \cite{PSPCS1}, here this object depends also on the partition. Namely, instead to use the set of vertices of the graph we need to introduce the set of vertices of a partitioned graph.

To $(G, \mathcal{P}) \in \mathfrak{G}^{part}$ the associated set  of vertices $V(\mathcal{P})$ is defined from $V(G)$   
replacing each $v^{\sharp}$ with the set $\mathcal{P}^{v^{\sharp}}$:
$$V(\mathcal{P})= (V(G) \setminus V^{\sharp}) \sqcup \bigsqcup_{v^{\sharp} \in V^{\sharp}} V(\mathcal{P},  v^{\sharp} ) $$
where 
$V(\mathcal{P},  v^{\sharp} )$ is a set in bijection with $ \mathcal{P}^{v^{\sharp}}$
$$V(\mathcal{P},  v^{\sharp} ) \leftrightarrow \mathcal{P}^{v^{\sharp}}.$$

Denote by
$$w^{\flat} : \R \rightarrow \R^3$$
the map
$$t \mapsto (t,0,0) .$$



\begin{definition}
For $(G, \mathcal{P}) \in \mathfrak{G}^{part}$, the multi-loop space  $\underline{\digamma}_{G,\mathcal{P}}$  is defined as the set of arrays $({w}_v , t_v)_{v \in V(\mathcal{P})}  $  where  
\begin{itemize}
    \item ${w}_v : \mathbb{S}^1 \rightarrow \R^3$,  $t_v : H_v  \hookrightarrow \mathbb{S}^1$  when $v \in  V \setminus V^{\sharp}$;
\item  ${w}_v : \R \rightarrow \R^3$, $t_v : H_v  \hookrightarrow \R$ when $v \in V(\mathcal{P},  v^{\sharp} ) $ for any $v^{\sharp} \in V^{\sharp}$. 
\end{itemize}
 We require that each $t_v$ is compatible with the cyclic order of $H_v$ .

\end{definition}



With these notations the  forgetful compatibility condition is formulated as in \cite{PSPCS1}.

We say that  (\ref{flat-complex}) is forgetful compatible if there exists a collections of Borel-Moore chains 
\begin{equation} \label{flat-complex-loop}
    (\underline{Z}^{\flat}_{G,m, \mathcal{P}} )_{(G,m, \mathcal{P}) \in \mathfrak{G}^{part}_*}
\end{equation}
$$\underline{Z}^{\flat}_{G,m, \mathcal{P}} \in C_{|H(G)| + |m|}(\underline{\digamma}_{G,\mathcal{P}} , \mathfrak{o}_{H(G)}  )^{\text{Aut}(G,m, \mathcal{P})}$$
such that
\begin{equation} \label{chains-ev}
ev_*( \underline{Z}^{\flat}_{G,m, \mathcal{P}}    )={Z}^{\flat}_{G,m, \mathcal{P}}
\end{equation}

For each external edge $e \in E^{ex}(G)$ we have (the decorated graph $G_e$ is defined as in \cite{PSPCS1}):
\begin{itemize}
\item If $G$ is stable after removing $e $  
    \begin{equation} \label{stable-forget-condition-flat}    
     \underline{Z}^{\flat}_{G,m, \mathcal{P}} = \underline{Z}^{\flat}_{forget_e(G.m), \mathcal{P}} \times_{\textbf{forget}_e} \underline{\digamma}_{G,\mathcal{P}}   .
     \end{equation}
      \item in the case $G$ is unstable after removing $e $ and $e$ is removable  we require
 \begin{equation} \label{flat-removible-forget-condition}
 \underline{Z}^{\flat}_{G,m, \mathcal{P}} = \underline{Z}^{\flat}_{forget_e(G,m), \mathcal{P}} \times \underline{Z}_{G_e}^{\flat, \circledcirc} .
 \end{equation}     
  when $G_e \in Ann0$, and
 \begin{equation} \label{flat-removible-forget-condition2}
 \underline{Z}^{\flat}_{G,m} = \underline{Z}^{\flat}_{forget_e(G,m), \mathcal{P}} \times \underline{Z}^{\flat}_{Disk0}  
 \end{equation}  
 when $G_e= Disk0$.

     \item In the case  $e$ is not removable we require that $ \underline{Z}^{\flat}_{G,m, \mathcal{P}}$ is close in the $C^0$-topology to $ \underline{Z}^{\flat}_{forget_e(G.m), \mathcal{P}} \times \underline{Z}_{Disk0}^{\flat, \circledcirc}$:
 \begin{equation} \label{unstable-forget-condition-flat}
 \underline{Z}^{\flat}_{G,m, \mathcal{P}} \approx \underline{Z}^{\flat}_{forget_e(G.m), \mathcal{P}} \times \underline{Z}_{Disk0}^{\flat, \circledcirc} .
 \end{equation}
  \end{itemize}


As in \cite{PSPCS1}, in the notation above we consider any partition  $\mathcal{P}$ of  $V^+(G)$ also as a partition of $V^+(forget_e(G))$, using the obvious correspondence (note that for any $e \in E^{ex}(G)$ we have $V^+(forget_e(G)) \subset V^+(G)$).

\subsubsection{Cycle Condition} \label{cycle-condition-section}

We say that the cycle condition is satisfied for $(G, m, \mathcal{P}) \in \mathfrak{G}_*^{part}$ if 
\begin{equation} \label{cycle-condition-loop}
 \partial  \underline{Z}_{G, m,\mathcal{P}}^{\flat}  =   \sum_i (-1)^i \underline{Z}_{G,  \partial_i m , \mathcal{P}}^{\flat} + \sum_{e'}  \delta_{e'} \underline{Z}_{G', m', \mathcal{P}'}^{\flat}.
 \end{equation}





\subsubsection{Support Condition} \label{support-condition-definition}

Let $(G, \mathcal{P}) \in \mathfrak{G}^{part}$. For each positive real number $\rho_0 >0$  define   $ \digamma_{G, { \mathcal{P} }   }^{< \rho_0}$ as the set of elements
 $( w_v , t_v)_{v \in V(\mathcal{P})} \in \underline{\digamma}_{G, { \mathcal{P} }   }$  with the following properties:
\begin{itemize}
 \item 
$ Diameter ( \sqcup_{v \in V_c} w_v(\mathbb{S}^1) ) < \rho_0$  for any $c \in  Comp_0$  ;
\item
for any $v \in V(\mathcal{P},  v^{\sharp} ) $  there exist two finite arrays 
\begin{equation} \label{covering-support-flat}
(I_j)_{j \in J^0} , (I_j)_{j \in J^1}
\end{equation}  
of  disjoint open connected subsets of $\R$, with the following properties:
\begin{itemize}
\item for any $j \in J^0$ we have
 (up to parametrizations) $w_v= w^{\flat}$ on $I_j$ ; 
\item
let $P$ be the element of $ \mathcal{P}^{v^{\sharp}}$ corresponding to $v$. There exists a decomposition in disjoint sets $ P \cap V_{\vartriangle}= \sqcup_{j \in J_1}  V_{\vartriangle,j}$ with 
$ I_j  \cap Im(t_v)   \neq \emptyset$ if $V_{\vartriangle,j} = \emptyset$ and such that
$$Diam(w_v(I_j) \sqcup \bigsqcup_{v \in V_{\vartriangle,j} }w_v(\mathbb{S}^1)) < \rho_0$$
for any $j \in J^1$;
\item $\R = \sqcup_{j \in J^0 \sqcup J^1} I_j. $
\end{itemize}
\end{itemize}



\begin{definition}
We say that the support condition is satisfied for $(G,m, \mathcal{P})$ if there exists  a real positive number $\rho^{\dagger}$ such that 
$$ Supp(  \underline{Z}_{G,m, \mathcal{P}}^{\flat} )  \subset  \underline{\digamma}_{G, \mathcal{P}}^{<\rho^{\dagger}}.$$
\end{definition}

For $z \in \R^3$ and $\rho>0$ denote by $D_{\rho, z}$ the disk of radius $\rho$ with center $z$.



For  $\mathbf{z} = (z_v)_{v \in V^+(G)} \in  (\R^3)^{V^+(G)}$
denote by  $\underline{D}_{\rho, \mathbf{z}}^{\flat}$
the set of elements $({w}_v,t_v)_v  \in  \underline{\digamma}_{G,\mathcal{P}}$
such that


\begin{itemize}
\item  for any $v \in V(G)  \setminus V^{\sharp}$ we have
 $$  w_v(S^1) \subset  D_{\rho,z_v} ;$$
\item for any
 $v \in V(\mathcal{P},  v^{\sharp} ) $
 there exist two arrays $(I_j)_{j \in J^0}, (I_j)_{j \in J^1}$  as in (\ref{covering-support-flat})
 such that 
 $$  w_v(I_j) \subset  D_{\rho,z_v} \quad \forall j \in J^1.$$ 
 \end{itemize}


 

\subsubsection{Cluster Condition}

Let $Diag_{\mathcal{P}}^{\flat}$  be the set of $\mathbf{z} = (z_v)_{v \in V^+(G)} \in  (\R^3)^{V^+(G)}$ such that
\begin{itemize}
\item  $z_{v_1} =z_{v_2}$ when $v_1 \sim_{\mathcal{P}} v_2$;
\item $z_v \in w^{\flat}$ for any $v \in \sqcup_{P \in \mathcal{P}^{v^{\sharp}}} P$. 
\end{itemize}
We sometime consider  $\mathbf{z} \in Diag_{\mathcal{P}}^{\flat}$  as an array $ (z_{v})_{v \in V(\mathcal{P}) } \in ((\R)^3)^{V(\mathcal{P})}$ or
$ (z_{P})_{P \in \mathcal{P} } \in ((\R)^3)^{\mathcal{P}}$.


We have
$$\mathcal{P}' \prec \mathcal{P} \Longleftrightarrow Diag_{  \mathcal{P}  }^{\flat} \subset Diag_{\mathcal{P}'}^{\flat}  .$$

\begin{definition}

Let $\mathcal{P} $ be a partition of $G \in \mathfrak{G}$. For  $  \mathcal{P}^{\triangleright} $ with  $\mathcal{P}  \prec  \mathcal{P}^{\triangleright} $ and $\eta>0$,  
$ Cut_{\mathcal{P}^{\triangleright} , \eta} ( Diag_{\mathcal{P} }^{\flat})$ is defined as the set of $\mathbf{z}  \in Diag_{\mathcal{P} }^{\flat}$ such that
\begin{itemize}
\item for any $v^{\sharp} \in V^{\sharp}$ and $P^{\triangleright}  \in \mathcal{P}^{\triangleright, v^{\sharp} }$, the order  on $ w^{\flat}$ of the array of points $(z_P)_{P \subset P^{\triangleright}}$ is compatible with the cyclic order of $ \mathcal{P}^{v^{\sharp}}$;
 \item  $d ( \mathbf{z}, Diag_{\mathcal{P}^{\triangleright  \prime}}^{\flat}  ) > \eta  $ for each $\mathcal{P}^{\triangleright \prime}$ with $\mathcal{P}  \prec \mathcal{P}^{\triangleright\prime} \preccurlyeq \mathcal{P}^{\triangleright} $.
\end{itemize}

\end{definition}

Observe that for any $\mathcal{P} \prec \mathcal{P}^{\triangleright} $ and $\rho>0$ there exists  $\eta$ big enough such that
\begin{equation} \label{disjoint-disks-nuclear}
  \mathbf{z} \in  Cut_{\mathcal{P}^{\triangleright} , \eta} ( Diag_{\mathcal{P} }^{\flat}),  v_1 \nsim_{\mathcal{P}^{\triangleright}} v_2 \Longrightarrow  {D}_{\rho, z_{ v_1}} \cap {D}_{\rho, z_{ v_2}} = \emptyset.
 \end{equation}


 Denote by   $\underline{D}_{\rho, \mathbf{z}}^{\flat} (\mathcal{P}^{\triangleright})$ the set of elements $({w}_v,t_v)_v \in \underline{D}_{\rho, \mathbf{z}}^{\flat}$ such that for each $v^{\sharp} \in V^{\sharp}$, $P \in V(\mathcal{P},  v^{\sharp} ) $ and ${P}^{\triangleright} \in \mathcal{P}^{\triangleright}$ with $P \subseteq  {P}^{\triangleright}$ the following happen : if $H_{v^{\sharp}} \cap P \neq H_{v^{\sharp}} \cap {P}^{\triangleright}$ then the ordering of the array  $t_P$ is compatible with the ordering of $H_{v^{\sharp}} \cap P$ obtained from the cyclic ordering of $H_{v^{\sharp}}$.

 
Property (\ref{disjoint-disks-nuclear}) tells us that, for $\eta$ big enough, we have an inclusion of sets
\begin{equation} \label{inclusion-disks-partion-loop-flat}
     \underline{D}_{\rho, \mathbf{z}}^{\flat} (\mathcal{P}^{\triangleright}) \hookrightarrow  \underline{ \digamma}_{G, \mathcal{P}^{\triangleright} } 
\end{equation}
which is compatible with the cycle condition in the sense that  on   the chains of $ \underline{D}_{\rho, \mathbf{z}}^{\flat}$ we have $\delta_e \neq 0$ for  any $e \notin  E(G, \mathcal{P})$.

\begin{definition}
For $(G,m, \mathcal{P}) \in \mathfrak{G}_*^{part} $, $\rho, \eta> 0$, 
we say that (\ref{flat-complex}) satisfies the  cluster condition for $(G,m, \mathcal{P}, \rho, \eta)$ if 
\begin{equation} \label{matching-nuclear}
\underline{Z}^{\flat}_{G,m, \mathcal{P}} \cap \underline{D}_{\rho ,  \mathbf{z}}^{\flat}  (\mathcal{P}^{\triangleright}) = \underline{Z}^{\flat}_{G,m, \mathcal{P}^{\triangleright}}  \cap \underline{D}_{\rho , \mathbf{z}}^{\flat} (\mathcal{P}^{\triangleright}),
  \end{equation}
for each $\mathcal{P}^{\triangleright} $ with $\mathcal{P} \prec \mathcal{P}^{\triangleright} $ and $ \mathbf{z} \in  Cut_{\mathcal{P}^{\triangleright} , \eta} ( Diag_{\mathcal{P} }^{\flat}) $.

We say that the cluster condition is satisfied for
$(G,m, \mathcal{P} , \rho)$ if  there exists $\eta$ big enough such that the cluster condition is satisfied for
$(G,m,  \mathcal{P}, \rho, \eta)$ .
\end{definition}

The cluster condition may be considered as a boundary condition at infinity for  $\underline{Z}^{\flat}_{G,m, \mathcal{P}}$ in the sense that it defines the chain outside a compact set in terms of the partitions  $\mathcal{P}' \prec \mathcal{P}$,

\subsubsection{Translation Condition}

The group $Transl_{\mathcal{P}}:=(\R^3)^{\mathcal{P}^0} \times \prod_{v^{\sharp}} (\R)^{\mathcal{P}^{v^{\sharp}}}$  acts on $\underline{\digamma}_{G,\mathcal{P}} $ by translation .   
This action yields an action of $Transl_{\mathcal{P}}$ on $ C_{|H(G)| + |m|}(\underline{\digamma}_{G,\mathcal{P}} , \mathfrak{o}_{H(G)}  )^{\text{Aut}(G,m, \mathcal{P})}$.

We say that the Translation Condition is satisfied for $(G,m, \mathcal{P})$ if $\underline{Z}^{\flat}_{G,m, \mathcal{P}}$ is invariant by $Transl_{\mathcal{P}}$ (up to triangulations and reparametrizations).


\subsubsection{Isotopies}

Isotopies of flat $PS$-cycles are defined adapting the definition of isotopies of $PS$-cycles.


Let $\tilde{Z}^{\flat, \circledcirc}$ be an isotopy between annulus data ${Z}^{\flat, \circledcirc,0}$ and  ${Z}^{\flat, \circledcirc,1}$. Take $Z^{\flat,0} \in \mathcal{Z}(M|  {Z}^{\flat, \circledcirc,0})$ and $Z^{\flat,1} \in \mathcal{Z}(M|  {Z}^{\flat, \circledcirc,1})$. An isotopy  
of flat $PS$-cycles
  between $Z^{\flat,0}$ and $Z^{\flat,1}$  consists in a collection of chains
\begin{equation} \label{collection-chains-isotopy-flat}
\tilde{Z}^{\flat}=  ( \tilde{Z}_{G,m, \mathcal{P}}^{\flat} )_{G,m, \mathcal{P}} ,
\end{equation}
with
$$\tilde{Z}_{G,m, \mathcal{P}}^{\flat} \in C_{|H(G)| + |m|+1} ( (\R^3)^{H(G)} \times \R  , \mathfrak{o}_{H(G)})^{\text{Aut}(G,m, \mathcal{P})}  .$$
We require that
\begin{itemize}
\item there exists  $\lambda^{\dagger} >0$  big enough such that
$$ \tilde{Z}^{\flat, <-\lambda^{\dagger}} =  Z^{\flat,0}  \times  \R_{<-\lambda^{\dagger}}  , \quad \tilde{Z}^{\flat, >\lambda^{\dagger}} =  Z^{\flat,1}  \times \R_{>\lambda^{\dagger}}   ;$$

 \item 
the transversality condition is satisfied;


\item the forgetful compatibility condition is satisfied.

\item the cycle condition is satisfied;

\item the translation condition is satisfied;

\item the support condition is satisfied;

\item the cluster condition is satisfied.


\end{itemize}

We denote by $\tilde{\mathcal{Z}}^{\flat}(M|  \tilde{Z}^{\flat, \circledcirc})$ the space of isotopies of flat $PS$-cycles. 



The transversality condition, forgetful compatibility condition,  cycle condition, translation condition and cluster condition are defined analogously to section  \ref{isotopies-section}.


In the cluster condition we consider the one parameter version of the condition (\ref{matching-nuclear}), namely we require that
\begin{equation} \label{matching-nuclear-iso}
\underline{Z}^{\flat}_{G,m, \mathcal{P}} \cap(  \underline{D}_{\rho ,  \mathbf{z}'} \times \R ) = \underline{Z}^{\flat}_{G,m, \mathcal{P'}}  \cap (  {\underline{D}}_{\rho , \mathbf{z}'}    \times \R ),
  \end{equation}
for each $ \mathbf{z} \in  Cut_{\mathcal{P}^{\triangleright} , \eta} ( Diag_{\mathcal{P} }^{\flat}) $.


\subsubsection{Existence and Unicity}

Denote by $G^{\heartsuit}$ the decorated graph defined by
$$ H(G^{\heartsuit})= \emptyset, \chi_{\vartriangle}(G^{\heartsuit})=0, d_{\vartriangle}(G^{\heartsuit})=0 , Comp_0(G^{\heartsuit})= \emptyset,  V_{\vartriangle}(G^{\heartsuit})= V^{\sharp} .$$

The following Propositions are extensions of the Propositions appearing in \cite{PSPCS1}, and they are proved in the same way.

\begin{proposition} \label{existence-flat-proposition}
There exists a flat $PS$-cycle $ {Z}^{\flat}$ such that 
    $$  {\underline{Z}}^{\flat}_{G^{\heartsuit}} = \mathbf{w}^{\flat}$$
    $$H(G)= \emptyset,  \underline{Z}_G^{\flat} \neq 0 \Rightarrow G = G^{\heartsuit}.$$
    $$  \chi_{\vartriangle}(G) \neq 0 \text{   or   } d(G) \neq 0 \Rightarrow     \underline{Z}_{G,m}^{\flat}.$$
    \end{proposition}

\begin{proposition} \label{unicity-flat-proposition}
Let $Z^{\flat,0}$ and $Z^{\flat,1}$ be  flat $PS$-cycles satisfying the conditions   of  Proposition \ref{existence-flat-proposition}. There exists an isotopy of flat $PS$-cycles between  $Z^{\flat,0}$ are $Z^{\flat,1}$.
 \end{proposition}

\subsection{Flat-Approaching $PS$-cycles}

Fix an  an embedded oriented link
$$\mathbf{w} = ( w_v^{\sharp} )_{v \in V^{\sharp}} .$$

Let $\mathbf{fr} \in \mathfrak{Fr}(M,   \mathbf{w}  )$. Let   $Z^{\flat, \circledcirc}$ be a flat annulus data.



Fix a flat-annulus-data $ Z^{\flat, \circledcirc}$, a flat-$PS$-cycle $Z^{\flat} \in \mathcal{Z}^{\flat}( Z^{\flat, \circledcirc} )$ and a  flat-approaching annulus data $Z^{\wr, \circledcirc}$ converging to $(Z^{\flat, \circledcirc}, \mathbf{fr} ) $ in the sense of \cite{PSPCS1}.

A flat-approaching $PS$-cycles  $\mathcal{Z}^{\wr}$  converging to $(Z^{\flat} , \mathbf{fr})$ with annulus data $Z^{\wr, \circledcirc}$ is defined by a  collection of chains
\begin{equation} \label{asimptoticaly-flat-complex}
    ( Z^{\wr , (n)}_{G,m, \mathcal{P}} )_{G,m, \mathcal{P}}
\end{equation}
$$Z^{\wr , (n)}_{G,m, \mathcal{P}} \in C_{|H(G)| + |m| }(M^{H(G)} , \mathfrak{o}_{H(G)}  )^{\text{Aut}(G,m, \mathcal{P})}$$ 
such that :
\begin{itemize}
\item  
 transversality condition is satisfied;
\item   forgetful compatibility condition is satisfied;
\item support  condition is satisfied;
    \item cluster condition  is satisfied;
    \item convergence  condition is satisfied .
\end{itemize}

Below we explain the meaning of  the conditions of the definition which are an extension of the conditions considerate in \cite{PSPCS1}.

\subsubsection{Transversality condition}

We say that the transversality condition is satisfied if $Z^{\wr , (n)}_{G,m, \mathcal{P}}$ is transversal  to   $ \cap_{e \in E'} \pi_e^{-1}(Diag)$ for each subset $E' \subset E^{in }(G , \mathcal{P})  \setminus E^l$.

\subsubsection{Forgetful Compatibility} \label{forgetful-approching-section}

For $(G, \mathcal{P}) \in \mathfrak{G}^{part}$, the multi-loop space  $\underline{M}_{G,\mathcal{P}}$ associated to $(G, \mathcal{P})$  is defined as the set of arrays $( {w}_v , t_v)_{v \in V(\mathcal{P})}  $  where   
  $${w}_v : \mathbb{S}^1 \rightarrow M,  t_v : H_v  \hookrightarrow \mathbb{S}^1, \quad \forall v \in V(\mathcal{P}) .$$ 
  We require that each $t_v$ is compatible with the cyclic order of $H_v$;




We say that there (\ref{flat-complex}) is forgetful compatible if there exists a collections of  chains 
\begin{equation} \label{flat-complex-loop-approching}
    (\underline{Z}^{\wr , (n)}_{G,m, \mathcal{P}} )_{G,m, \mathcal{P}}
\end{equation}
$$\underline{Z}^{\wr , (n)}_{G,m, \mathcal{P}} \in C_{|H(G)| + |m| }(\underline{M}_{G,\mathcal{P}} , \mathfrak{o}_{H(G)}  )^{\text{Aut}(G,m)}$$
such that
$$ev_*( \underline{Z}^{\wr , (n)}_{G,m, \mathcal{P}}    )=Z^{\wr , (n)}_{G,m, \mathcal{P}}.$$
For each external edge $e \in E^{ex}(G)$ we require the following conditions:
\begin{itemize}
\item if $G$ is stable after removing $e $  we require
    \begin{equation} \label{stable-forget-condition-approching}    
     \underline{Z}_{G,m, \mathcal{P}}^{\wr , (n)} = \underline{Z}_{G',m'}^{\wr , (n)} \times_{\textbf{forget}_e} \underline{M}_{G,\mathcal{P}}  .
     \end{equation}
      \item in the case $G$ is unstable after removing $e $ and $e$ is removable  we require
 \begin{equation} \label{removible-forget-condition-approching}
 \underline{Z}_{G,m}^{\wr , (n)} = \underline{Z}_{G',m'}^{\wr , (n)} \times \underline{Z}_{G_e}^{\circledcirc} .
 \end{equation}     
  when $G_e \in Ann0$, and
 \begin{equation} \label{removible-forget-condition-approching2}
 \underline{Z}_{G,m}^{\wr , (n)} = \underline{Z}_{G',m'}^{\wr , (n)} \times \underline{Z}_{Disk0}  
 \end{equation}  
 when $G_e= Disk0$. 
 
     \item In the case  $e$ is not removible we require that $ \underline{Z}_{G,m}$ is close in the $C^1$-topology to $ \underline{Z}_{G',m'} \times \underline{Z}_{Disk0}^{}$. We write
 \begin{equation} \label{unstable-forget-condition-approching}
 \underline{Z}_{G,m}^{\wr , (n)} \approx \underline{Z}_{G',m'}^{\wr , (n)} \times \underline{Z}_{Disk0}^{} .
 \end{equation}
  \end{itemize}



 .


\subsubsection{Support Condition}   \label{support-condition-definition-approching}

Let $(G, \mathcal{P}) \in \mathfrak{G}^{part}$. For $\rho >0$ define the set  $\underline{M}_{G, \mathcal{P}}^{< \rho}$ as the set of elements
$( w_v , t_v)_v \in \underline{M}_{G, \mathcal{P}}$ with the following properties:
\begin{itemize}
 \item 
$ Diameter ( \sqcup_{v \in V_c} w_v(\mathbb{S}^1) ) < \rho$  for any $c \in  Comp_0(G)$  ;
\item
for any $v \in V(\mathcal{P},  v^{\sharp} ) $ there exist two finite arrays  
\begin{equation} \label{covering-support-approching}
(I_j)_{j \in J^0} , (I_j)_{j \in J^1}
\end{equation}  
 of  disjoint open connected subsets of $S^1$  with the following properties:
\begin{itemize}
\item  for any $j \in J^0$ we have  
 (up to parametrizations) $w_v= w_{v^{\sharp}}^{\sharp}$ on $I_j$ ;
\item 
let $P$ be the element of $ \mathcal{P}^{v^{\sharp}}$ corresponding to $v$. There exists a decomposition in disjoint sets $ P \cap V_{\vartriangle}= \sqcup_{j \in J_1}  V_{\vartriangle,j}$ with 
$ I_j  \cap Im(t_v)   \neq \emptyset$ if $V_{\vartriangle,j} = \emptyset$ and such that
$$Diam(w_v(I_j) \sqcup \bigsqcup_{v \in V_{\vartriangle,j} }w_v(\mathbb{S}^1)) < \rho_0$$
for any $j \in J^1$;
\item $S^1 = \sqcup_{j \in J^0 \sqcup J^1} I_j. $
\end{itemize}
\end{itemize}

\begin{definition} \label{support-approching}
We say that the support condition is satisfied for $(G,m, \mathcal{P})$ if there exists  a real positive number $\rho^{\dagger}$ such that 
$$ Supp(  \underline{Z}_{G,m, \mathcal{P}}^{\wr, (n)} )  \subset  \underline{M}_{G, \mathcal{P}}^{<\rho^{\dagger}/n}.$$
\end{definition}

For $z \in M$ and $\rho>0$ denote by $D_{\rho, z}$ the disk of radius $\rho$ with center $z$.



For  $\mathbf{z} = (z_v)_{v \in V^+(G)} \in  M^{V^+(G)}$
denote by  $\underline{D}_{\rho, \mathbf{z}}$
the set of elements $({w}_v,t_v)_v  \in  \underline{M}_{G,\mathcal{P}}$
such that


\begin{itemize}
\item  for any $v \in V(G)  \setminus V^{\sharp}$ we have
 $$  w_v(S^1) \subset  D_{\rho,z_v} ;$$
\item for any
 $v \in V(\mathcal{P},  v^{\sharp} ) $
 there exist two arrays $(I_j)_{j \in J^0}, (I_j)_{j \in J^1}$  as in (\ref{covering-support-approching})
 such that 
 $$  w_v(I_j) \subset  D_{\rho,z_v} \quad \forall j \in J^1.$$ 
 \end{itemize}






\subsubsection{Cluster Condition} \label{cluster-condition-iso-section}

Let $Diag_{\mathcal{P}}$ be  the set of $\mathbf{z} = (z_v)_{v \in V^+(G)} \in M^{V^+(G)}$ such that 
\begin{itemize}
\item
$z_{v_1} =z_{v_2}$ when $v_1 \sim_{\mathcal{P}} v_2$ ;
\item $z_v \in w_{v^{\sharp}}^{\sharp}$ when $v \in \sqcup_{P \in \mathcal{P}^{v^{\sharp}}} P$.
\end{itemize}
An element $\mathbf{z}  \in Diag_{\mathcal{P}}$  may be also considered as an array  $ (z_{v})_{v \in V(\mathcal{P}) } $.

Observe that
$$\mathcal{P}' \prec \mathcal{P} \Longleftrightarrow Diag_{  \mathcal{P}  } \subset Diag_{\mathcal{P}'} .$$

\begin{definition}

Let $\mathcal{P} $ be a partition of $G \in \mathfrak{G}$. For  $  \mathcal{P}^{\triangleright} $ with  $\mathcal{P}  \prec  \mathcal{P}^{\triangleright} $ and $\eta>0$  define 
$ Cut_{\mathcal{P}^{\triangleright} , \eta} ( Diag_{\mathcal{P} })$ as the set of $\mathbf{z}  \in Diag_{\mathcal{P} }$ such that
\begin{itemize}

\item for any $v^{\sharp} \in V^{\sharp}$ and $P^{\triangleright}  \in \mathcal{P}^{\triangleright, v^{\sharp} }$, the cyclic order  on $ w_{ v^{\sharp}}^{\sharp}$ of the array of points $(z_P)_{P \subset P^{\triangleright}}$ is compatible with the cyclic order of $ \mathcal{P}^{v^{\sharp}}$.

 \item  $d ( \mathbf{z}, Diag_{\mathcal{P}^{\triangleright, \prime}}  ) > \eta  $ for each $\mathcal{P}^{\triangleright, \prime}$ with $\mathcal{P}  \prec \mathcal{P}^{\triangleright, \prime} \preccurlyeq \mathcal{P}^{\triangleright} $.
\end{itemize}

\end{definition}


For $\mathcal{P} \prec \mathcal{P}^{\triangleright} $,
for each $\rho>0$ there exists  $\eta$ big enough such that
\begin{equation} \label{disjoint-disks-approching}
  \mathbf{z} \in  Cut_{\mathcal{P}^{\triangleright} , \eta} ( Diag_{\mathcal{P} }),  v_1 \nsim_{\mathcal{P}^{\triangleright}} v_2 \Longrightarrow  {D}_{\rho, z_{ v_1}} \cap {D}_{\rho, z_{ v_2}} = \emptyset.
 \end{equation}


 Denote by   $\underline{D}_{\rho, \mathbf{z}}(\mathcal{P}^{\triangleright})$ the set of elements $({w}_v,t_v)_v \in \underline{D}_{\rho, \mathbf{z}}$ such that the for each $v^{\sharp} \in V^{\sharp}$, $P \in V(\mathcal{P},  v^{\sharp} ) $ and ${P}^{\triangleright} \in \mathcal{P}^{\triangleright}$ with $P \subseteq  {P}^{\triangleright}$ the following happen : if $H_{v^{\sharp}} \cap P \neq H_{v^{\sharp}} \cap {P}^{\triangleright}$ then the array  $t_P$ is compatible with the ordering of $H_{v^{\sharp}} \cap P$ obtained from the cyclic ordering of $H_{v^{\sharp}}$.


Property (\ref{disjoint-disks-approching}) tells us that, for $\eta $ big enough, we have an inclusion of sets
\begin{equation} \label{inclusion-disks-partion-loop-approching}
     \underline{D}_{\rho, \mathbf{z}} (\mathcal{P}^{\triangleright}) \hookrightarrow   \underline{M}_{G, \mathcal{P}^{\triangleright} } 
\end{equation}
which is compatible with the cycle condition in the sense that $\delta_e \neq 0$ on   the chains of $ \underline{D}_{\rho, \mathbf{z}}$ for  $e \notin  E(G, \mathcal{P})$.

Let $(G,m,\mathcal{P}) \in \mathfrak{G}_*^{part} $, $\rho, \eta> 0$.
We say that the cluster condition is satisfied for $(G,m, \mathcal{P}, \rho, \eta)$ if 
\begin{equation} \label{matching-asintontically}
\underline{Z}^{\wr, (n)}_{G,m, \mathcal{P}} \cap \underline{D}_{\rho/n ,  \mathbf{z}}(\mathcal{P}^{\triangleright})  = \underline{Z}^{\wr, (n)}_{G,m, \mathcal{P}^{\triangleright}}  \cap \underline{D}_{\rho/n , \mathbf{z}}  \cap \underline{D}_{\rho/n ,  \mathbf{z}}(\mathcal{P}^{\triangleright})  ,
  \end{equation}
for each $\mathcal{P}^{\triangleright} $ with $\mathcal{P} \prec \mathcal{P}^{\triangleright} $ and $ \mathbf{z} \in  Cut_{\mathcal{P}^{\triangleright} , \eta/n} ( Diag_{\mathcal{P} }) $.

We say that  the cluster condition is satisfied for
$(G,m, \mathcal{P} , \rho)$ if there exists $\eta$ big enough such that the cluster condition is satisfied for
$(G,m,  \mathcal{P}, \rho, \eta)$ .





\subsubsection{Convergence Condition} \label{convergence-condition-asi-section}


Denote by $Dil_n$  the dilation of ratio $n$ of the tangent space on $\mathbf{z} $.

We say that the convergence condition is satisfied for $(G,m,\mathcal{P})$ if    for each $\rho>0$
  \begin{equation} \label{convergence} 
 Dil_n( \text{exp}^{-1}(\underline{D}_{\rho/n, \mathbf{z} } \cap \underline{Z}_{G,m,\mathcal{P}}^{\wr, (n)} ) ) \rightarrow \underline{D}_{\rho, \mathbf{0}}^{\flat} \cap  {\underline{Z}}^{\flat}_{G,m, \mathcal{P}} 
   \end{equation}
 uniformly for $\mathbf{z}  \in  Diag_{\mathcal{P}}$.  Here we have used $\mathbf{fr} $ to identify $T_zM$ with $\R^3$ for each $z \in M$.

We say that $Z$ satisfies the convergence condition if the convergence condition is satisfied for each $(G,m,\mathcal{P})$.

\subsubsection{Isotopies}

Let $\tilde{Z}^{\flat, \circledcirc}$ be an isotopy between the  flat annulus data ${Z}^{\flat, \circledcirc,0}$ and  ${Z}^{\flat, \circledcirc,1}$. Let $\tilde{Z}^{\flat} \in \tilde{\mathcal{Z}}(M|  \tilde{Z}^{\flat, \circledcirc})$ be an isotopy  between   the flat $PS$-cycles $Z^{\flat,0} \in \mathcal{Z}(M|  {Z}^{\flat, \circledcirc,0}),$ and $Z^{\flat,1} \in \mathcal{Z}(M|  {Z}^{\flat, \circledcirc,1})$.

Let $\mathbf{w}^0$ and  $\mathbf{w}^1$ be oriented embedded links. Let 
$$ \tilde{\mathbf{w}} = ( \mathbf{w}^{\lambda} )_{\lambda \in [0,1]}$$ 
be an isotopy between $\mathbf{w}^0$ and  $\mathbf{w}^1$.
Let $\mathbf{fr}^0$ and $\mathbf{fr}^1$ be frames of $M$ compatible with  $\mathbf{w}^0$ and $\mathbf{w}^1$ respectively. Let $\tilde{\mathbf{fr}}$ be an isotopy  between  $\mathbf{fr}^0$ and $\mathbf{fr}^1$ compatible with $\tilde{\mathbf{w}}$.

Let  ${Z}^{\wr, \circledcirc,0}$ and  ${Z}^{\wr , \circledcirc,1}$ be asymptotic flat annulus data converging to $({Z}^{\flat, \circledcirc,0}, \mathbf{fr}^0), ({Z}^{\flat, \circledcirc,1},\mathbf{fr}^1)$ respectively.

Let $\tilde{Z}^{\wr, \circledcirc}$ be an isotopy between  ${Z}^{\wr, \circledcirc,0}$ and  ${Z}^{\wr , \circledcirc,1}$ converging to $(   \tilde{Z}^{\flat, \circledcirc}, \tilde{\mathbf{fr}})$ . Let $Z^{\wr,0} \in \mathcal{Z}(M|  {Z}^{\circledcirc,0})$ and $Z^{\wr,1} \in \mathcal{Z}(M|  {Z}^{\circledcirc,1})$ converging  to $(Z^{\flat,0 }, \mathbf{fr}^0), (Z^{\flat, 1 },\mathbf{fr}^1)$ respectively.

An isotopy  
of  flat approching $PS$-cycles
  between $Z^{\wr,0}$ and $Z^{\wr,1}$ converging to $( \tilde{Z}^{\flat} ,\tilde{\mathbf{fr}}  )$   
   consists in a collection of chains
\begin{equation} \label{collection-chains-isotopy-approching}
\tilde{Z}^{\wr}=  ( \tilde{Z}_{G,m, \mathcal{P}}^{\wr} )_{G,m, \mathcal{P}} 
\end{equation}
$$\tilde{Z}_{G,m, \mathcal{P}}^{\wr} \in C_{|H(G)| + |m|+1} ( M^{H(G) \times \R } , \mathfrak{o}_{H(G)})^{\text{Aut}(G,m, \mathcal{P})}  $$
such that
\begin{itemize}
\item there exists  $\lambda^{\dagger} >0$  big enough such that
$$ \tilde{Z}^{\wr, <- \lambda^{\dagger}} =  Z^{\wr,0}  \times  \R_{<-\lambda^{\dagger}}  , \quad \tilde{Z}^{\wr, > \lambda^{\dagger}} =  Z^{\wr,1}  \times \R_{> \lambda^{\dagger}}   ;$$

 \item 
the transversality condition is satisfied;


\item the forgetful compatibility condition is satisfied.

\item the cycle condition is satisfied;

\item the support condition is satisfied;

\item the cluster condition is satisfied;

\item the convergence condition is satisfied.


\end{itemize}

We denote by $\tilde{\mathcal{Z}}^{\wr}(M|  \tilde{Z}^{\wr, \circledcirc})$ the space of isotopies of asymptotic flat $PS$-cycles.


In order to explain the conditions above we need to introduce the one parameter version $\tilde{\underline{M}}_{G,\mathcal{P}}$ of the multi-loop space of subsection \ref{forgetful-approching-section}.

For any $\lambda \in \R$ denote by ${\underline{M}}_{G, \mathcal{P}}^{\lambda}$  the multi-loop space of subsection \ref{forgetful-approching-section} defined using $\mathbf{w}^{\lambda}$.
Let $ \tilde{\underline{M}}_{G, \mathcal{P}}$ be the set of arrays $((w_v,t_v)_v, \lambda)$ with $(w_v,t_v)_v \in  {\underline{M}}_{G, \mathcal{P}}^{\lambda}$,

The support condition, the cluster condition and convergence condition are  the one parameter version of the analogous conditions for flat-approaching $PS$-cycles. 

For any real number $rho >0$ set
$$ \tilde{\underline{M}}_{G, \mathcal{P}}^{<\rho} = \bigsqcup_{\lambda}  ({\underline{M}}_{G, \mathcal{P}}^{\lambda, <\rho} \times \{\lambda\} ) \subset  \tilde{\underline{M}}_{G, \mathcal{P}} ,$$
that is, the set of arrays $((w_v,t_v)_v, \lambda)$ with $(w_v,t_v)_v \in  {\underline{M}}_{G, \mathcal{P}}^{\lambda, <\rho}$.

We say that the \emph{support condition} is satisfied for $(G,m, \mathcal{P}) \in \mathfrak{G}^{part}_*  $  if there exists  a  real positive number $\rho^{\dagger}$ such that
$$ Supp(  \tilde{\underline{Z}}_{G,m, \mathcal{P}}^{\wr, (n)} )  \subset  \tilde{\underline{M}}_{G, \mathcal{P}}^{<\rho^{\dagger}/n} \quad \forall n .$$


Set
$$\tilde{Diag}_{\mathcal{P}}= \bigsqcup_{\lambda} Diag_{\mathcal{P}}^{\lambda} \times \{ \lambda \},$$
where, for each $\lambda$, $Diag_{\mathcal{P}}^{\lambda}$ is defined using $\mathbf{w}^{\lambda}$.
For $ \tilde{\mathbf{z}} = (\mathbf{z}^{\lambda})_{\lambda}  $ and $\rho>0$ set
$$\tilde{\underline{D}}_{\rho, \tilde{\mathbf{z}} } =  \bigsqcup_{\lambda} (\underline{D}_{\rho, \mathbf{z}^{\lambda} } \times \{ \lambda \}) .$$
For  $\mathcal{P} \prec \mathcal{P}^{\triangleright} $ and $\eta >0$ set
$$  Cut_{\mathcal{P}^{\triangleright} , \eta} ( \tilde{Diag}_{\mathcal{P} })= \bigsqcup_{\lambda}  (  Cut_{\mathcal{P}^{\triangleright} , \eta} ( Diag_{\mathcal{P} }^{\lambda}) \times \{\lambda\}  )$$

Let $(G,m,\mathcal{P}) \in \mathfrak{G}_*^{part} $, $\rho, \eta> 0$.
We say that the cluster condition is satisfied for $(G,m, \mathcal{P}, \rho, \eta)$ if 
\begin{equation} \label{matching-asintontically-isotopy}  
\tilde{\underline{Z}}^{\wr, (n)}_{G,m, \mathcal{P}} \cap \tilde{\underline{D}}_{\rho/n, \tilde{\mathbf{z}} }(\mathcal{P}^{\triangleright}) = \tilde{\underline{Z}}^{\wr, (n)}_{G,m, \mathcal{P}^{\triangleright}}  \cap \tilde{\underline{D}}_{\rho/n, \tilde{\mathbf{z}} } ,
  \end{equation}
for each $\mathcal{P}^{\triangleright} $ with $\mathcal{P} \prec \mathcal{P}^{\triangleright} $ and $ \mathbf{z} \in  Cut_{\mathcal{P}^{\triangleright} , \eta/n} ( \tilde{Diag}_{\mathcal{P} }) $.

We say that  the \emph{cluster condition} is satisfied for
$(G,m, \mathcal{P} , \rho)$ if there exists $\eta$ big enough such that the cluster condition is satisfied for
$(G,m,  \mathcal{P}, \rho, \eta)$ .

We say that  the \emph{convergence condition} is satisfied for $(G,m, \mathcal{P})$ if
  \begin{equation} \label{convergence-loop-isotopy} 
 Dil_n( \text{exp}^{-1}(\underline{Z}_{G,m,\mathcal{P}}^{\wr, (n)}  \cap \tilde{\underline{D}}_{\rho/n, \tilde{\mathbf{z}} } )) \rightarrow     \tilde{\underline{Z}}^{\flat}_{G,m, \mathcal{P}} \cap (\underline{D}_{\rho, \mathbf{0}}^{\flat} \times \R)   
 \end{equation}
 uniformly for $\tilde{\mathbf{z}}  \in  \tilde{Diag}_{\mathcal{P}}$,    for each $\rho>0$ . In  (\ref{convergence-loop-isotopy})  we have used $\tilde{\mathbf{fr}}$ to identify the tangent space of $T_{(z,\lambda)}(M \times \R)$ with $\R^3 \times \R$ for each $(z, \lambda) \in M \times \R$.

\subsubsection{Existence and Unicity}

The next Propositions are extensions of the one appearing in \cite{PSPCS1} and they can be proved in the same way.

\begin{proposition}   \label{existence-asymptotic-proposition}
Let $Z^{\flat}$ be a flat $PS$-cycles. Let  $\mathbf{w}$ be oriented embedded links. Let $\mathbf{fr}$  be a frame of $M$ compatible with $\mathbf{w}$.
There exists an    flat-approaching $PS$-cycle $Z^{\wr}$ converging  to $(Z^{\flat }, \mathbf{fr})$.
\end{proposition}

\begin{proposition} \label{unicity-asymptotic-proposition}

Let $Z^{\flat, 0}, Z^{\flat, 1}$ be flat $PS$-cycles. Let $\mathbf{w}^0,\mathbf{w}^1$ be oriented embedded links. Let $\mathbf{fr}^0$ and $\mathbf{fr}^1$ be frames of $M$ compatible with $\mathbf{w}^0$ and  $\mathbf{w}^1$ respectively.

Let   $Z^{\wr,0}, Z^{\wr,1}$  be flat-approaching $PS$-cycles converging  to $(Z^{\flat,0 }, \mathbf{fr}^0), (Z^{\flat, 1 },\mathbf{fr}^1)$ respectively.

For each   $\tilde{Z}^{\flat}$ isotopy between  $Z^{\flat, 0}$ and $Z^{\flat, 1}$,   $ \tilde{\mathbf{w}}$ isotopy between $\mathbf{w}^0$ and  $\mathbf{w}^1$,  $\tilde{\mathbf{fr}}$  isotopy between   $\mathbf{fr}^0$ and $\mathbf{fr}^1$ compatible with  $ \tilde{\mathbf{w}}$, there exists an isotopy of flat-approaching $PS$-cycles $\tilde{Z}^{\wr}$ between $Z^{\wr,0}$ and  $Z^{\wr,1}$ converging to  $(\tilde{Z}^{\flat},  \tilde{\mathbf{fr}}  )$.




\end{proposition}

\subsection{Proof of Theorem  \ref{coherent-theorem}} \label{proof-coherent-section}

In this section we prove Theorem \ref{coherent-theorem}. We shall characterize  $Z^{  \mathbf{w}, \mathbf{fr} }  $ extending the support and cluster conditions. 




We extend the support condition as follows. For $G \in \mathfrak{G}$ and $\rho >0$ define the set  $\underline{M}_{G}^{< \rho}$ as the set of elements
$( w_v , t_v)_v \in \underline{M}_{G}$ with the following properties:
\begin{itemize}
 \item 
$ Diameter ( \sqcup_{v \in V_c} w_v(\mathbb{S}^1) ) < \rho$  for any $c \in  Comp_0(G)$  ;
\item
There exists a decomposition in disjoint sets $ V_{\vartriangle}= \sqcup_{v^{\sharp} \in V^{\sharp}}  V_{\vartriangle}^{v^{\sharp}}$ such that the following happen.
For any $v^{\sharp} \in V^{\sharp}$ there exist two finite arrays 
\begin{equation} \label{covering-support}
(I_j)_{j \in J^0} , (I_j)_{j \in J^1}
\end{equation}  
of  disjoint open connected subsets of $ S^1$, with the following properties:
\begin{itemize}
\item for any $j \in J^0$ we have (up to parametrizations) $w_{v^{\sharp}}= w^{\sharp}_{v^{\sharp}}$ on $I_j$ ; 
\item
there exists a decomposition in disjoint sets $ V_{\vartriangle}^{v^{\sharp}}= \sqcup_{j \in J_1}  V_{\vartriangle,j}$  with 
$ I_j  \cap Im(t_v)   \neq \emptyset$ if $V_{\vartriangle,j} = \emptyset$ and such that
$$Diam(w_{v^{\sharp}}(I_j) \sqcup \bigsqcup_{v \in V_{\vartriangle,j} }w_v(\mathbb{S}^1)) < \rho_0$$
for any $j \in J^1$;
\item $ S^1 = \sqcup_{j \in J^0 \sqcup J^1} I_j. $
\end{itemize}

\end{itemize}

We say that the support  condition is satisfied if  for  $(G,m) \in \mathfrak{G}_*  $  if there exists  a  real positive number $\rho^{\dagger}$ such that
$$ Supp(  \underline{Z}_{G,m}^{( \mathbf{w}, \mathbf{fr} ), (n)} )  \subset  \underline{M}_{G}^{<\rho^{\dagger}/n} \quad \forall n .$$

In order to extend the cluster condition we define, for $(G, \mathcal{P})  \in \mathfrak{G}^{part}$ and $\eta>0$,
$ Cut_{\eta} ( Diag_{\mathcal{P} })$ as the set of $\mathbf{z}  \in Diag_{\mathcal{P} }$ such that
\begin{itemize}
\item for each $v^{\sharp} \in V^{\sharp}$,  the cyclic order  on $ w_{ v^{\sharp}}^{\sharp}$ of the array of points $(z_v)_{v \in V(\mathcal{P}, v^{\sharp})}$  is compatible with the cyclic order of $ \mathcal{P}^{v^{\sharp}}$;
 \item  $d ( \mathbf{z}, Diag_{\mathcal{P}^{\triangleright, \prime}}  ) > \eta  $ for each $\mathcal{P}^{\triangleright, \prime}$ with $\mathcal{P}  \prec \mathcal{P}^{\triangleright, \prime} $.
\end{itemize}

 Denote by   $\underline{D}_{\rho, \mathbf{z}}^{\ast}$ the set of elements $({w}_v,t_v)_v \in \underline{D}_{\rho, \mathbf{z}}$ such that the for each $v^{\sharp} \in V^{\sharp}$  and $P \in V(\mathcal{P},  v^{\sharp} )$ the following happen:  
 if $H_{v^{\sharp}} \cap P \neq H_{v^{\sharp}} $ then the array  $t_P$ is compatible with the ordering of $H_{v^{\sharp}} \cap P$ obtained from the cyclic ordering of $H_{v^{\sharp}}$.



For $\eta $ big enough, we have an inclusion of sets
\begin{equation} \label{inclusion-disks-partion-loop}
     \underline{D}_{\rho, \mathbf{z}}^{\ast} \hookrightarrow   \underline{M}_{G} 
\end{equation}
compatible with the cycle condition.



We require that for each $(G,m,\mathcal{P}) \in \mathfrak{G}_*^{part}$, and $\rho>0$ there exists $\eta$ big enough such that
\begin{equation} \label{matching-asintontically-coherent}  
\underline{Z}_{ (G,m) } ^{( \mathbf{w}, \mathbf{fr}),(n)} \cap \underline{D}_{\rho/n ,  \mathbf{z}}^{\ast}  = \underline{Z}_{G,m, \mathcal{P}}^{\wr, (n)}  \cap \underline{D}_{\rho/n , \mathbf{z}}^{\ast}
  \end{equation}
for each $ \mathbf{z} \in Cut_{\eta/n}( Diag_{\mathcal{P}})$ and $n$ big enough.

\begin{lemma} \label{coherent-existence-lemma}
Given a flat-approaching $PS$-cycle ${Z}_{}^{\wr}$, there exists a unique nice $PS$-cyclce ${Z}^{( \mathbf{w}, \mathbf{fr}) } \in \mathcal{Z}$ which satisfies the  support property and the cluster condition (\ref{matching-asintontically-coherent}).
\end{lemma}

\begin{proof}

The existence of $\underline{Z}_{  \mathbf{w}, \mathbf{fr} } $ can be proved adapting the proof of  Proposition \ref{existence-asymptotic-proposition}.




Let $(G,m) \in \mathfrak{G}_*$. Assume  we have defined   $ \underline{Z}^{  ( \mathbf{w}, \mathbf{fr})}_{\prec (G,m) }  $.

Take $\rho^{\dagger}$ big enough such that the support condition of radius $\rho^{\dagger}$ is satisfied for  $ \underline{Z}^{  ( \mathbf{w}, \mathbf{fr})}_{(G',m')}  $  when $(G',m') \prec (G,m)$ and for  $\underline{Z}^{\wr}_{(G',m')} $   when $(G',m') \preccurlyeq (G,m)$.



As in the proof of  Proposition \ref{existence-asymptotic-proposition} we may define real positive numbers 
 $( \rho_{\mathcal{P}},   \eta_{\mathcal{P}})_{\mathcal{P}}$ such that  

\begin{itemize}

\item \begin{equation} \label{covering-boundary-condition-asymptotic}
    \underline{M}_{G,\mathcal{P}}^{ <  \rho^{\dagger}/n }  \subset   \underline{M}_{G, \mathcal{P}}( \rho_{\mathcal{P}/n} ) \sqcup  \bigsqcup_{\mathcal{P}' \prec \mathcal{P}}  \underline{M}_{G, \mathcal{P}}(\mathcal{P}', \rho_{\mathcal{P}'}/n,   \eta_{\mathcal{P}'}/n ) ;
\end{equation} 
\item \begin{equation} \label{not-overlapping-condition-asymptotic}
 \underline{M}_{G,   \mathcal{P}}(\mathcal{P}_1, \rho_{\mathcal{P}_1/n},   \eta_{\mathcal{P}_1}/n ) \cap  \underline{M}_{G,  \mathcal{P} }(\mathcal{P}_2, \rho_{\mathcal{P}_2}/n,   \eta_{\mathcal{P}_2} /n ) \neq \emptyset \Longrightarrow
  \mathcal{P}_1 \preccurlyeq \mathcal{P}_2 \text{   or   }  \mathcal{P}_2 \preccurlyeq \mathcal{P}_1,
  \end{equation}
when $\mathcal{P}_1 \prec \mathcal{P}$ and $\mathcal{P}_2 \prec \mathcal{P}$. 
\item 
\begin{equation} \label{not-overlapping-condition-asymptotic-final}
 \underline{M}_{G}(\mathcal{P}_1, \rho_{\mathcal{P}_1/n},   \eta_{\mathcal{P}_1}/n ) \cap  \underline{M}_{G }(\mathcal{P}_2, \rho_{\mathcal{P}_2}/n,   \eta_{\mathcal{P}_2} /n ) \neq \emptyset \Longrightarrow
  \mathcal{P}_1 \preccurlyeq \mathcal{P}_2 \text{   or   }  \mathcal{P}_2 \preccurlyeq \mathcal{P}_1,
  \end{equation}
for each $\mathcal{P}_1,\mathcal{P}_2 $;

\item
\begin{equation} \label{disjoint-disks-asymptotic}
  \mathbf{z}' \in Cut_{\mathcal{P}, \eta_{\mathcal{P}'}} ( Diag_{\mathcal{P}'}), v_1 \sim_{\mathcal{P}} v_2 ,   v_1 \nsim_{\mathcal{P}'} v_2 \Longrightarrow  {D}_{\rho_{ \mathcal{P}'}, z_{ v_1}'} \cap  {D}_{\rho_{ \mathcal{P}'}, z_{ v_2}'} = \emptyset .
\end{equation} 
for $\mathcal{P}' \prec \mathcal{P}$. 
Moreover we assume

\begin{equation} \label{disjoint-disks-asymptotic-final}
  \mathbf{z} \in Cut_{ \eta_{\mathcal{P}}} ( Diag_{\mathcal{P}}),    v_1 \nsim_{\mathcal{P}} v_2 \Longrightarrow  {D}_{\rho_{ \mathcal{P}}, z_{ v_1}} \cap  {D}_{\rho_{ \mathcal{P}}, z_{ v_2}} = \emptyset .
\end{equation}

\item
 $\underline{Z}^{\wr}$ satisfies the cluster condition for
\begin{equation}  \label{matching-preview-asymptotic}
    (G',m', \mathcal{P},  \rho_{ \mathcal{P}}, \eta_{ \mathcal{P}})
    \end{equation}
when $(G',m') \preccurlyeq (G,m)$    ;

\item
  $ \underline{Z}^{  ( \mathbf{w}, \mathbf{fr})}_{\prec G}$ satisfies the cluster condition (\ref{matching-asintontically-coherent}) for
\begin{equation}  \label{matching-preview-asymptotic-final}
    (G',m', \mathcal{P},  \rho_{ \mathcal{P}}, \eta_{ \mathcal{P}})
    \end{equation}
    when $(G',m') \prec (G,m)$    .


\end{itemize}



Condition (\ref{covering-boundary-condition-asymptotic}) implies
\item \begin{equation} \label{covering-boundary-condition-asymptotic-end}
    \underline{M}_{G}^{ <  \rho^{\dagger}/n }  \subset    \bigsqcup_{\mathcal{P}}  \underline{M}_{G}(\mathcal{P}, \rho_{\mathcal{P}}/n,   \eta_{\mathcal{P}}/n ) .
\end{equation}

There exists a unique  chain $\underline{Z}^{  ( \mathbf{w}, \mathbf{fr}), (n)}_{G,m } $  with support on $ \underline{M}_{G}^{< \rho^{\dagger}/n}$ such that for each $ \mathcal{P}$ we have
\begin{equation} \label{definition-condititon-final}
 \underline{Z}_{G,  m  }^{( \mathbf{w}, \mathbf{fr}), (n)}  =  \underline{Z}_{G,  m  , \mathcal{P}}^{\wr, (n)} 
 \end{equation}
on $ \underline{M}_{G} (\mathcal{P}, \rho_{G,m, \mathcal{P}},   \eta_{G,m,\mathcal{P}} )     $.
Properties ( \ref{matching-preview-asymptotic}) and  (\ref{not-overlapping-condition-asymptotic}) imply that the conditions are compatible for different $\mathcal{P}$.

$ \underline{Z}_{G,  m }^{( \mathbf{w}, \mathbf{fr}),(n)}$ satisfies the transversality condition.

Conditions (\ref{matching-preview-asymptotic}),    (\ref{matching-preview-asymptotic-final}) 
and the cycle condition for $\underline{Z}^{\wr}$ tells us that  $ \underline{Z}^{( \mathbf{w}, \mathbf{fr}), (n)}$ satisfies the cycle condition for $(G,m)$.
.

The cluster condition follows from property    (\ref{definition-condititon-final}) as in \cite{PSPCS1}.


\end{proof}

We now consider isotopies. As before, let $\mathbf{w}^0$ and  $\mathbf{w}^1$ be oriented embedded links and let
$$ \tilde{\mathbf{w}} = ( \mathbf{w}^{\lambda} )_{\lambda \in [0,1]}$$ 
be an isotopy between $\mathbf{w}^0$ and  $\mathbf{w}^1$.
Let $\mathbf{fr}^0$ and $\mathbf{fr}^1$ be frames of $M$ compatible with  $\mathbf{w}^0$ and $\mathbf{w}^1$ respectively and let $\tilde{\mathbf{fr}}$ be an isotopy  between  $\mathbf{fr}^0$ and $\mathbf{fr}^1$ compatible with $\tilde{\mathbf{w}}$.

Define $ \tilde{\underline{M}}_{G}$ as the set of arrays $((w_v,t_v)_v, \lambda)$ with $(w_v,t_v)_v \in  {\underline{M}}_{G}^{\lambda}$,
where, for each $\lambda \in \R$, the space ${\underline{M}}_{G}^{\lambda}$ is defined using $\mathbf{w}^{\lambda}$.

Set 
$$ \tilde{\underline{M}}_{G}^{<\rho^{\dagger}} = \bigsqcup_{\lambda}  ({\underline{M}}_{G}^{\lambda, <\rho^{\dagger}} \times \{\lambda\} )  \subset  \tilde{\underline{M}}_{G},$$
i.e., the set of arrays $((w_v,t_v)_v,\lambda)$ such that $(w_v,t_v)_v \in  {\underline{M}}_{G}^{\lambda, <\rho_0}$.

We say that the support condition is satisfied for $(G,m) \in \mathfrak{G}_*  $  if there exists  a  real positive number $\rho^{\dagger}$ such that
$$ Supp(   \underline{Z}_{G,m}^{( \mathbf{w}, \mathbf{fr} ), (n)} )  \subset  \tilde{\underline{M}}_{G}^{<\rho^{\dagger}/n} \quad \forall n .$$



For each partition $\mathcal{P}  $ and  real positive $\eta >0$ set
$$  Cut_{ \eta} ( \tilde{Diag}_{\mathcal{P} })= \bigsqcup_{\lambda} Cut_{\eta} ( Diag_{\mathcal{P} }^{\lambda} ) \times \{ \lambda \}   $$

Let $(G,m,\mathcal{P}) \in \mathfrak{G}_*^{part} $, $\rho, \eta> 0$.
We say that the cluster condition is satisfied for $(G,m, \mathcal{P}, \rho, \eta)$ if 
\begin{equation} \label{matching-asintontically-isotopy-final}  
 \underline{Z}_{G,m}^{( \mathbf{w}, \mathbf{fr} ), (n)} \cap \tilde{\underline{D}}_{\rho/n, \tilde{\mathbf{z}} }^{\ast} = \tilde{\underline{Z}}^{\wr, (n)}_{G,m, \mathcal{P}}  \cap \tilde{\underline{D}}_{\rho/n, \tilde{\mathbf{z}} }^{\ast} ,
  \end{equation}
for each $\mathcal{P}$ and $ \tilde{\mathbf{z}} \in  Cut_{ \eta/n} ( \tilde{Diag}_{\mathcal{P} }) $.

We say that  the cluster condition is satisfied for
$(G,m, \mathcal{P} , \rho)$ if there exists $\eta$ big enough such that the cluster condition is satisfied for
$(G,m,  \mathcal{P}, \rho, \eta)$ .

\begin{lemma}   \label{coherent-iso-existence-lemma}



Let   $Z^{\wr,0}, Z^{\wr,1}$  be flat-approaching $PS$-cycles and 
let $\underline{Z}^{  \mathbf{w}^0, \mathbf{fr}^0 }$ and $\underline{Z}^{  \mathbf{w}^1, \mathbf{fr}^1 } $
be the corresponding coherent cycles as in Lemma \ref{coherent-existence-lemma}.

To an isotopy $\tilde{Z}^{\wr}$  of flat-approaching $PS$-cycles between  $Z^{\wr,0} $ and $Z^{\wr,1}$ corresponds an isotopy $\underline{\tilde{Z}}^{  \mathbf{w}, \mathbf{fr} } $ between  
$\underline{Z}^{  \mathbf{w}, \mathbf{fr} }$ and $\underline{Z}^{  \mathbf{w}, \mathbf{fr} } $.


\end{lemma}

\begin{proof}

This claim can be proved as in the proof of  Proposition \ref{unicity-asymptotic-proposition}.



Let $(G,m) \in \mathfrak{G}_* $. Assume we have defined $\tilde{Z}_{\prec (G,m)}^{ ( \mathbf{w}, \mathbf{fr})}$.

Take $\rho^{\dagger}$ big enough such that the support property of radius $\rho^{\dagger}$ is satisfied for  $ \tilde{\underline{Z}}^{  ( \mathbf{w}, \mathbf{fr})}_{G,m }  $ for $(G',m'))$ when $(G',m') \prec (G,m)$, for  $\underline{\tilde{Z}}^{\wr} $  for $(G',m')$ when $(G',m') \preccurlyeq (G,m)$,   for ${Z}_{}^{\wr}$ for $(G',m')$ when  $(G',m') \preccurlyeq (G,m)$.


As in the proof of  Proposition \ref{unicity-asymptotic-proposition} we may define real positive numbers 
 $( \rho_{\mathcal{P}},   \eta_{\mathcal{P}})_{\mathcal{P}}$ such that


\begin{itemize}





\item Conditions (\ref{covering-boundary-condition-asymptotic}), (\ref{not-overlapping-condition-asymptotic}),
 (\ref{not-overlapping-condition-asymptotic-final})  , (\ref{disjoint-disks-asymptotic}) and  (\ref{disjoint-disks-asymptotic-final}) are satisfied;


 \item $\underline{\tilde{Z}}^{\wr}$ satisfies the cluster condition for
\begin{equation}  \label{matching-preview-asymptotic-iso}
    (G',m', \mathcal{P},   \rho_{ \mathcal{P}}, \eta_{ \mathcal{P}})
    \end{equation}
for $(G',m') \preccurlyeq (G,m)$   ;


\item
  $ \tilde{\underline{Z}}^{  ( \mathbf{w}, \mathbf{fr})}$ satisfies the cluster condition
  for
\begin{equation}  \label{matching-preview-asymptotic-final-iso}
    (G',m', \mathcal{P},  \rho_{ \mathcal{P}}, \eta_{ \mathcal{P}})
    \end{equation}
    when $(G',m') \prec (G,m)$ ;

\item $Z^{\wr, 0}$ and $ Z^{\wr, 1} $ satisfies the cluster condition for
\begin{equation}  \label{matching-inductive-lemma-asi-flat-start-end}
    (G',m',  \mathcal{P}, \rho_{ \mathcal{P}}, \eta_{ \mathcal{P}})
\end{equation}
for $(G',m') \preccurlyeq (G,m)$  ;

\item
  $ \underline{Z}^{  ( \mathbf{w}_0, \mathbf{fr}_0)},  \underline{Z}^{  ( \mathbf{w}_1, \mathbf{fr}_1)} $ satisfies the cluster condition (\ref{matching-asintontically-coherent}) for
\begin{equation}  \label{matching-preview-asymptotic-final-iso-end}
    (G',m', \mathcal{P},  \rho_{ \mathcal{P}}, \eta_{ \mathcal{P}})
    \end{equation}
    when $(G',m') \preccurlyeq (G,m)$ .

\end{itemize}



There exists a unique chain $\underline{\tilde{Z}}^{  ( \mathbf{w}, \mathbf{fr}), (n)}_{G,m } $   with support on $ \tilde{\underline{M}}_{G}^{< \rho^{\dagger}} $ such that for each $ \mathcal{P}$ we have
\begin{equation} \label{definition-condititon-final-iso}
 \underline{\tilde{Z}}_{G,  m }^{( \mathbf{w}, \mathbf{fr}), (n)}  =  \underline{\tilde{Z}}_{G,  m  , \mathcal{P}}^{\wr, (n)} 
 \end{equation}
on $ \tilde{\underline{M}}_{G, {\mathcal{P}} }(\mathcal{P}, \rho_{G,m, \mathcal{P}},   \eta_{G,m,\mathcal{P}} )      $ .
Properties ( \ref{matching-preview-asymptotic-iso}) and  (\ref{not-overlapping-condition-asymptotic}) imply that the conditions are compatible for different $\mathcal{P}$.


 $ \underline{\tilde{Z}}_{G,  m  , \mathcal{P}}^{ ( \mathbf{w}, \mathbf{fr}), (n)}$ satisfies the transversality condition.

Conditions (\ref{matching-preview-asymptotic-iso}),    (\ref{matching-preview-asymptotic-final-iso}) 
and the cycle condition for $ \underline{\tilde{Z}}^{\wr, (n)}$ tell us that  $ \underline{\tilde{Z}}^{( \mathbf{w}, \mathbf{fr}), (n)}$ satisfies the cycle condition for $(G,m)$.



From  (\ref{matching-inductive-lemma-asi-flat-start-end}) we have 
\begin{equation}  \label{start-end-cluster}
{ {\underline{\tilde{Z}}}}_{G,m   }^{( \mathbf{w}, \mathbf{fr}),(n), < -\lambda^{\dagger}} = \underline{Z}_{G,m }^{( \mathbf{w}, \mathbf{fr}), 0, (n)} \times \R_{<-\lambda^{\dagger}}, \quad {\underline{\tilde{Z}}}_{G,m  }^{( \mathbf{w}, \mathbf{fr}), (n), > \lambda^{\dagger}} =  \underline{Z}_{G,m }^{( \mathbf{w}, \mathbf{fr}),1 , (n)} \times \R_{>\lambda^{\dagger}}
\end{equation}
for $\lambda^{\dagger} \gg 0$ and $n \gg 0$.

The cluster condition follows from property    (\ref{definition-condititon-final-iso}) as in \cite{PSPCS1}.

\end{proof}




\section{Reduced $PSH$}
 
 

We now define a quotient $PS$-Homology which can be easily described in terms of  module over the ring of formal power series $\Q[[\hbar, a]]$, in the formal variables $\hbar$ and $a$.

Consider on the set of the graphs the equivalence relation generated by the relation
 \begin{multline} \label{equivalence-zero-component}
       ( \chi_{\vartriangle} , d_{\vartriangle} , V_{\vartriangle}, Comp_0 , (V_c, d_c , g_c)_c , (H_v )_v  , E ) \sim  \\   
( \chi_{\vartriangle} + \chi_{c'} , d_{\vartriangle} + d_{c'} , V_{\vartriangle}, Comp_0 \setminus c' , (V_c,d_c, g_c)_c , (H_v )_v  , E ).
 \end{multline}
if $c' \in Comp_0$ is a component with $V_{c'} = \emptyset$,

Let $\mathcal{Z}^{null}$  be the set of $PS$-cycles $Z$ with the property that
$$ \sum_{[G] = T} Z_{G,m}=0$$
for each equivalence class of graphs $T$, and $m$ a marking of $T$.


Set
$$   PSH^{red}(M| \underline{Z}^{\circledcirc}) = PSH(M|  \underline{Z}^{\circledcirc})/\mathcal{Z}^{null}. $$

The action of $\hbar^{k_0} a^{d_0}$ for  $k_0,d_0 \in \Z_{\geq 0} $ on $PSH$ is inducted by the action on $\mathfrak{G}$ defined by 
$$ \chi_{\vartriangle} (\hbar^{k_0} a^{d_0} G) = \chi_{\vartriangle} (G) - k_0,  d_{\vartriangle} (\hbar^{k_0} a^{d_0} G) = d_{\vartriangle}(G) + d_0 .$$


\subsection{Embedded Links}

We now study the space $PSH(M| \mathbf{w} , \underline{Z}^{\circledcirc , \diamond}  )$ when  
$$ \mathbf{w} = (w_v^{\sharp})_{v \in V^{\sharp}}$$ is an embedded link.

In the this situation, the obstructions of subsection \ref{obstruction-section} simplifies: it is concentrated on the set of decorated graphs $G$ with
$$V(G)= V^{\sharp}, H(G)=\emptyset ,$$
and it is given by
$$\mathfrak{Obs}_G^{(n)} =  \underline{Z}^{1,(n)}_G -  \underline{Z}^{0,(n)}_G + pr_*( \sum_{(G', e')}  \delta_{e'} \tilde{\underline{Z}}_{G'}^{(n)} +   \sum_{(G', v')} \delta_{v'} \tilde{\underline{Z}}_{G'}^{(n)}) .$$

\begin{proposition} \label{one-dimension}
Let $ \underline{Z}^{\circledcirc, \diamond} $ be a nice annulus data. 
If $ \mathbf{w}$ is an embedded link  then $PSH^{red}(M| \mathbf{w} , \underline{Z}^{\circledcirc , \diamond}  )$ is a rank one free module over $\Q[[\hbar, a]]$.
\end{proposition}

\begin{proof}

Denote by $G^{\heartsuit}$ the decorated graph with $ H(G^{\heartsuit})= \emptyset, \chi_{\vartriangle}=0, d_{\vartriangle}=0 , Comp_0= \emptyset, D_{\vartriangle} = \emptyset, V_{\vartriangle}= V^{\sharp}$ .
Fix a nice $PS$-cycle $Z^{\heartsuit} \in \mathcal{Z}(\underline{Z}^{\circledcirc , \diamond}, \mathbf{w} ) $ with  
\begin{equation} \label{generator-prop2}
\underline{Z}^{\heartsuit}_{G^{\heartsuit}}= \mathbf{w}.
\end{equation}
The existence of $\underline{Z}^{\heartsuit}$ can be seen applying recursively Lemma \ref{extension-nice}.

We claim that the map
$$    \Q[[\hbar, a]] \rightarrow PSH^{red}(M| \mathbf{w} , \underline{Z}^{\circledcirc , \diamond}  )$$
\begin{equation} \label{iso-map}
  q(\hbar, a) \mapsto  q(\hbar, a)  Z^{\heartsuit} 
\end{equation}
is an isomorphism.

We first  prove that  the map (\ref{iso-map}) is injective. Fix  $ q(\hbar, a) \in    \Q[[\hbar, a]]  $ and assume that there exists an isotopy $\tilde Z$ between $q(\hbar,a) Z^{\heartsuit}$ and a null $PS$-cycle $Z^{null} \in \mathcal{Z}^{null}$.

Let  $r \hbar^{\kappa_0} a^{d_0}$ be the leading term of  $q$. In particular $ Z_{\prec \hbar^{\kappa_0} a^{d_0} G^{\heartsuit}} =0$. 



For each $G \in \mathfrak{G}$ denote by $V^{\star}(G)$ the set of $v \in V(G)$ with $v \notin V^{\sharp}$ and $H_v = \emptyset$. Denote by $\mathfrak{G}^{\star}$ the set of $G \in \mathfrak{G}$ with  $V^{\star}(G) = \emptyset$.


For each $ G \in  \mathfrak{G}^{\star}  $ set
$$\tilde{\underline{Z}}_G^{\star} = \sum_{G'} \pi_*( \tilde{\underline{Z}}_{G'} )$$
where the sum run over the graphs $G' \in \mathfrak{G}$ such that $\delta_{V^{\star}(G')} G' =G$ and $\pi$ denotes the obvious  projection $\mathfrak{L}_{\hat{V}(G')}(M) \rightarrow \mathfrak{L}_{\hat{V}(G)}(M)$.


The cycle condition for $\tilde{\underline{Z}}$  tells us 
\begin{equation} \label{cycle-star}
  \partial  \tilde{\underline{Z}}_{G}^{\star}  =    \sum_{(G', e')}  \delta_{e'} \tilde{\underline{Z}}_{G'}^{\star} . 
  \end{equation}
The terms with $\delta_{v'}$ are cancelled in the sum.


Define the chain $  \tilde{\underline{Z}}_{G'}^{\star, \prime}  $ taking the image under the map $(t,s) \mapsto (t+s,s) $ of the chain $\tilde{\underline{Z}}_{G'}^{\star} \times \R_s$, intersecting with the set  $\mathfrak{L}_{\hat{V}(G')}(M) \times \R_t \times \R_{ s \geq 0 }$ and projecting in the $\R_s$-direction. Equation (\ref{cycle-star}) yields
$$   \partial  \tilde{\underline{Z}}_{G'}^{\star, \prime}  =    \sum_{(G'', e'')}  \delta_{e''} \tilde{\underline{Z}}_{G''}^{\star, \prime} -  \tilde{\underline{Z}}_{G'}^{\star}. $$
From the last expression we obtain
$$  \underline{Z}_{ \hbar^{\kappa_0} a^{d_0} G^{\heartsuit}} = \partial ( pr_* (\tilde{\underline{Z}}^{\star}_{ \hbar^{\kappa_0} a^{d_0} G^{\heartsuit}})+ \sum_{G' ,e'}   \delta_{e'} pr_* ( {\tilde{\underline{Z}}}_{G'}^{\star, \prime}  )  ) .$$
We obtain a contradiction since by definition
$   \underline{Z}_{ \hbar^{\kappa_0} a^{d_0} G^{\heartsuit}} \neq 0 $ and the chain we are making the boundary of has support close to $\mathbf{w}$.






Now we prove that the map (\ref{iso-map}) is surjective. We need to show that for each $Z \in \mathcal{Z}$  there exists a formal power series  $q(\hbar,a) \in \Q[[\hbar ,a]]$ such that $Z$ and $q(\hbar, a)  Z^{\heartsuit}$ are isotopic up a null $PS$-cycle.  The formal power series $q$ is constructed using an inductive argument on the graphs.
Assume we have defined $q$  and $\tilde{Z}_{\prec  G}^{(n)}  \in \tilde{\mathcal{Z}}_{\prec G} $ isotopy between   $ Z_{\prec  G}^{(n)}$ and $ (q(\hbar,a) Z^{\heartsuit,(n)})_{\prec G} $.  
The obstruction to define $\tilde{Z}_{\preccurlyeq G} $ are concentrated on the graphs $G$ with  $H(G) = \emptyset$, $V(G) = V^{\sharp}$. Let $r \in \Q$ such that  $\mathfrak{Obs}_G = r \mathbf{w}$, up to small isotopy. The replacement  
   $$ q  \leadsto q + r a^{d(G)}  \hbar^{\kappa(G)} $$ 
makes the obstruction to vanishe.

Finally, an argument similar to the one used in the proof of  the injectivity above yields the independence of  $q$ by $n$, at least for $n$ big enough.

\end{proof}

\subsection{Singular Links} \label{multi-loop-equivalence-section}

We now extend the result of the preview section to singular links.

\subsubsection{Twist Operation on multi-loops}


Let $\mathbf{w}= (w_v)_{v \in V}$ be a multi-loop. 
Let $(v_1,t_1), (v_2,t_2) \in V \times \mathbb{S}^1$   with $ (v_1,t_1) \neq (v_2,t_2)$ and $w_{v_1}(t_1)= w_{v_2}(t_2)$. Define the multi-loop $  \mathbf{w}' =(w_{v'}')_{v' \in V'}$ as follows: 
\begin{itemize}
\item if $v_1 \neq v_2$,  $ V'$ is obtained replacing $v_1$ and $v_2$ by an element $v_0$. $w'_{v_0}$ is obtained gluing $\mathbb{S}^1_{v_1}$ with   $\mathbb{S}^1_{v_2}$ at the points $t_1$ and $t_2$  in the only orientation preserving way ;
\item if $v_1=v_2 $ , $  V'$ is obtained replacing $v_1$ by two elements $v_1', v_2'$. $w'_{v_1'}$ and $w'_{v_2'}$ are obtained   attaching $t_1$ and $t_2$ in $\mathbb{S}^1_{v_1}$ and splitting it in two circles in the only orientation preserving way.
\end{itemize}
We say that the multi-loops  $ \mathbf{w}$ and $  \mathbf{w}'$ are related by twist.

\subsubsection{$PSH$} \label{singular-subsection-PSH}

Let $\mathfrak{w}$ be a finite set of multi-loops closed by twist relation.  We denote by $\mathcal{Z}(M| \mathfrak{w} , \underline{Z}^{\circledcirc, \circ})$ 
the space of nice $PS$-cycles whose limit support is  contained in $\mathfrak{w}$.
 We denote by $PSH^{red}(M|\mathfrak{w}, \underline{Z}^{\circledcirc, \diamond})$ the elements of $PSH^{red}(M| \underline{Z}^{\circledcirc, \diamond})$  represented by   $\mathcal{Z}(M| \mathfrak{w} , \underline{Z}^{\circledcirc, \circ})$.


\begin{lemma} \label{crossing-dimension}
Let $\mathfrak{w}$ be a finite set of multi-loops closed by twist relation. $PSH^{red}(M| \mathfrak{w}, U^{\blacktriangle})$  is a rank $| \mathfrak{w} |$ free module over $\Q[[\hbar, a]]$.
\end{lemma}

\begin{proof}
Write $\mathfrak{w} = \{ \mathbf{w}_j \}_{j \in J}$.


For each $j \in J$, take a $PS$-cycle $Z^{ \heartsuit, j}$  with the property  (\ref{generator-prop2}) using the multi-loop $\mathbf{w}_j $. 
The argument of the proof of Proposition  \ref{one-dimension} can be easily extended to show that $\{ Z^{ \heartsuit ,j} \}_{j \in J}$ is a basis for the module $PSH(M| \mathfrak{w}, U^{\blacktriangle})$.



\end{proof}

\subsection{Transfer Map}

We now consider isotopies of links .  
Let  $\tilde{\mathfrak{w}}=(\mathfrak{w}^{\lambda})_{{\lambda} \in [0,1]}$ be a one parameter family of finite sets of multi-loops closed by twist relation. 
We say that  $\tilde{\mathfrak{w}}=(\mathfrak{w}^{\lambda})_{{\lambda} \in [0,1]}$ is piece-wise smooth if  it is union of a finite set of families of the type (\ref{basic-family-loops}).
We denote by 
$ \tilde{\mathcal{Z}}(M| \tilde{\mathfrak{w}}, \tilde{\underline{Z}}^{\circledcirc, \diamond})$ the set of elements of 
$ \tilde{Z}^{\diamond} \in \tilde{\mathcal{Z}}(M| \tilde{\underline{Z}}^{\circledcirc, \diamond})$ whose limit support is contained in $\tilde{\mathfrak{w}} $.  

We say that a  piecewise-smooth one parameter family   $\tilde{\mathfrak{w}}= (\mathfrak{w}^{\lambda})_{{\lambda} \in [0,1]}$ has increasing singularities if is finite union of 
families  (\ref{basic-family-loops}) with $b=1$.


 
 

 \begin{lemma} \label{deformation-cycles}
 

 Let $\tilde{U}^{\blacktriangle}$ be an isotopy between Euler Structures  $U^{\blacktriangle,0}$  and  $U^{\blacktriangle, 1}$ and let  $\tilde{\mathfrak{w}}= (\mathfrak{w}^{\lambda})_{{\lambda} \in [0,1]}$ be a  piecewise-smooth   one parameter family  of equivalence classes of multi-loops with increasing singularities.  

 
To  the pair $(\tilde{\mathfrak{w}}, \tilde{U}^{\blacktriangle})$ it is associated canonically a map in $PSH$:
 \begin{equation} \label{transfer-map-loops}
 transfer_{  ( \tilde{\mathfrak{w}},  \tilde{U}^{\blacktriangle})  }   : PSH(M|  \mathfrak{w}^0,U^{\blacktriangle,0}  )  \rightarrow PSH(M|  \mathfrak{w}^1, U^{\blacktriangle,1}  ) .
  \end{equation}

 Let $\tilde{U}^{\blacktriangle, \prime }$ be another isotopy between  $U^{\blacktriangle,0}$  and  $U^{\blacktriangle, 1}$ and
  $\tilde{\mathfrak{w}}^{ \prime }$ be  another isotopy between  $\mathfrak{w}^0$  and  $\mathfrak{w}^1$.
If there exists an isotopy of isotopies between $\tilde{U}^{\blacktriangle}$  and $\tilde{U}^{\blacktriangle, \prime }$ and between 
  $\tilde{\mathfrak{w}}$ and $\tilde{\mathfrak{w}}^{ \prime }$, 
we have
$$transfer_{( \tilde{\mathfrak{w}}, \tilde{U}^{\blacktriangle}) }  =  transfer_{ (\tilde{\mathfrak{w}}^ { \prime },  \tilde{U}^{\blacktriangle, \prime } )} .$$

 \end{lemma}

In the above statement the notion of isotopy of isotopies of piecewise-smooth   one parameter family  of equivalence classes of multi-loops with increasing singularities means the following.

We say that   $(  \mathfrak{w}^{{\lambda},{\lambda'}} )_{{\lambda},{\lambda'}}$  is piece-wise smooth  if it is union of a finite number of sets of the type
  \begin{equation}   \label{basic-two-family-loops}
    (\mathbf{w}^{{\lambda},{\lambda'}})_{{\lambda} \in [a,b], {\lambda'} \in [0,1]}
    \end{equation}
     piece-wise smooth.

By an isotopy of isotopies between 
  $\tilde{\mathfrak{w}}$ and $\tilde{\mathfrak{w}}^{ \prime }$ we mean a  two-parameter family  $(  \mathfrak{w}^{{\lambda},{\lambda'}} )_{{\lambda},{\lambda'}}$ piece-wise smooth such that  
  \begin{itemize}
       \item in the sets  (\ref{basic-two-family-loops}) we can assume $b=1$;       
        \item   $ \mathfrak{w}^{0,{\lambda'}}= \mathfrak{w}^{0}$ for each ${\lambda'} \in [0,1]$ ;
       \item $ \mathfrak{w}^{1,{\lambda'}}= \mathfrak{w}^1$  for each ${\lambda'} \in [0,1]$;
       \item  $ (\mathfrak{w}^{{\lambda},0})_{\lambda}= \tilde{\mathfrak{w}}$,  $ (\mathfrak{w}^{{\lambda},1})_{\lambda}= \tilde{\mathfrak{w}}^{\prime}$.
     \end{itemize}

     

 \begin{proof}
 
The existence of the map (\ref{transfer-map-loops}) follows from the following claim:  for each $Z^{0 , \diamond} \in \mathcal{Z}(M|    \mathfrak{w}^0 ,{\underline{Z}}^{\circledcirc, \diamond,0}  )$ there exist $Z^{1, \diamond} \in \mathcal{Z}(M|  \mathfrak{w}^1,{\underline{Z}}^{\circledcirc, \diamond,1}  )$ and  $\tilde{Z}^{\diamond} \in \tilde{ \mathcal{Z}}(M| \tilde{\mathfrak{w}}, \tilde{\underline{Z}}^{\circledcirc, \diamond}  )$ with $\tilde{Z}^{\diamond}$ isotopy between $Z^{0 , \diamond}$ and $Z^{1 , \diamond}$. Moreover the homology class $[Z^{1,\diamond}] \in PSH(M|  \mathfrak{w}^1,  U^{\blacktriangle,1}   )$ is determined by the homology class $[Z^{0, \diamond}] \in PSH(M|  \mathfrak{w}^0,  U^{\blacktriangle,0}  )$. 
 
 The claim may be proved adapting the  proof of Lemma \ref{transfer-annulus}.  The isotopy $\tilde{Z}$ obtained from Lemma  \ref{extension-isotopies-inductive}  may be constructed satisfying the condition $\lim \tilde{Supp} (\tilde{Z})  \subset  \tilde{\mathfrak{w}}$. 

Now let $\tilde{Z}^{\diamond, \prime} \in \tilde{ \mathcal{Z}}(M| \tilde{\mathfrak{w}}^{\prime}, \tilde{\underline{Z}}^{\circledcirc, \diamond, \prime}  )$ defined as before using $ \tilde{\mathfrak{w}}^{\prime}$ and  $ \tilde{\underline{Z}}^{\circledcirc, \diamond, \prime}$. 
Using Lemma \ref{extension-isotopies-isotopies} we can construct an isotopy of isotopies  between $\tilde{Z}$ and $\tilde{Z}^{\prime}$ with limit support contained in
 $(  \mathfrak{w}^{t,s} )_{t,s}$.

 \end{proof}

\section{Normalized-$PSH$}   \label{wilson-section}

In the preview sections we associated a $PSH$-class to a frame $\mathbf{fr}$ of the manifold $M$. 
In  this section we introduce a suitable quotient of $PSH$ which allows to associate a class to 
each triple $(\mathbf{w}, \nu,  U^{\blacktriangle})$ consisting of a framed link $(\mathbf{w}, \nu)$ and a compatible Euler Structure $ U^{\blacktriangle}$ of $M$.

 
 Let $\mathfrak{G}^{null}$ be the set of graphs $G^{} \in \mathfrak{G}^{}$ with at least a component $c \in Comp_0(G^{}) $  with $V_{c} = \emptyset$.
 Let $\mathcal{Z}^{null}$ be the set of $PS$-cycles $Z$ with support  in  $\mathfrak{G}^{ null}$.   

Define the normalized $PSH$ as 
$$   PSH^{\intercal}(M| \underline{Z}^{\circledcirc}) = PSH(M|  \underline{Z}^{\circledcirc})/\mathcal{Z}^{null}. $$










In normalized $PSH$  the obstruction
  of subsection \ref{obstruction-section} 
is considered  up to elements of  $\mathcal{Z}^{null}$.
 Hence the obstruction is concentrated on decorated graphs $G$ with 
\begin{equation} \label{obstruction-normalized}
H(G)= \emptyset ,  Comp_0(G)= \emptyset.
\end{equation}

The next two propositions are proved as Proposition \ref{one-dimension} and Proposition \ref{crossing-dimension}.

\begin{proposition} \label{one-dimension-normalized}
Let $ \underline{Z}^{\circledcirc, \diamond} $ be a nice annulus data. 
When $ \mathbf{w}$ is an embedded link , $PSH^{\intercal}(M| \mathbf{w} , \underline{Z}^{\circledcirc , \diamond}  )$ is a rank one free module over $\Q[[\hbar, a]]$.
\end{proposition}


\begin{proposition} \label{crossing-dimension-normalized}
Let $\mathfrak{w}$ be a finite set of multi-loops closed by the twist operation. $PSH^{\intercal}(M| \mathfrak{w}, U^{\blacktriangle})$  is a rank $| \mathfrak{w} |$ free module over $\Q[[\hbar, a]]$.
\end{proposition}




Now we have the following


\begin{lemma} \label{lemma-canonical-one-dimensional}
  If $\mathbf{w}_{empty}= \emptyset$ is the empty  link 
\begin{equation} \label{zero-MCH}
PSH^{\intercal}(M|  \mathbf{w}_{empty},U)= \Q[[\hbar, a]]   
\end{equation}
canonically, i.e., there exists a canonical generator of  $PSH^{\intercal}(M|  \mathbf{w}_{empty},U)$.
\end{lemma}
\begin{proof}
The proof of the Lemma is a similar to the  proof of Lemma \ref{one-dimension}. 
We consider a $PS$-cycle $Z^{\heartsuit}$ such that
\begin{equation} \label{normalized-condition2}
 H(G) =\emptyset, \underline{Z}_G^{\heartsuit} \neq 0 \Longrightarrow  G= G^{\heartsuit} ,
\end{equation}
\begin{equation} \label{normalized-condition3}
  {\underline{Z}}^{\heartsuit}_{G^{\heartsuit}} = 1.
\end{equation}
In the inductive construction of $Z^{\heartsuit}$ we can moreover assume  
\begin{equation} \label{normalized-condition1}
Z^{\heartsuit}_{G,m} \neq 0 \Longrightarrow V_{\vartriangle}(G)= \emptyset  .
\end{equation}



Let $Z^{0, \heartsuit} , Z^{1, \heartsuit}$ be nice $PS$-cycles satisfying the above conditions. 
 From (\ref{obstruction-normalized})  we know that the obstruction to the existence of the isotopy between $Z^{0, \heartsuit}$ and $Z^{1, \heartsuit}$ is concentrated in the graphs $G$ with $H(G) = \emptyset$, $V(G) = \emptyset$, and it is given by 
$$\mathfrak{obs}_G^{(n)} = \underline{Z}^{1,(n)}_G -  \underline{Z}^{0,(n)}_G + pr_*(   \sum_{(G', v')} \delta_{v'} \tilde{\underline{Z}}_{G'}^{(n)})  .$$
Condition  (\ref{normalized-condition1}) implies that  the terms with $\delta_{e'}$ are absent.

Condition (\ref{normalized-condition2}) and (\ref{normalized-condition3}) imply $ \underline{Z}^{1,(n)}_G =  \underline{Z}^{0,(n)}_G$.


Let $(G',v')$ be such that $\delta_{v'}G'=G$ and $\delta_{v'} \tilde{\underline{Z}}_{G'}^{(n)} \neq 0$.  Condition  (\ref{normalized-condition1}) tells us that $v' \in V_0(G')$ . It follows that 
$G \in \mathfrak{G}^{nul}$.

   \end{proof}

The map 
$$PSH^{\intercal}(M|  \mathbf{w}_{empty}) \rightarrow \Q[[\hbar, a]] $$
is written explicitly as
$$Z  \mapsto \sum_{k,d \geq 0}  \hbar^{k} a^{d} \underline{Z}_{ \hbar^{k} a^{d} G^{\heartsuit}} .$$
Here we consider  $ \underline{Z}_{ \hbar^{k} a^{d} G^{\heartsuit}} \in \Q$.


\subsection{Coherent Cycles}



A Coherent $PSH$-class, defined as in Section  \ref{coherent-section}, gives in particular a normalized $PSH$-class. The construction of the $PS$-cycle depends on the choice of a frame of $M$ compatible with the link. However the associated  normalized-$PSH$-class depends only on the inducted frame of $\mathbf{w}$.  Indeed, it is possible to define a more  general notion of Coherent Cycle depending on a frame of $\mathbf{w}$ which leads to a well defined normalized $PSH$-class. 


\subsubsection{Framed Links and Euler Structures}

Let 
$$\mathbf{w}= (w_v^{\sharp} )_{v \in V^{\sharp}}$$ 
be an oriented embedded link.



A frame of  $\mathbf{w}$ is usually defined as a section of $T M |_{\mathbf{w}}$ never belonging to the tangent space of $\mathbf{w}$. In this paper we prefer to
 consider a frame $\nu$ of  $\mathbf{w}$ as a map  $\nu: \R^3 \rightarrow \Gamma(T M |_{\mathbf{w}})$ such that
\begin{itemize}
\item $\nu((1,0,0)), \nu((0,1,0)) , \nu((0,0,1))$ are sections everywhere linearly independent on $\mathbf{w}$;
\item  $\nu((1,0,0))$ is everywhere tangent to $\mathbf{w}$ and in positive direction. 
\end{itemize}
$ \nu((0,1,0))$ defines a frame in the usual sense. Indeed, the  correspondence 
\begin{equation} \label{two-frames}
 \nu \mapsto \nu((0,1,0))
 \end{equation}
  determines an equivalence between the isotopies classes of the two definitions of frames (observe that the fiber of the map (\ref{two-frames}) is a contractible space.)


It is well known, and easy to check, that the set of isotopy classes of frames of a knot is a torsor on $\Z$. Hence the frames of $\mathbf{w}$ is a torsor on $\Z^{V^{\sharp}}$.

We say that an Euler Structure $U^{\blacktriangle}$ is compatible with $\nu$ if 
the restriction of $U^{\blacktriangle}$ to $\mathbf{w}$  is constant when written in the trivialization of $T M  $ defined by $\nu$, i.e., $U^{\blacktriangle}= \nu(\mathbf{a})$ for some $\mathbf{a} \in \R^3$. 




Fix a small tubular neighborhood $B_r(\mathbf{w}) $ of $ \mathbf{w}$. 
Any diffeomorphism 
\begin{equation} \label{diffeo-tubular}
 B_r(\mathbf{w})  \cong V^{\sharp} \times D^2_r \times S^1 
\end{equation}
defines in particular a frame of $ \mathbf{w}$. The space of diffeomorphisms  (\ref{diffeo-tubular}) compatible with a fixed frame is a contractible space.

Fix a diffeomorphism  (\ref{diffeo-tubular}) compatible with $\nu$. 
We consider  $M \setminus \mathbf{w}$ as a open manifold with collar at infinity $B_r(\mathbf{w})  \setminus \mathbf{w}$ . 
The Euler structures on  $M \setminus \mathbf{w}$ compatible with the frame of the collar are identified  with the Euler Structures on $M$ compatible with $\nu$.



Denote by $  \mathfrak{FrEul}(M \setminus \mathbf{w})$ the set of isotopy classes of Euler structures on  $M \setminus \mathbf{w}$ compatible with the frame of the collar. This object  is a particular case of Framed Euler structures introduced in \cite{Boundary-States}.   
$\mathfrak{FrEul}(M \setminus \mathbf{w})$ is a torsor on $H_1(M \setminus \mathbf{w}, \Z)$. 


The obvious map 
\begin{equation} \label{framed-not-framed}
    \mathfrak{FrEul}(M \setminus \mathbf{w}) \rightarrow  \mathfrak{Eul}(M)
\end{equation}
 is compatible with the map $H_1(M \setminus \mathbf{w}, \Z) \rightarrow H_1(M, \Z)$.
The group $ \text{Ker} \{ H_1( \partial B_r(\mathbf{w}) , \Z) \rightarrow H_1(B_r(\mathbf{w}), \Z)   \} \cong \Z^{V^{\sharp}}$ acts transitively on the fibers of (\ref{framed-not-framed}).



\subsubsection{Euler Alignanents}

Given an Euler Structure $U$ on $M$, we now define a suitable deformation  of $U$ to an Euler Structure $U^{\ast}$ compatible with $\nu$ such that $U^{\ast}$ admits a preferred isotopy with $U$, canonically determinate up to isotopy of isotopies. Up to small isotopies $U^{\ast}$ is determinate by an extra geometric data which we now define.


\begin{definition}
Let $U$ be an Euler Structure. An Euler Alignment of an embedded oriented link for $  \mathbf{w}$ is an isotopy between the unit section of  $T \mathbf{w} $ and $U |_{\mathbf{w}}$ as sections of $TM |_{\mathbf{w}}$. 
\end{definition}

\begin{lemma}
The set of isotopy classes of Euler Alignment of $\mathbf{w}$  is a torsor on $\Z^{V^{\sharp}}$ .
\end{lemma}

\begin{remark} \label{remark-eurler-frame}
Fix a point $\mathbf{a}$ of the two-dimentional sphere  $S^2$.
It is easy to check that the isotopy classes of Euler Alignments are identified with the isotopy classes of isotopies between $\nu (\mathbf{a})$ and $U |_{\mathbf{w}}$.  Indeed any path of the two-dimentional sphere  $S^2$ between $(1,0,0)$  and $\mathbf{a}$   yields an isotopy between $\nu( \mathbf{a})$ and $T \mathbf{w}$. Any such isotopy yields a bijection 
$$ \{  \text{ isotopies between }T \mathbf{w} \text{ and } U |_{\mathbf{w}} \} /\sim \rightarrow   \{ \text{ isotopies between }\nu( \mathbf{a}) \text{ and } U |_{\mathbf{w}}  \}  /\sim $$
where in both the sides we quotient by the  isotopy of isotopies.
 Since $S^2$ is simply connected the bijection does not depend on the path we have chosen.  
\end{remark}









\begin{proposition} \label{euler-structure-compatible-lemma}
Let $( \mathbf{w}, \nu )$ be a framed link and  let $U^{\blacktriangle}$ be an Euler structure.

To each isotopy $\kappa$ between $T \mathbf{w} $ and $U |_{\mathbf{w}}$ it is associated an Euler Structure  $U^{\kappa}$ compatible with $( \mathbf{w}, \nu )$ with the property that
there is an  isotopy $\tilde{U}^{\kappa}$  between   $U^{\blacktriangle}$ and $U^{\kappa}$ canonically determinate up to isotopy of isotopies.

If $\kappa^1$ and $\kappa^2$ are isotopic, the associated Euler structures $U^{\kappa,1}, U^{\kappa,2}$ are isotopic with isotopy trivial outside $B_{r/2}(\mathbf{w}) $.


\end{proposition}
\begin{proof}

Using the frame of $B_r(\mathbf{w}) $ defined by (\ref{diffeo-tubular}), we can view $U$  on $B_r(\mathbf{w}) $ as a map  
$$U|_{B_r(\mathbf{w}) }: B_r(\mathbf{w})  \rightarrow S^2 .$$

 Using the frame $ \nu $ we may consider $\kappa$ as a map
$$\kappa : [0,1] \times \mathbf{w} \rightarrow S^2.$$

Denote by $D^2_r$ the disk of radius $r$ in $\R^2$.
Fix a map $f : D^2_r  \rightarrow  D^2_r$  with the following properties 
\begin{itemize}
 \item $f$ is the identity on $\partial D^2_r$ ;
 \item  $f(D^2_{r/2}) = \{ 0 \} \in D^2_{r/2}$;
 \item $f$ is injective on $ D^2_{r} \setminus    D^2_{r/2}$.
  \end{itemize}
 It is easy to check that  $f$ exists and it is uniquely determined up to isotopy.

Define $U^{\kappa}$ as follows:
\begin{itemize}
\item $U^{\kappa}=U$ outside $B_r(\mathbf{w}) $;
\item $U^{\kappa}=U \circ f$ on $B_r(\mathbf{w})  \setminus B_{r/2}(\mathbf{w}) $;
\item $U^{\kappa}=\kappa \circ (radius \times Id)$ on $B_{r/2}(\mathbf{w}) \cong V^{\sharp} \times D^2_{r/2}  \times S^1 $, where $radius : D^2_{r/2} \rightarrow [0,1]$ is the (suitable rescaled) radius function.
\end{itemize} 
 
Any isotopy between $f$ and the identity map of $D^2_r$ yields an isotopy between $U^{\kappa}$ and $U$.  This isotopy is uniquely determinate up to isotopy of isotopies. The Proposition follows.
 
\end{proof}

\subsubsection{Definition of Flat-Approching-$PS$-Cycle}

In the definition of  flat-approaching $PS$-cycles we formulate a weaker version of the convergence condition    (\ref{convergence}) that yields a well-defined normalized $PSH$-class  and that can be formulated using a frame of the link instead of a frame of $M$ .

 As above, fix a tubular neighborhood $ B_r(\mathbf{w}) $ of  $\mathbf{w}$ and endow $ B_r(\mathbf{w}) $ with a frame compatible with $( \mathbf{w}, \nu )$ as in (\ref{diffeo-tubular}).
 
Fix  a flat $PS$-cycle   $Z^{\flat}$ . 
 We say that  $\mathcal{Z}^{\wr}$ converges to $( \mathbf{w}, \nu  , Z^{\flat} )$ if  for each $(G,m,\mathcal{P})$  such that for each $R' < R^{\dagger}$
the convergence condition (\ref{convergence})  is satisfied uniformly for 
$\mathbf{z} = (z_v)_{v \in V^+(G)} \in Diag_{{\mathcal{P}}}$ with 
$$z_v \in  B_{R'}(\mathbf{w}) \quad \forall v \in  V^+(G) .$$

\subsubsection{Existence and Unicity}

We now extend Theorem  \ref{coherent-theorem} to normalized $PSH$.

\begin{theorem}
 Let $( \mathbf{w}, \nu ) $ be a framed link  and let $ U^{\blacktriangle}$ be
an Euler Structure compatible with $( \mathbf{w}, \nu ) $. To the triple $(\mathbf{w}, \nu  ,  U^{\blacktriangle})$  it is associated canonically a normalized $PSH$-class 
\begin{equation} \label{coherent-class}
[Z^{( \mathbf{w}, \nu  ,  U^{\blacktriangle}) }] \in PSH^{\intercal}(M| \mathbf{w},  U^{\blacktriangle} ).
\end{equation}

 Let $ \tilde{\mathbf{w}}$  be an  isotopy of embedded oriented links between  $\mathbf{w}^0, \mathbf{w}^1$ . 
  Let $\nu  ^0$ and $ \nu  ^1$ be frames of    $\mathbf{w}^0, \mathbf{w}^1$ and let $U^{\blacktriangle,0},U^{\blacktriangle,1} $ be Euler Structures compatible with  $\nu  ^0, \nu  ^1$ respectively.
Let  $\tilde{\nu  }$ be an isotopy between  $\nu  ^0$ and $ \nu  ^1$ and
let $\tilde{U}^{\blacktriangle}$ be an isotopy  between   $U^{\blacktriangle,0}$ and $U^{\blacktriangle,1} $ compatible with $\tilde{\nu  }$. We have
\begin{equation} \label{transfer-coherent-normalized}
Transfer_{ \tilde{\mathbf{w}}, \tilde{nu}, \tilde{U}^{\blacktriangle}}( [Z^{( \mathbf{w}^0, \nu  ^0,  U^{\blacktriangle,0}) }  ]  )   =   [Z^{( \mathbf{w}^1, \nu  ^1,  U^{\blacktriangle,1}) }   ].
\end{equation}
\end{theorem}


\begin{proof}





Pick a collection of small real positive numbers $ (R_{G,m}^{\dagger})_{G,m} $ with $ R_{G,m}^{\dagger} < R_{G',m'}^{\dagger} $ when $(G',m') \prec (G, m) $.




Fix a flat $PS$-cycle $Z^{\flat}$. 



As in Lemma \ref{coherent-existence-lemma} , to construct $ Z^{(\mathbf{w}, \nu, U^{\blacktriangle}) }$ we need to prove the existence of a flat-approaching $PS$-cycle $Z^{\wr}$ converging to $(Z^{\flat},\mathbf{w})$.

 To prove the existence of $Z^{\wr}$ we use an inductive argument as in the proof of Lemma \ref{extension-PS-lemma} and Proposition \ref{existence-asymptotic-proposition} (see \cite{PSPCS1}). 
  Assume we have defined $\underline{Z}^{\wr}_{\prec G}$.
 We can apply the argument of the proof of Lemma \ref{extension-PS-lemma} (see (\cite{PSPCS1}))  to obtain an extension $\underline{Z}^{\wr, \prime}_{G}$.  We may assume that $\underline{Z}^{\wr, \prime}_{G}$ satisfies the cluster condition (note that $\underline{Z}^{prel}$ used in the proof satisfies the cluster condition) but however $\underline{Z}^{\wr, \prime}_{G}$ does not satisfies in general the converge condition. 
 From another hand the argument of the proof  Proposition \ref{existence-asymptotic-proposition} yields  an extension $\underline{Z}^{\wr, tub}_{G}$ on $\mathfrak{L}_{\hat{V}(G)}(B_{R}(\mathbf{w})) $ such that the cluster and convergence conditions are satisfied.  
In order to satisfies the convergence condition for $\underline{Z}^{\wr, \prime}_{G}$ we want to modify it such that it agrees with  $\underline{Z}^{\wr, tub}_{G}$ on $\mathfrak{L}_{\hat{V}(G)}(B_{R}(\mathbf{w}))$.
In the case $E^{ex}(G) \neq \emptyset$ this step is trivial since we use forgetful compatibility. Let us assume $E^{ex}(G) = \emptyset$. 
 On $\mathfrak{L}_{\hat{V}(G)}(B_{R}(\mathbf{w}))$ we may define inductively chains  $\underline{Y}_G$ associated to   $\underline{Z}^{\wr, tub}_{G} - \underline{Z}^{\wr, \prime}_{G}$ as in (\ref{trivial-cycle}) .  We may inductively assume that $\underline{Y}_G$ satisfies the cluster condition. 
 We define the extension which satisfies the convergence condition using the formula
 \begin{equation} \label{extension-modified}
 \underline{Z}^{\wr}_{G,m}  = \underline{Z}^{\wr, \prime}_{G,m} + \partial^{red}  \underline{Y}_{G,m}  -   \sum_{0 \leq i \leq |m|} (-1)^i {\underline{Y}}_{G, \partial_i m}
 \end{equation}
 where we consider  the chains $ \underline{Y}_{G,m}$ as chains on $\mathfrak{L}_{\hat{V}(G)}(M)$, 
Observe that the convergence condition implies 
\begin{equation} \label{convergence-implication}
G \neq G^{\heartsuit}, H(G)= \emptyset,  V(G)=V_{\vartriangle}(G)   \Rightarrow \underline{Z}_G^{\wr, (n)} = 0   \text{   for   } n \gg 0.
\end{equation}



Now let  $Z^{0, \wr},Z^{1, \wr}$ be flat-approaching $PS$-cycles as above converging to  $(Z^{0, \flat},\mathbf{w}^0, \nu^0  ), (Z^{1, \flat},\mathbf{w}^1, \nu^1)$ respectively.
Let $\tilde{\mathbf{w}}, \tilde{\nu}$ and $\tilde{U}^{\blacktriangle}$ be isotopies as stated in the Theorem. 


Let  $ \tilde{Z}^{\flat}$ be an isotopy of flat $PS$-cycles between $Z^{0, \flat}$ and $Z^{1, \flat}$,
In order to prove identity  (\ref{transfer-coherent-normalized}) we need to define an isotopy of coherent $PS$-cycles  $\tilde{Z}^{(\tilde{\mathbf{w}}, \tilde{\nu},  \tilde{U}^{\blacktriangle}) }$ between $Z^{(\mathbf{w}^0, \nu^0  ,  U^{\blacktriangle,0}) } $ and $Z^{(\mathbf{w}^1, \nu^1,  U^{\blacktriangle,1}) } $.  For this we need to prove   the existence of an isotopy  $\tilde{Z}^{ \wr}$   between $Z^{0, \wr}$ and $Z^{1, \wr}$ converging to $(\tilde{Z}^{\flat}, \tilde{\mathbf{w}}, \tilde{\nu})$ (see Lemma \ref{coherent-iso-existence-lemma}). We proceed again inductively as follows.




Assume we have defined $\tilde{Z}^{ \wr}_{\prec G}$ . 
In the case $H(G) \ne \emptyset$, the same argument used for (\ref{extension-modified}) applied to isotopies yields $\tilde{\underline{Z}}^{ \wr}_{\preccurlyeq G}$  (here we use the proof of Lemma \ref{extension-PS-lemma-isotopy} and Proposition  \ref{unicity-asymptotic-proposition}).
Hence we  need to consider only the case  $H(G) = \emptyset$.

Consider an element 
$(w_v)_v$
 belonging to the support of 
$$\mathfrak{Obs}_G^{(n)} =  \underline{Z}^{1,(n)}_G -  \underline{Z}^{0,(n)}_G + pr_*( \sum_{(G', e')}  \delta_{e'} \tilde{\underline{Z}}_{G'}^{(n)}|_{(-\lambda^{\dagger}, \lambda^{\dagger})} +   \sum_{(G', v')} \delta_{v'} \tilde{\underline{Z}}_{G'}^{(n)}|_{(-\lambda^{\dagger}, \lambda^{\dagger})   }) .$$
(See subsection \ref{obstruction-section} for the notation.)

Set 
$$V_{0,tub}= \{ v \in V_0 (G) | supp (w_v) \subset B_{R}(\mathbf{w})  \}.$$

Assume first  $V_{0,tub} \neq V_0(G)$.
In this case we would like to apply the argument used to eliminate the obstructions of the proof of Proposition \ref{one-dimension}, i.e.,  
we pick an element $v_0 \in V(G) \setminus V^{\sharp}$ and and we shrink $w_{v_0}$.
However, in order to make the operation  compatible with the convergence condition, in this case we cannot pick $v_0$ completely arbitrarily but we make the following assumption:
\begin{equation} \label{pick-condition}
V_{0,tub} \neq \emptyset  \Rightarrow v_0 \in V_{0,tub}.
\end{equation} 




Now assume $V_{0,tub} = V_0(G)$.
Relation (\ref{convergence-implication}) tells us that the terms involving $\underline{Z}^{1,(n)}_{G} $ and $\underline{Z}^{0,(n)}_{G}$ are trivial. Hence we have to consider the following two cases:
\begin{itemize} 
\item $((w_v)_v) \in Supp(pr_*(\delta_{e'} \tilde{\underline{Z}}_{G'}) )$ for some  $(G',e')$ with $G'/e'=G$. Let $p \in Supp(pr_*(\tilde{\underline{Z}}_{G'}))$ be the corresponding element.  From the assumption we obtain $ p \in \mathfrak{L}_{\hat{V}(G)} (B_{R'}(\mathbf{w})) $. The convergences condition yields a contradiction. 
\item $(w_v)_v \in Supp(\delta_{v'} \tilde{\underline{Z}}_{G'} )$ for some $(G',v')$ with $G'/v'=G$. Here we have two subcases:
\begin{itemize}
\item $Supp(w_{v'}) \subset B_{R_{G}}(\mathbf{w})$:  let $p \in Supp(pr_*(\tilde{\underline{Z}}_{G', \emptyset}))$ be the corresponding element.   Again we have $ p \in  \mathfrak{L}_{\hat{V}(G)} (B_{R}(\mathbf{w}))$ and the convergences condition yields a contradiction; 
\item $Supp(w_{v'}) \nsubseteq B_{R_{G}}(\mathbf{w})$:  condition (\ref{pick-condition}) tells us that $V_{0,tub} =\emptyset$. So $V_0(G') = \{ v' \} $.  Since $G= \delta_{v'}G'$ we have $V_0(G)= \emptyset$ and $Comp_0(G) \neq \emptyset$. It follows that $G \in \mathfrak{G}^{nul}$.
\end{itemize}
\end{itemize}






 


\end{proof}

Now we consider Euler Structures not necessarily compatible with $(\mathbf{w}, \nu)$ . We use Proposition \ref{euler-structure-compatible-lemma}  and we set for any Euler Alignment $\kappa$
\begin{equation} \label{coherent-class2}
[Z^{(\mathbf{w}, \nu, \kappa) }] \in PSH^{\intercal}(M| \mathbf{w},  U^{\blacktriangle} )
\end{equation}
$$[Z^{(\mathbf{w}, \nu, \kappa) }] = [transfer_{\tilde{U}^{\blacktriangle}}(Z^{(\mathbf{w}, \nu, U^{\kappa} )})] .$$
Here $\tilde{U}^{\kappa}$ is the preferred isotopy between $U^{\kappa}$ and $U$ canonically defined up to isotopies of isotopies.

\subsection{Skein}

The usual definition of Skein involves oriented embedded links equipped  with an isotopy class of frames. In the context of this paper we need to consider also the extra geometric data given by an isotopy class of Euler Alignments and consider an extended version of the Skein relations.

In this section we show that the $PSH$-classes  (\ref{coherent-class})  satisfy an extended version of the Skein relations.
After including singular links we prove  that the Skein  is isomorphic to the normalized $PSH$.

 

\subsubsection{Extended Skein Relations}


We start defining a set of formal power series associated to relations of the $PSH$-classes  (\ref{coherent-class}). After we show that these formal power series are universals, that is do not depend on the geometric data we have fixed to define the  $PSH$-classes.

Let $\mathbf{w}_{\times}$ be an oriented link with exactly one crossing singularity.
Denote by $\mathbf{w}_+$ the embedded  link obtained smoothing the singularity by overcrossing , by $\mathbf{w}_-$ the embedded  link obtained smoothing the singularity by undercrossing and by $\mathbf{w}_0$ the embedded link obtained removing the crossing point in the only orientation-preserving way.


Let $\xi_{\times}$ be a section of $T M|_{\mathbf{w}_{\times}}$ never tangent to $\mathbf{w}_{\times}$ on the smooth points and not belonging to the plane of the crossing on the crossing point. 
Up to small isotopy, $\xi_{\times}$ admits canonical deformations to sections  $\xi_+,\xi_-,\xi_0$  of $T M|_{\mathbf{w}_{+}},T M|_{\mathbf{w}_{-}},T M|_{\mathbf{w}_{0}}$ respectively, never tangent to    $\mathbf{w}_+$, $\mathbf{w}_-$, $\mathbf{w}_0$. This process yield frames $\nu_+,\nu_-,\nu_0$  of $\mathbf{w}_+$, $\mathbf{w}_-$, $\mathbf{w}_0$ respectively.


Let $\kappa_{\times}'$ be an isotopy between  $v_{\times}$ and $U|_{\mathbf{w}_{\times}}$. 
Up to small isotopies, $\kappa_{\times}'$  can be deformed canonically  to isotopies $\kappa_{+}'$,
$\kappa_{-}'$, $\kappa_{0}'$ between  $v_+,v_-,v_0$ and $U|_{\mathbf{w}_{+}}, U|_{\mathbf{w}_{-}},U|_{\mathbf{w}_{0}}$.
Using remark \ref{remark-eurler-frame}   we obtain Euler Alignments  $\kappa_{+}$,$\kappa_{-}$, $\kappa_{0}$ for  $\mathbf{w}_+$, $\mathbf{w}_-$, $\mathbf{w}_0$ respectively.

\begin{lemma}   \label{skein-relation-lemma} 

There exist  two formal power series $A(\hbar,a), \beta(\hbar,a) \in \Q[[\hbar, a]]$ such that there exists a small isotopy of nice $PS$-cycles between
$\beta Z^{(\mathbf{w}_+ , \nu_+, \kappa_+)  } -  \beta^{-1} Z^{(\mathbf{w}_- , \nu_-, \kappa_-  )}$ and $ A  Z^{(\mathbf{w}_0 , \nu_0,  \kappa_0 ) }$, canonical up to isotopy of isotopies.


The leading terms of $A$ and $\beta$ are given by
$$ \beta = 1 + .... , \thickspace A = \hbar(1 + ....) $$  
and the following identities hold
\begin{equation} \label{simmetry-opposite}
\beta(-\hbar,a)  = \beta(\hbar,a)^{-1},  \bigskip A(-\hbar,a) = - A(\hbar,a).
\end{equation}
\end{lemma}
\begin{proof}


From Proposition \ref{crossing-dimension-normalized} we know   that $  PSH(M, \mathbf{w}_{\times})$ is a rank two module over $\Q[[\hbar, a]]$. Moreover
$  Z^{(\mathbf{w}_- , \nu_-, \kappa_-  )}, Z^{(\mathbf{w}_0 , \nu_0,  \kappa_0 ) }$
is a pair of generators. It follows that there exist two formal power series $p,q$ such that
$$Z^{(\mathbf{w} , \nu_+, \kappa_+)  } =  p Z^{(\mathbf{w} , \nu_-, \kappa_-  )} + q  Z^{(\mathbf{w} , \nu_0,  \kappa_0 ) }$$
with
$$p = 1 + .... $$
$$  q = \hbar(1 + ....) .$$
The first relation follows from the fact that $\mathbf{w}_+$ and $\mathbf{w}_-$ are isotopic as multi-loops.
The second relation follows considering  the graph obtained adding an edge to $G^{\heartsuit}$.


Set 
$$\beta = p^{\frac{1}{2}}, \thickspace A = \beta^{-1}q.$$ 



\end{proof}

Denote by $\nu^{+1}$  the frame of $\mathbf{w}$ obtained twisting a frame $\nu$ by $+1$. 

\begin{lemma}   \label{twist-frame-lemma} 
There exists a formal power series $\alpha(\hbar, a) \in \Q[[\hbar, a]] $ such that
\begin{equation}
[ Z^{(\mathbf{w} , \nu^{+1},  \kappa )}] = \alpha(\hbar,a)  [Z^{(\mathbf{w} , \nu,  \kappa ) }].
\end{equation}
The leading term is given by
$$ \alpha = 1+ . . .    \quad  . $$

\end{lemma}
\begin{proof}
The existence of $\alpha$ follows from Proposition \ref{one-dimension-normalized}   taking $Z^{\heartsuit}= Z^{\mathbf{w} , \nu,  \kappa  }$. 


\end{proof}

Denote by $\kappa^{ +1}$ the twisting of the Euler Alignment $\kappa$.

\begin{lemma}  \label{lemma-twist-kappa}
    There exists an  formal power series $\theta(\hbar, a)   \in \Q[[\hbar, a]]  $ such that 
    $$[Z^{(\mathbf{w} , \nu, \kappa^{ +1}) }] = \theta(\hbar ,a ) [Z^{(\mathbf{w} , \nu, \kappa )} ].$$   
    The leading term is given by
    $$\theta = \hbar a +  . . .    \quad $$
\end{lemma}
\begin{proof}

As before the existence of $\theta$ follows from Proposition \ref{one-dimension-normalized}   taking $Z^{\heartsuit}= Z^{\mathbf{w} , \nu,  \kappa }$. 

\end{proof}

Let $B$ a small ball in $M$.   We fix and  identification of $B$ with $\R^3$.  
Let  $\mathbf{w}_{unknot}$ be a circle in $\R^2 $ with center in the origin. 
$\mathbf{w}_{unknot}$ admits a frame $\nu_{unknot}$ canonical up to isotopies ( let $\nu_{unknot} ((0,1,0))$ be defined by the radial vectors pointing toward the origin.)

 Use the identification $B= \R^3$  to  trivialize the tangent bundle of $B$.   Taking $B$ small enough, we may assume that  the Euler structure $U$ is close  to a not zero vector $\xi_{eul} \in \R^3$ everywhere in $B$.
Since $ \nu_{unknot} ((0,0,1))$ is a constant section we have in particular an isotopy between $ \nu_{unknot} ((0,0,1))$ and $U$ canonical up to isotopy of isotopies. 
 Using remark \ref{remark-eurler-frame}   we obtain an Euler Allignament $\kappa_{unknot}$ .




\begin{lemma} \label{lemma-unknot}
Assume that the unknot $\mathbf{w}_{unknot}$ is unlinked with the link $\mathbf{w}$.

There exists a  formal power series $r(\hbar, a) \in \Q[[\hbar,a]]$ such that  
$Z^{\mathbf{w} \sqcup \mathbf{w}_{unknot} ,\nu \sqcup \nu_{unknot},  \kappa \sqcup \kappa_{unknot} } $ and $r(\hbar, a) Z^{\mathbf{w}, \nu  , \kappa } $ are small isotopic as nice $PS$-cycles.
(In the case of $\mathbf{w}$ is the empty link,   we let $ Z^{\mathbf{w}, \nu  , \kappa } $ be the canonical generator of $PSH^{\intercal}(M| \mathbf{w}_{\empty},   U^{\blacktriangle}  )$ given in Lemma \ref{lemma-canonical-one-dimensional}.)

The leading term of $r(\hbar, a)$ is given by 
$$r = 1+  . . .    \quad $$

\end{lemma}

\begin{proof}

Consider an isotopy $\tilde{\mathbf{w}}$ with support in $B$
 contracting $\mathbf{w}_{unknot}$ to a point . The same argument of the proof of Lemma \ref{deformation-cycles}   
 yields a $PSH$-class
$$ transfer_{\tilde{\mathbf{w}}} (Z^{\mathbf{w} \sqcup \mathbf{w}_{unknot} ,\nu \sqcup \nu_{unknot},  \kappa \sqcup \kappa_{unknot} }) \in PSH^{\intercal}(M|  \mathbf{w}  , U^{\blacktriangle}   )$$
and so defines a  $r(\hbar, a) \in \Q[[\hbar,a]]$ since $Z^{\mathbf{w}, \nu  , \kappa }$  is a generator of $PSH^{\intercal}(M|  \mathbf{w}  , U^{\blacktriangle}   )$ as  $ \Q[[\hbar,a]]$  module.

$r(\hbar, a) $ does not depend on $\nu$ and $\kappa$ since the transfer map is equivariant with respect to the twist of $\nu$ and $\kappa$.

The fact 
that  $r(\hbar, a)$ does not depend on the embedding of $\mathbf{w}_{unknot}$ and on  $\tilde{\mathbf{w}}$ follows considering isotopies of isotopies.



\end{proof}

We claim that the formal power series appearing in  Lemmas   \ref{skein-relation-lemma}, \ref{twist-frame-lemma},  \ref{lemma-twist-kappa}, \ref{lemma-unknot} are universal, i.e.,  they do not depend on the geometric data of the problem, such as  the manifold $M$, the links and the Euler Structure.
Below we show this claim using a gluing argument based on the results of  \cite{PSPCS3}.
Let us recall briefly the results of  \cite{PSPCS3} about $PS$-homology on Open Manifolds and Gluing of open Coherent $PS$-cycles.

First we consider the cylinder  $\Sigma \times \R$, where $\Sigma$ is a compact oriented surfaces. We fix a finite subset  of points $Punct \subset \Sigma$. We fix an Euler Structure  $ U^{\blacktriangle, cyl}$ invariant by translation in the $\R$-factor.  Using the invariance in the $\R$-direction we define cylinder-$PS$-cycles and  we show that there exists a  special cylinder-$PS$-cycle $Z^{\spadesuit}$ canonically defined up to isotopies.

Now let $M^{open}$ be an open manifold equipped with a collar at infinity $\Sigma \times \R^+$.  
We consider Euler Structures  $ U^{\blacktriangle, open}$ extending $ U^{\blacktriangle, cyl}$. We consider open links whose restriction to $\Sigma \times \R^+$ agrees with $Punct \times \R^+$.  We consider frames of open links  invariant by translation in the $\R^+$-factor.
To  $Z^{\spadesuit}$ we associate a boundary condition for $PS$-cycles  and using this we define open-$PSH$ which we denote as  $PSH(M^{open}| Z^{\spadesuit}, U^{\blacktriangle, open})$.





We also need to consider the case where  the boundary of $M^{open}$ is endowed with opposite orientation, i.e., we have a collar $-\Sigma \times \R^+$. In this situation we define anti-boundary conditions defined by $Z^{\spadesuit}$.

Now, let $M$ be a compact manifold  obtained as gluing of  two oriented open manifolds $M^{in}, M^{out}$ with $\partial M^{in}= -\partial M^{out}= \Sigma$:
\begin{equation} \label{gluing-manifold-formula}
M = M^{in} \sqcup_{\Sigma \times [0,1]} M^{out} .
\end{equation} 

Assume that the Euler Structure  $U^{\blacktriangle}$ is compatible with (\ref{gluing-manifold-formula}) in the sense that  it  is the gluing of open Euler structures of $M^{in}$ and $ M^{out}$:
$$  U^{\blacktriangle}=U^{\blacktriangle,in} \sqcup_{U^{\blacktriangle,neck}}  U^{\blacktriangle,out} .$$

In \cite{PSPCS3} we define the gluing map of open-$PSH$-classes
\begin{equation} \label{gluing-classes}
PSH(M^{in}| Z^{\spadesuit}, U^{\blacktriangle,in} )  \times PSH(M^{out}| Z^{\spadesuit}, U^{\blacktriangle,out} ) \rightarrow PSH(M| U^{\blacktriangle}),
\end{equation}
$$ [ Z^{in}] \times   [ Z^{out}]  \mapsto   [ Z^{in}]   \Cup_{Z^{\spadesuit}}   [ Z^{out}] .$$ 
The product is compatible with the action of $\Q[[g_s,a]]$.

Next we extend the notion of Coherent Cycle in the context of open manifolds. For this we define  Open-flat-$PS$-cycles are adapting the notion of boundary condition. 
An important difference with the closed case is that there is no a special class of isotopies of open-flat-$PS$-cycle (this is a consequence of the incompatibility between the boundary condition and the translation invariance). So to define a Coherent $PSH$-class we need to fix a open flat $PS$-cycle $Z^{\flat,open}$.  We can then associate to any framed open link $(\mathbf{w}^{open}, \nu^{open})$ compatible with $U^{\blacktriangle,open}$ a normalized-open-$PSH$-class 
$$ [ Z^{(\mathbf{w}^{open}, \nu^{open}, U^{\blacktriangle,open})}] \in PSH^{\intercal}(M^{open}| Z^{\spadesuit}, U^{\blacktriangle,open}).$$

 
In \cite{PSPCS3} we discuss the gluing of  Coherent Cycles. In the context of (\ref{gluing-manifold-formula}) let $\mathbf{w}$ be a link on $M$ which gluing of two  open links  :
\begin{equation} \label{gluing-links}
 \mathbf{w}=\mathbf{w}^{in} \sqcup_{Punct \times [0,1]}  \mathbf{w}^{out} .
 \end{equation}
Equip $\mathbf{w}$ with a frame  $\nu$ which is the gluing of  open frames  $\nu^{in} , \nu^{out} $:
\begin{equation} \label{gluing-framed-links}
 \nu=\nu^{in} \sqcup_{\nu^{neck}}  \nu^{out} .
 \end{equation}

 For our purpose it is important to compare the $PSH$-class obtained making the gluing (\ref{gluing-classes}) of two open coherent $PSH$-classes with the canonical coherent $PSH$ defined for closed manifolds.  
For this in \cite{PSPCS3} it is introduced the notion of gluing-flat-$PS$-cycle.  There is natural map between the space of flat-$PS$-cycles considered in the preview sections and the space of gluing-flat-$PS$-cycles which yields an equivalence on the isotopies classes:
 $$ \mathcal{Z}^{\flat, gluing}/isotopies =   \mathcal{Z}^{\flat}/isotopies = \Q[[g_s,a]] .$$
The construction of a $PSH$-class associate to a flat-$PS$-cycle extends to any gluing-flat-$PS$-cycle. This correspondence is compatible with the action of $\Q[[g_s,a]]$.
 
 Analogously to (\ref{gluing-classes}) we may consider the gluing of open flat $PS$-cycles with the opposite boundary conditions and obtain a gluing flat $PS$-cycle;
  \begin{equation} \label{gluing-flat}
Z^{\flat, gluing} =  Z^{\flat, in} \Cup_{Z^{\flat, \spadesuit}} Z^{\flat, out}, 
 \end{equation} 
 
 We have the following fundamental 
 \begin{lemma} (\cite{PSPCS3})  \label{gluing-lemma-coherent}
 Use  $Z^{\flat, in}, Z^{\flat, out}$ to define Coherent $PS$-Cycles $Z^{(\mathbf{w}^{in}, \nu^{in}, U^{\blacktriangle,out})}, Z^{(\mathbf{w}^{out}, \nu^{out}, U^{\blacktriangle, out})}$. The gluing  
  \begin{equation} \label{gluing-coherent}
     [ Z^{(\mathbf{w}^{in}, \nu^{in}, U^{\blacktriangle,out})}]  \Cup_{Z^{\spadesuit}} [Z^{(\mathbf{w}^{out}, \nu^{out}, U^{\blacktriangle, out})}] \in   PSH(M| U^{\blacktriangle})
       \end{equation}  
 agrees with the $PSH$ class defined using $Z^{\flat, gluing} $.
In particular (\ref{gluing-coherent}) differs from the canonical  coherent $PSH$-class  $  [ Z^{(\mathbf{w}, \nu, U^{\blacktriangle})}]  $  by the element of   $\Q[[g_s,a]]$ associated to (\ref{gluing-flat}).
 \end{lemma}
 The Lemma follows easily from the constructions.  For our purpuse we need to consider frames not necessary compatible with $U^{\blacktriangle}$ and Euler Alignments.

First we fix an isotopy $\kappa^{cyl}$ on $Punct \times \R$   between $T (Punct \times \R)$ and $U$  invariant under translations in the $\R$ factor.
 Using the transfer map we associate to $\kappa^{cyl}$ a boundary condition on open $PS$-cycles.



Given an open Manifolds $M^{open}$ and an open link $\mathbf{w}^{open}$ ,  we consider an  Euler Alignment $\kappa^{open}$ for $\mathbf{w}^{open}$  which agrees with $\kappa^{cyl}$ on the collar $Punct \times \R^+$.
Analogously to (\ref{coherent-class2}) we obtain
$$[Z^{(\mathbf{w}^{open}, \nu^{out}, \kappa^{open}) }] \in PSH^{\intercal}(M| \mathbf{w},  U^{\blacktriangle}, Z^{\spadesuit} ). $$

Now consider a compact manifold $M$ described as gluing of two manifolds as in (\ref{gluing-manifold-formula}).   Let  $\kappa$ be an Euler Alignment compatible with the gluing:
$$ \kappa = \kappa^{in} \sqcup_{\kappa^{cyl}} \kappa^{out} .$$
  Lemma (\ref{gluing-lemma-coherent}) tells as that  
\begin{equation} \label{gluing-coherent-kappa}
 [ Z^{(\mathbf{w}^{in}, \nu^{in},  \kappa^{in})}]  \Cup_{Z^{\spadesuit}} [Z^{(\mathbf{w}^{out}, \nu^{out},  \kappa^{out})}]. 
  \end{equation} 
 agrees with the $PSH$ class associated to $ (\mathbf{w} , \nu, \kappa )$ using $Z^{\flat, gluing} $.

After this background we can prove the following  

\begin{proposition}
The formal power series appearing in Lemmas   \ref{skein-relation-lemma}, \ref{twist-frame-lemma}, \ref{lemma-twist-kappa},\ref{lemma-unknot} are universal.
\end{proposition}
\begin{proof}

Up to isotopy, we may assume that each framed link considered in the Lemmas \ref{skein-relation-lemma}, \ref{twist-frame-lemma}, \ref{lemma-twist-kappa}, \ref{lemma-unknot}  is of the type (\ref{gluing-framed-links}), with $M^{in}$ given by a small ball $B$.

In the context of 
Lemmas   \ref{skein-relation-lemma} we take for $B$ a small ball surrounding the crossing singularity . 
We assume that there exists a triple $ (\mathbf{w}^{out}, \nu^{out}, \kappa^{out})$ on $M^{out}$ such that

$$  \mathbf{w}_+=\mathbf{w}_+^{in} \sqcup_{\mathbf{w}^{neck}}  \mathbf{w}^{out} ,   \mathbf{w}_-=\mathbf{w}_-^{in} \sqcup_{\mathbf{w}^{neck}}  \mathbf{w}^{out} ,   \mathbf{w}_0=\mathbf{w}_0^{in} \sqcup_{\mathbf{w}^{neck}}  \mathbf{w}^{out}.$$

$$  \nu_+=\nu_+^{in} \sqcup_{\nu^{neck}}  \nu^{out} ,   \nu_-=\nu_-^{in} \sqcup_{\nu^{neck}}  \nu^{out} ,   \nu_0= \nu_0^{in} \sqcup_{ \nu^{neck} } \nu^{out}$$

$$  \kappa_+=\kappa_+^{in} \sqcup_{\kappa^{neck}}  \kappa^{out} ,   \kappa_-=\kappa_-^{in} \sqcup_{\kappa^{neck}}  \kappa^{out} ,   \kappa_0=\kappa_0^{in} \sqcup_{\kappa^{neck}}  \kappa^{out}.$$


In the context 
of Lemma  \ref{twist-frame-lemma},  we assume that $\nu$ and $\nu^{+1}$ agrees on $M^{out}$.

In the context of
 Lemma  \ref{lemma-twist-kappa},  we assume that $\kappa$ and $\kappa^{+1}$ agree outside a small ball $M^{out}$.

In the context of 
Lemma  \ref{lemma-unknot} we take the support of knot  contained on $B$.




Lemmas   \ref{skein-relation-lemma}, \ref{twist-frame-lemma}, \ref{lemma-twist-kappa}, \ref{lemma-unknot} can be easily extended  to open manifolds. We claim that relation ( \ref{gluing-coherent-kappa}) implies that the formal power series obtained from the links of $M$ coincide with the formal power series obtained from the corresponding links of $M^{in}$.
For the relations in Lemma  \ref{twist-frame-lemma}, \ref{lemma-twist-kappa}, \ref{lemma-unknot}   the claim follows immediately from Lemma  \ref{gluing-lemma-coherent} and  (\ref{gluing-coherent-kappa}) since for each relation the involved gluing coherent $PS$-cycles  are defined using the same gluing-flat-$PS$-cycle  (\ref{gluing-flat}). 
For the relation in Lemma \ref{skein-relation-lemma} the gluing-flat-$PS$-cycles (\ref{gluing-flat}) are the same for the case $+$ and $-$ but not for the case $0$, and so we have two elements of $\Q[[g_s,a]]$. Using the action $\Q[[g_s,a]]$ it is easy to see that the difference between them does not depend on $Z^{\flat, in}, Z^{\flat, out}$ and therefore it must be zero. Hence the claim follows as for the other relations.



\end{proof}

Let $Skein(M)[[\hbar,a]]^+$ be the set of formal power series in the formal variables $\hbar,a$ with coefficients triples
\begin{equation} \label{skein-generator}
(\mathbf{w} , \nu, \kappa)
\end{equation}  
as in Lemma  \ref{euler-structure-compatible-lemma},
modulo the relations 
\begin{equation} \label{crossing-skein}
 \beta(\hbar,a) (\mathbf{w}_+ , \nu_+, \kappa_+  )-  \beta(\hbar,a)^{-1} (\mathbf{w}_- , \nu_-, \kappa_-  ) \sim A(\hbar,a) (\mathbf{w}_0 , \nu_0, \kappa_0 )
 \end{equation}
\begin{equation} \label{twisting-skein}
 (\mathbf{w} , \nu^{+1}, \kappa  ) \sim \alpha(\hbar,a) (\mathbf{w} , \nu, \kappa )
 \end{equation} 
 \begin{equation} \label{shift-euler-skein}
  (\mathbf{w} , \nu, \kappa^{ +1} ) \sim  \theta(\hbar ,a ) (\mathbf{w} , \nu, \kappa )
  \end{equation}   
 \begin{equation} \label{unknot-skein}
( \mathbf{w} \sqcup  \mathbf{w}_{unknot}  , \nu \sqcup \nu_{unknot} , \kappa \sqcup \kappa_{unknot}     ) \sim   r(\hbar, a)( \mathbf{w} , \nu  , \kappa  ).
  \end{equation} 
The coefficients appearing in the relations are not independent but they are related by the formula:
$$A(\hbar,a)    r(\hbar, a)=   \beta(\hbar,a)  \alpha(\hbar,a)  \theta(\hbar ,a ) - (   \beta(\hbar,a)  \alpha(\hbar,a)  \theta(\hbar ,a ))^{-1} .$$
A similar relation is satisfied in the standard Skein relations where the Euler Structure is not considered and so the coefficient $ \theta(\hbar,a)$ does not appear. The proof in our situation uses the same argument of the standard case, the coefficient  $ \theta(\hbar ,a ) $ appears  with a similar argument used for the coefficient  $ \alpha(\hbar,a)$.

Define a map of $\Q[[\hbar,a]]$-moduli
\begin{equation} \label{skein-map-smooth}
 Skein(M)[[\hbar,a]]^+  \rightarrow PSH^{\intercal}(M| U^{\blacktriangle}) .
 \end{equation}
 $$  (\mathbf{w} , \nu, \kappa ) \mapsto [  Z^{ (\mathbf{w} , \nu, \kappa )}].$$
The Lemmas above tell us that (\ref{skein-map-smooth}) is well defined and it is compatible with the isotopies of both the sides.

\subsubsection{Singular Skein}

The map (\ref{skein-map-smooth}) is injective but not surjective.  However it is surjective up to small isotopies, that is good enough for the applications.   In order to obtain an actual isomorphism in this section we extend  $Skein(M)[[\hbar,a]]^+$ to include  singular links in a suitable sense. 


Consider  arrays
\begin{equation} \label{framed-smoothing}
    (\mathbf{w}^{\ast} , \nu^{\ast},  \kappa^{\ast}, \mathfrak{w})   
\end{equation}
where
$(\mathbf{w}^{\ast} , \nu^{\ast},  \kappa^{\ast} )$ is a triple as in (\ref{skein-generator}),  $\mathfrak{w}$ is a finite set of multi-loops closed by twist relation as in Subsection \ref{singular-subsection-PSH}.  
 We require that the oriented link $\mathbf{w}^{\ast}$ is close in the $C^0$-topology to an element  $\mathbf{w} \in \mathfrak{w}$ .   In the special case $ \mathfrak{w} = \emptyset$ we assume $\mathbf{w}^{\ast}$ is close in the $C^0$-topology to a constant loop.

We consider two arrays (\ref{framed-smoothing}) equivalent
\begin{equation} 
\label{isotopy-skein-equivalence-singular}
( \mathbf{w}^{\ast}, \nu^{\ast},  \kappa^{\ast},  \mathfrak{w}) \sim   ( \mathbf{w}^{\ast, \prime}, \nu^{\ast, \prime},  \kappa^{\ast, \prime}, \mathfrak{w})  
 \end{equation}
if $\mathbf{w}^{\ast}$ is isotopic to $\mathbf{w}^{\ast, \prime}$ with isotopy close in the $C^0$-topology to an element  $\mathbf{w}$ of   $ \mathfrak{w}$, $\nu^{\ast}$ is isotopic to $ \nu^{\ast, \prime}$ and $ \kappa^{\ast, \prime}$ is isotopic to $ \kappa^{\ast}$.


Define the singular Skein $ \overline{Skein}(M)[[\hbar,a]]^+$ as  the space of formal power series with coefficients on the arrays 
(\ref{framed-smoothing}) modulo   the following  equivalence relations: 
\begin{itemize}
\item
\begin{equation} \label{crossing-skein-equivalence-singular}
  \beta (  \mathbf{w}^{\ast}_+,  \nu_+^{\ast}, \kappa^{\ast}_+,\mathfrak{w})-  \beta^{-1} ( \mathbf{w}^{\ast}_- ,  \nu_-^{\ast}, \kappa^{\ast}_- , \mathfrak{w}) \sim A (\mathbf{w}^{\ast}_0 , \nu_0^{\ast}, \kappa^{ \ast}_0,  \mathfrak{w})
 \end{equation}
 where $ (\mathbf{w}^{\ast}_+, \nu_+^{\ast}, \kappa^{\ast}_+ ), (\mathbf{w}^{\ast}_-,  \nu_-^{\ast}, \kappa^{\ast}_- ) ,  ( \mathbf{w}^{\ast}_0  , \nu_0^{\ast}, \kappa^{ \ast}_0 )$ are related as in Lemma \ref{skein-relation-lemma} with  $\mathbf{w}_{\times}^{\ast}$  close (compatibly with the crossing singularity) in the $C^0$-topology to an element  $\mathbf{w}_{\times}$ of   $ \mathfrak{w}$ ;
 \item
 \begin{equation} \label{twisting-skein-equivalence-singular}
 ( \mathbf{w}^{\ast}, \nu^{\ast, +1}, \kappa^{\ast}, \mathfrak{w}) \sim \alpha(\hbar,a) ( \mathbf{w}^{\ast}, \nu^{\ast}, \kappa^{ \ast}, \mathfrak{w} );
 \end{equation} 
\item
\begin{equation} \label{shift-euler-skein-equivalence-singular}
  ( \mathbf{w}^{\ast}, \nu^{\ast}, \kappa^{ \ast, +1} ,  \mathfrak{w}) \sim  \theta(\hbar ,a ) ( \mathbf{w}^{\ast}, \nu^{\ast}, \kappa^{\ast},   \mathfrak{w});
  \end{equation}  
  \item
 \begin{equation} \label{unknot-skein-equivalence-singular}
  (  \mathbf{w}^{\ast} \sqcup \mathbf{w}_{unknot}, \nu^{\ast} \sqcup \nu_{unknot} , \kappa^{\ast} \sqcup \kappa_{unknot} , \mathfrak{w}) \sim   r(\hbar, a)(  \mathbf{w}^{\ast}, \nu^{\ast}  , \kappa^{\ast},  \mathfrak{w} ) .
  \end{equation}
 \end{itemize}
Above we have used the obvious extensions of the notations of   Lemmas    \ref{skein-relation-lemma}, \ref{twist-frame-lemma}, \ref{lemma-twist-kappa}, \ref{lemma-unknot}.

Isotopies of $ \overline{Skein}(M)[[\hbar,a]]^+$ are defined straightforwardly considering deformations of $\mathfrak{w}$ in (\ref{framed-smoothing}).

\begin{lemma}
There exists a natural map of $\Q[[\hbar,a]]$-moduli
\begin{equation} \label{singular-skein-map}
\overline{Skein}(M)[[\hbar,a]]^+ \rightarrow PSH^{\intercal}(M| U^{\blacktriangle}).
\end{equation}
\end{lemma}
\begin{proof}
The map  (\ref{singular-skein-map}) is defined applying the transfer map  (\ref{transfer-map-loops}) to the coherent cycle associate to $ ( \mathbf{w}^{\ast}, \nu^{\ast}, \kappa^{\ast})$:
\begin{equation} \label{skein-map-singular}
  ( \mathbf{w}^{\ast}, \nu^{\ast},  \kappa^{\ast}, \mathfrak{w}) \mapsto transfer_{\tilde{\mathbf{w}}^{\ast}} (  Z^{( \mathbf{w}^{\ast}, \nu^{\ast}, \kappa^{\ast})}) 
 \end{equation}
where $\tilde{\mathbf{w}}^{\ast}$ is a small isotopy  between  $\mathbf{w}^{\ast}$ and an element $ \mathbf{w} \in \mathfrak{w}$ .
The map (\ref{skein-map-singular}) is well defined since $\tilde{\mathbf{w}}^{\ast}$ is uniquely determined up to isotopy of isotopies.

Lemmas  \ref{skein-relation-lemma}, \ref{twist-frame-lemma}, \ref{lemma-twist-kappa}, \ref{lemma-unknot}   tell us  that the map ( \ref{skein-map-singular}) is compatible with the equivalence relations (\ref{isotopy-skein-equivalence-singular}), (\ref{crossing-skein-equivalence-singular}), (\ref{twisting-skein-equivalence-singular}), (\ref{unknot-skein-equivalence-singular}).
In the case of  (\ref{crossing-skein-equivalence-singular}) we  apply the transfer map to  the singular link $\mathbf{w}_{\times}^{\ast}$. 
  
\end{proof}

\begin{theorem}
The map (\ref{singular-skein-map}) is an isomorphism compatible with isotopies.
\end{theorem}
\begin{proof}
The surjectivity follows from Lemma    \ref{lemma-smoothing-cycles}.

To prove injectivity consider  two equivalent nice $PS$-cycles $Z^{\diamond,0}$ and $Z^{\diamond,1}$. And let $Z^{ \diamond,0, \ast}, Z^{\diamond, 1, \ast}$ be nice $PS$-cycles related  to $Z^{\diamond,0}$ and $Z^{\diamond,1}$ as in Lemma    \ref{lemma-smoothing-cycles}  .  
$Z^{ \diamond, 0, \ast}$ and $ Z^{\diamond, 1, \ast}$ may be seen as elements of $\overline{Skein}(M)[[\hbar,a]]^+$.

Lemma \ref{lemma-isotopy-smoothing} applied to $Z^{\diamond, 0, \ast}, Z^{\diamond, 1,\ast}$ tells us that there exists an isotopy $\tilde{Z}^{\diamond, \ast}$ between $Z^{\diamond, 0, \ast}, Z^{\diamond, 1,\ast}$. The isotopy of nice $PS$-cycles $\tilde{Z}^{\diamond, \ast}$ defines an isotopy in $\overline{Skein}(M)[[\hbar,a]]^+$ between $Z^{ \diamond, 0, \ast}$ and $ Z^{ \diamond, 1, \ast}$.
\end{proof}




\subsubsection{Smoothing the limit-support}


\begin{lemma} \label{lemma-smoothing-cycles}
For each nice $PS$-cycle $Z^{\diamond}$ there exists a nice $PS$-cycle $Z^{\diamond, \ast}$, and an isotopy $\tilde{Z}^{\diamond}$ between $Z^{\diamond}$ and $Z^{\diamond, \ast}$ such that 
\begin{itemize}
\item $\lim supp({Z}^{\diamond, \ast})$ consists of embedded links; 
\item $\lim supp(\tilde{Z}^{\diamond})$ is close in $C^0$-topology to $\lim supp (Z^{\diamond})$.
\end{itemize}
\end{lemma}

\begin{proof}

Assume we have defined $Z^{\diamond, \ast}_{\prec G}$ and $\tilde{Z}^{\diamond}_{\prec G}$.



Consider  first the case $H(G) = \emptyset$. For each $\mathbf{w} \in \lim supp (Z^{\diamond}_{G})$ take an isotopy $\tilde{\mathbf{w}}^{\ast}= (\mathbf{w}^s)_{s \in [0,1]} $ of multi-loops, such that 
\begin{itemize}
\item  $\mathbf{w}^0 =\mathbf{w} $; 
\item $\mathbf{w}^s$ is  close in the $C^0$-topology to $\mathbf{w}$ for each $s \in [0,1]$;
\item $\mathbf{w}^s$ is an embedded link for  each $s \in (0,1]$.
\end{itemize}
Let $\tilde{Supp}^{\ast}$ be the set of the isotopies $\tilde{\mathbf{w}}^{\ast}$ obtained in this way .
 
Using $\tilde{Supp}^{\ast}$ we can define  $Z^{\diamond, \ast, (n)}_{ G}$ and $\tilde{Z}^{\diamond, (n)}_{ G}$ for $n$ big enough such that
$$ \lim \tilde{supp} (\tilde{Z}^{\diamond}_{G})= \tilde{Supp}^{\ast} .$$


Now consider the case $H(G) \neq \emptyset$. Use Lemma  \ref{extension-PS-lemma} to define $Z^{\diamond, \ast}_{ G}$ and Lemma  \ref{extension-PS-lemma-isotopy} to define $\tilde{Z}^{\diamond}_{ G}$. The proof of the Lemmas tells us that the chains can be taken  
 such that 
 $$\lim supp (Z^{\diamond, \ast}_{\prec G})= \lim supp (Z^{\diamond, \ast}_{\preccurlyeq G}) , \lim \tilde{supp} (\tilde{Z}^{\diamond}_{\prec G})= \lim \tilde{supp} (\tilde{Z}^{\diamond }_{\preccurlyeq G}).$$



\end{proof}

\begin{lemma} \label{lemma-isotopy-smoothing}
Let $Z^{\diamond,0}, Z^{\diamond,1}$ be nice $PS$-cycles. Let  $Z^{\diamond , 0, \ast} , Z^{\diamond, 1, \ast}$ be nice $PS$-cycles related to  $Z^{\diamond, 0}, Z^{\diamond, 1}$  as in Lemma \ref{lemma-smoothing-cycles} . For each isotopy $\tilde{Z}^{\diamond}$ between $Z^{\diamond, 0}$ and $ Z^{\diamond, 1}$, there exists an isotopy  $\tilde{Z}^{\diamond, \ast}$ between   $Z^{\diamond, 0, \ast}$ and $ Z^{\diamond, 1, \ast}$ such that $\lim \tilde{supp}({\tilde{Z}^{\diamond, \ast}})$
is the union of a finite set of  one-parameter families of multi-loops  (\ref{basic-family-loops}) which are  embedded links crossing transversally and close in  $C^0$-topology to  $\lim \tilde{supp}(\tilde{Z}^{\diamond})$.


\end{lemma}
\begin{proof}


Using induction on the graphs, we shall define $\tilde{Z}^{\diamond, \ast}$ and an isotopy of isotopies $\tilde{\tilde{Z}}^{\diamond}$ between 
  $\tilde{Z}^{\diamond}$ ,
 $\tilde{Z}^{\diamond,0, \ast}$,
  $\tilde{Z}^{\diamond,1 ,\ast }$ ,
  $\tilde{Z}^{\diamond, \ast}$.

 Assume we have defined $\tilde{Z}^{\diamond, \ast}_{\prec G}$ and an isotopy of isotopies $\tilde{\tilde{Z}}^{\diamond}_{\prec G}$ between 
  $\tilde{Z}^{\diamond}_{\prec G}$ ,
 $\tilde{Z}^{\diamond,0, \ast}_{\prec G}$,
  $\tilde{Z}^{\diamond,1, \ast }_{\prec G}$ ,
  $\tilde{Z}^{\diamond, \ast}_{\prec G}$.


 Consider  first the case $H(G) = \emptyset$. Let $\tilde{Supp}^{\ast,0} , \tilde{Supp}^{\ast,1}$ be the sets corresponding  to  ${Z}^{\diamond,0}_{ G}, {Z}^{\diamond,1}_{ G}  $
  as in the proof of Lemma \ref{lemma-smoothing-cycles}.
 
 For each for $(\mathbf{w},t) \in \lim \tilde{supp} (\tilde{Z}^{\diamond}_{ G})$ take an isotopy $\tilde{\mathbf{w}}^{\ast}= (\mathbf{w}^s)_{s \in [0,1]} $ of multi-loops such that 
\begin{itemize}
\item  $\mathbf{w}^0 =\mathbf{w} $; 
\item $\mathbf{w}^s$ is  close in the $C^0$-topology to $\mathbf{w}$ for each $s \in [0,1]$;
\item  $\tilde{\mathbf{w}}^{\ast} \in \tilde{Supp}^{\ast,0}$  if $t \ll 0$,  $\tilde{\mathbf{w}}^{\ast} \in \tilde{Supp}^{\ast,1}$  if $t \gg 0$;
\item to each one parameter family (\ref{basic-family-loops}) associated to $\tilde{supp} (\tilde{Z}^{\diamond}_{ G})$ corresponds a two parameter family $(\mathbf{w}^{s,t})_{s \in [0,1],a \leq t \leq b}$ which consists  generically of embedded links  crossing transversally in a codimension-one locus of the space of the parameters $t,s$.    
\end{itemize}

Let $\tilde{\tilde{Supp}}^{\ast}$ be the set of pairs $(\tilde{\mathbf{w}}^{\ast},t)$  obtained in this way.
 
Using $\tilde{\tilde{Supp}}^{\ast}$ we can define  $\tilde{Z}^{\diamond, \ast, (n)}_{ G}$ and $\tilde{\tilde{Z}}^{\diamond, (n)}_{ G}$ for $n$ big enough such that
$$  \lim \tilde{supp} (\tilde{\tilde{Z}}^{\diamond, \ast}_{ G}) = \tilde{\tilde{Supp}}^{\ast}.$$




Now consider the case $H(G) \neq \emptyset$. Use Lemma  \ref{extension-PS-lemma-isotopy} and its analogous for isotopy of isotopies to define $\tilde{Z}^{\diamond, \ast}_{\preccurlyeq G}$ and   $\tilde{\tilde{Z}}^{\diamond }_{\preccurlyeq G}$. The proof of the Lemmas tells us that  the chains  can be taken    
 such that 
 $$\lim \tilde{supp} (\tilde{Z}^{\diamond, \ast}_{\prec G})= \lim \tilde{supp} (\tilde{Z}^{\diamond, \ast}_{\preccurlyeq G}) , \lim  \tilde{\tilde{supp}} (\tilde{\tilde{Z}}^{\diamond}_{\prec G})= \lim \tilde{\tilde{supp}} (\tilde{\tilde{Z}}^{\diamond }_{\preccurlyeq G}).$$
\end{proof}

\end{document}